\documentclass{article}

\usepackage{graphicx}
\usepackage{xcolor}
\usepackage{hyperref}
\usepackage{dsfont}

\usepackage{tikz-cd} 
\usepackage{enumerate} 

\usepackage{pgf,pgfplots}
\pgfplotsset{compat=1.17}
\usepackage{tikz}

\usepackage{caption}
\usepackage{subcaption}

\usepackage[margin=1in, top=0.5in]{geometry}

\usepackage{amsmath,amsfonts,amssymb,amsthm,mathrsfs}

\usepackage{cleveref}

\usepackage{algorithm}
\usepackage{algpseudocode}

\usepackage[normalem]{ulem}

\usepackage{cancel}

\usepackage{soul}

\theoremstyle{plain} 

\newtheorem{theorem}{Theorem}[section]
\newtheorem{lemma}[theorem]{Lemma}
\newtheorem*{lemma*}{Lemma}
\newtheorem{proposition}[theorem]{Proposition}
\newtheorem*{proposition*}{Proposition}
\newtheorem{corollary}[theorem]{Corollary}

\theoremstyle{definition} 

\newtheorem{remark}[theorem]{Remark}
\newtheorem{definition}[theorem]{Definition}
\newtheorem*{definition*}{Definition}

\newtheorem{construction}[theorem]{Construction}
\newtheorem{example}[theorem]{Example}

\newcommand{\R}{\mathbb{R}}
\newcommand{\Z}{\mathbb{Z}}
\newcommand{\N}{\mathbb{N}}

\newcommand{\Fi}{\mathbb{F}}

\newcommand{\A}{\mathcal{A}}
\newcommand{\B}{\mathcal{B}}
\newcommand{\C}{\mathcal{C}}
\newcommand{\M}{\mathcal{M}}
\newcommand{\Nn}{\mathcal{N}}
\newcommand{\I}{\mathcal{I}}
\newcommand{\II}{\mathbb{I}}
\newcommand{\Oo}{\mathcal{O}}
\newcommand{\J}{\mathcal{J}}

\newcommand{\X}{\mathfrak{X}}

\newcommand{\one}{\mathds{1}}

\newcommand{\eps}{\varepsilon}

\newcommand{\norm}[1]{\left\|#1\right\|}
\DeclareMathOperator{\im}{im}
\DeclareMathOperator{\End}{End}
\DeclareMathOperator{\Hom}{Hom}

\DeclareMathOperator{\id}{id}
\DeclareMathOperator{\Sing}{Sing}

\DeclareMathOperator{\Ch}{Ch}
\DeclareMathOperator{\Crit}{Crit}

\DeclareMathOperator{\PCh}{PCh}
\DeclareMathOperator{\EPCh}{EPCh}
\DeclareMathOperator{\TPCh}{TPCh}
\DeclareMathOperator{\TEPCh}{TEPCh}
\DeclareMathOperator{\wbCh}{wbCh}
\DeclareMathOperator{\wbChgen}{wb^+Ch}
\DeclareMathOperator{\MC}{MC}
\DeclareMathOperator{\Vect}{Vect}
\DeclareMathOperator{\PVect}{PVect}

\DeclareMathOperator{\TPVect}{TPVect}
\DeclareMathOperator{\Span}{Span}
\DeclareMathOperator{\Mult}{Mult}
\DeclareMathOperator{\ind}{ind}
\DeclareMathOperator{\Barc}{Bar}
\DeclareMathOperator{\tBarc}{tBar}
\DeclareMathOperator{\Rep}{Rep}

\newcommand{\argdot}{\hspace{1pt}\cdot\hspace{1pt}} 

\newcommand\TODO[3]{\hbox to 0pt{\textcolor{#1}{$^\bullet$}}\marginpar{\footnotesize \textcolor{#1}{\begin{flushleft}#2: #3\end{flushleft}}}}

\usepackage[mathlines]{lineno}

\usepackage{etoolbox} 

\newcommand*\linenomathpatch[1]{%
  \cspreto{#1}{\linenomath}%
  \cspreto{#1*}{\linenomath}%
  \csappto{end#1}{\endlinenomath}%
  \csappto{end#1*}{\endlinenomath}%
}
\newcommand*\linenomathpatchAMS[1]{%
  \cspreto{#1}{\linenomathAMS}%
  \cspreto{#1*}{\linenomathAMS}%
  \csappto{end#1}{\endlinenomath}%
  \csappto{end#1*}{\endlinenomath}%
}

\expandafter\ifx\linenomath\linenomathWithnumbers
  \let\linenomathAMS\linenomathWithnumbers
  \patchcmd\linenomathAMS{\advance\postdisplaypenalty\linenopenalty}{}{}{}
\else
  \let\linenomathAMS\linenomathNonumbers
\fi

\linenomathpatch{equation}
\linenomathpatchAMS{gather}
\linenomathpatchAMS{multline}
\linenomathpatchAMS{align}
\linenomathpatchAMS{alignat}
\linenomathpatchAMS{flalign}

\title{Tagged barcodes for the topological analysis of gradient-like vector fields}
\author{Clemens Bannwart\footnote{Department of Physics, Mathematics and Computer Science, University of Modena and Reggio Emilia, Italy,
\texttt{clemens.bannwart@gmail.com}} \and Claudia Landi\footnote{Department of Sciences and Methods for Engineering,  University of Modena and Reggio Emilia, Italy, 
\texttt{claudia.landi@unimore.it}}}
\date{}

\begin{document}

\maketitle

\begin{abstract}
    Intending to introduce a method for the topological analysis of fields, we present a pipeline that takes as an input a weighted and based chain complex, produces a factored chain complex, and encodes it as a barcode of tagged intervals (briefly, a tagged barcode). We show how to apply this pipeline to the weighted and based Morse chain complex of a gradient-like Morse-Smale vector field on a compact Riemannian manifold in both the smooth and discrete settings.
    Interestingly for computations, it turns out that there is an isometry between factored chain complexes endowed with the interleaving distance and their tagged barcodes endowed with the bottleneck distance. 
    Concerning stability, we show that the map taking a generic enough gradient-like vector field to its barcode of tagged intervals is continuous. 
    Finally, we prove that the tagged barcode of any such vector field can be approximated by the tagged barcode of a combinatorial version of it with arbitrary precision.   \\
    
 

\end{abstract}

\section*{Introduction}

In topological data analysis (TDA), it is usual to represent data as a continuous function $f\colon M \to \mathbb{R}$, also called a scalar field, and to analyze it through its persistence barcode. The latter is a topological invariant of $f$ that serves as a qualitative and quantitative summary. The persistence barcode of $f$ is built in three steps: first, one parametrizes the $f$-sublevel set filtration of $M$ by level values; secondly, one considers the homology with field coefficients of each sublevel set to obtain a  functor that associates with each real value the homology vector space of the corresponding sublevel set, commonly known as a persistence module; finally, one exploits the fact that tame persistence modules can be decomposed as finite direct sums of simpler persistence modules, whose support is given by real intervals, to define the persistence barcode of $f$ as the multiset of such intervals \cite{OudotPersistenceTheoryBook15}. Thanks to properties like stability and invariance, persistence barcodes have received much attention not only as a mathematical construct but also in applications outside of mathematics \cite{database}. Stability refers to the fact that it is possible to endow both the set of functions and the set of persistence modules with distances so that the above-described pipeline is continuous. Invariance refers to the fact that precomposing $f$ with a homeomorphism does not change the final persistence barcode. 

Of course, there is much interest in going beyond the case of a single scalar field. One generalization considers multiple scalar fields, i.e.~maps $f\colon M\to \R^n$ \cite{CarlssonZomorodian_MultiDPersistence}. 
In contrast,  the goal of this paper is to develop a persistence theory for vector fields on Riemannian manifolds, i.e. sections $v\colon M \to TM$ of the tangent bundle. In particular, we aim to define a barcode-like signature for vector fields with the properties of stability, invariance, and non-triviality: small perturbations of the vector field should yield small changes in the barcode; vector fields in the same geometrical-topological equivalence class should be assigned the same barcode; and barcodes should make it possible to distinguish different enough vector fields.

The main obstacle to generalizing persistence to vector fields arises already in the first step of the pipeline because of the lack of a natural filtration defined by a vector field.
One strategy to overcome this issue is to abandon the initial filtration step in favor of a simplification procedure. The theme of simplification has played an important role in the persistence theory of scalar fields from the beginning \cite{Edelsbrunner_TopPersistenceSimp} and has developed over the years. In some cases, especially in discrete settings, it might be possible to simplify the field itself, e.g., via the cancellation of pairs as in \cite{Bauer_OptimalTopologicalSimplification_2012}. In other cases, simplifications are purely algebraic, such as the cancellation of close pairs from \cite{Robbins_SkeletonizationPartitioningDMT_2015}. While for a function the choice of cancelable pairs is usually driven by its values, in the vector field case such information is not available, making pairings non-obvious.
In \cite{Chazal_CompWellVectorFields,Wangetal_VisualizingRobustness2013}, critical points are proposed to be simplified when they belong to a region where the norm of the vector field is small.

A more recent theme has been to work directly on the chain complex level, where one may decompose the parametrized chain complex associated with the sublevel set filtration, see e.g. \cite{related_work_decomposition, bcw,memoli_et_al:LIPIcs.SoCG.2023.51}, rather than decomposing persistence modules.
Other sources have focused on simplifying filtered chain complexes without changing the persistence barcode \cite{MischaikowNandaMTforFiltrations2013,Ebli-Hacker-Maggs_MorseSignalCompression_2024}. 
Simplifications of chain complexes, as in algebraic Morse theory, have also proven useful in other applications \cite{SkoldbergMorseTheoryfromAlgView}. 

Based on these priors, we propose to adopt a simplification strategy at the chain complex level. Thus, we will work with parametrized chain complexes that have epimorphic internal maps coming from taking quotients, rather than the usual monomorphic internal maps coming from filtrations.

In line with this program, after reviewing the necessary background in \Cref{sec:preliminaries}, our first contribution is the study of the category of factored chain complexes in \Cref{sec:tepch}. These are defined as tame epimorphic parametrized chain complexes, i.e.~functors that take real values to chain complexes, with internal maps all epimorphisms and that may fail to be isomorphisms at most at finitely many times. We show that objects in this category  can be decomposed into simple direct summands enumerated by a multiset of tagged real intervals, thus yielding a tagged barcode (see \Cref{sec:structure-thm}). After extending the standard interleaving and bottleneck distances to this setting, we prove the isometry theorem stating that the interleaving distance between two factored chain complexes is equal to the bottleneck distance between their tagged barcodes (see \Cref{sec:iso-thm-tepch}).

As a source of factored chain complexes, 
we give a general procedure that starts from a chain complex with chosen bases and weights on them, and constructs a factored chain complex via an algorithm that pairs basis elements according to their weights (see \Cref{sec:parametrize-ch}). We prove stability for this construction, in the sense that, if we start from two isomorphic based chain complexes, equipped with a generic set of weights, the corresponding factored chain complexes have tagged barcodes whose bottleneck distance is a continuous function of the weights (see \Cref{sec:stability}).

We exemplify how to derive a barcode of tagged intervals for a scalar field $f$, when $f$ is a Morse function, starting from its usual Morse complex generated by critical points. In this case, the weights are given by the differences in function values and the resulting tagged barcode can be related to the classical persistence barcode of $f$ (see \Cref{sec:scalar-fields}). 

To generalize from scalar fields to smooth vector fields, in \Cref{sec:barcode-of-vector-field} of this paper we confine ourselves to gradient-like Morse-Smale vector fields on a compact Riemannian manifold, to which we can assign a barcode of tagged intervals, using our construction. In this case, the weights are given by the distances between singular points of the field. 
Interestingly, a genericity condition on the weights is now needed, in the form of no two pairs of singular points having the same distance, to guarantee the uniqueness of the resulting tagged barcode (cf. \Cref{fig:breaking-ties}).  This marks a key difference with the gradient case, where it is not needed. 
This barcode gives a summary of the topological structure of the vector field, where each finite bar stands for a pair of fixed points that gets simplified. The bar spans two adjacent degrees corresponding to the indices of the paired points, with the length indicating the cost of the simplification. The infinite bars on the other hand are concentrated in only one degree, as they correspond to fixed points that do not get simplified and are related to the homology of the underlying manifold.

We show that the map taking a vector field to its tagged barcode is continuous, thus proving stability (see \Cref{sec:local-stab}). Moreover, two topologically equivalent vector fields with isometric singular sets have the same barcode, thus yielding invariance. 

Thanks to the generality of the approach, we can also apply it to combinatorial gradient-like vector fields~\cite{Forman_CombVectFieldsDynSyst}. As a bridge between the smooth and combinatorial settings, we show that the tagged barcode of a smooth vector field on a compact Riemannian manifold can be approximated arbitrarily well, in terms of bottleneck distance, by the tagged barcodes of combinatorial vector fields defined on sufficiently refined triangulations of the manifold (see \Cref{sec:comb-approx}).

In conclusion, this paper represents a proof of concept of a method to build a tagged barcode from any topological or combinatorial object that gives rise to a weighted and based chain complex, such as gradient-like vector fields. In perspective, it will be interesting to apply this method to more general vector fields also containing closed orbits, or even to higher-order tensor fields. 
However, already for vector fields with closed orbits, difficulties arise when trying to define invariants analogous to the Morse complex, as pointed out in \cite{AboutNonUniqueness2024}.
In the context of combinatorial Morse theory, an approach for the topological analysis of the dynamics of a not necessarily gradient-like vector field has been proposed in \cite{Dey-etal_ConnMatPersRed2024}, which uses connection matrices. 
By the generality of our approach, we could obtain tagged barcodes by applying our construction to such chain complexes as well, with an appropriate choice of weights. 

\tableofcontents

\section{Preliminaries}\label{sec:preliminaries}

In this section, we explain the necessary background material and fix our notation. We provide either proofs or references for all results.

We use the convention that $0$ is a natural number, i.e.~$\N = \{0,1,2,3,\ldots\}$. We fix an arbitrary field $\Fi$. For specific examples we will use the choice $\Fi=\Z/2\Z$.
We denote by $[0,\infty)$ the poset category of the totally ordered set of non-negative real numbers. We will often use the symbol $\one$ to denote the identity.

\subsection{Parametrized objects}\label{sec:par-obj}

Parametrized objects are a generalization of persistence modules and have been described in \cite{BubenikMetricsGeneralizedPersMod14} under the name of generalized persistence modules. This setting is convenient for us because it allows us to introduce some concepts in a higher generality so that we can then apply them in the case of parametrized vector spaces and also parametrized chain complexes. 

Let $\C$ be a category. A \textbf{parametrized object of $\C$} is a functor $X\colon [0,\infty) \to \C$. This means that for all $t \ge 0$ we have an object $X^t$ of $\C$ and for any pair $s\le t \in [0,\infty)$ we have a morphism $X^{s\le t}\colon X^s \to X^t$ in $\C$. These morphisms are called the \textbf{internal morphisms} (or \textbf{internal maps}) of $X$. The functoriality translates to the conditions that $X^{t\le t}$ is the identity on $X^t$ and whenever $0\le r\le s\le t$, then we have $X^{r\le t} = X^{s\le t} \circ X^{r\le s}$.
A \textbf{parametrized morphism} $\varphi\colon X \to Y$ between two parametrized objects $X,Y$ of $\C$ is a natural transformation of functors. This means that for all $t \ge 0$, $\varphi^t\colon X^t \to Y^t$ is a morphism in $\C$ and for $s\le t$ we have $\varphi^t \circ X^{s\le t} = Y^{s\le t} \circ \varphi^s$. 

By $\Vect$ we denote the category of finite dimensional vector spaces and linear maps and by $\Ch$ we denote the category of compact, non-negative chain complexes with coefficients in $\Fi$, where a chain complex $C_\bullet$ is called \textbf{non-negative}, if $C_i=0$ for $i<0$, and \textbf{compact}, if $\dim C_i < \infty$ for all $i$ and $C_i=0$ for all but finitely many $i$. 
In this paper, we assume all chain complexes to be compact, non-negative, and with field coefficients, so we will usually omit these adjectives. We will only consider parametrized objects in the categories $\Vect$, where we talk about parametrized vector spaces and parametrized (linear) maps, and $\Ch$, where we talk about parametrized chain complexes and parametrized (chain) maps.

\begin{definition}
    Let $X\colon [0,\infty) \to \C$ be a parametrized object of a category $\C$. We say that $X$ is 
    \begin{itemize}
        \item \textbf{left-constant at $t \in [0,\infty)$}, if either $t=0$ or there exists a real number $\eps>0$ such that for all $0\le \delta \le \eps$, the internal map $X^{t-\delta\le t}$ is an isomorphism,
        \item \textbf{right-constant at $t \in [0,\infty)$}, if there exists a real number $\eps>0$ such that for all $0\le \delta \le \eps$, the internal map $X^{t\le t+\delta}$ is an isomorphism,
        \item \textbf{tame}, if $X$ is right-constant at every $t \in [0,\infty)$, and left-constant everywhere but on a finite set.
        \item \textbf{epimorphic}, if all the internal maps are epimorphisms,
        \item \textbf{monomorphic}, if all the internal maps are monomorphisms.
    \end{itemize}
\end{definition}

We denote the category of parametrized objects and parametrized morphisms of $\C$ by $\text{P}\C$. We denote by $\text{TP}\C$ and $\text{TEP}\C$ the full subcategories of tame parametrized objects and tame epimorphic parametrized objects, respectively.
Note that a parametrized object $X$ is tame if and only if there exists a sequence 
\[
0 = t_0 < t_1 < \cdots < t_r < t_{r+1}=\infty
\]
such that, for any $i=0,\ldots, r$ and $t_i \le s\le t < t_{i+1}$, the map $X^{s\le t}$ is an isomorphism.
Since we are working mostly with the category of chain complexes, we introduce also some shorter terminology for that case: We say \textbf{filtered chain complex} for a tame monomorphic parametrized chain complex and \textbf{factored chain complex} for a tame epimorphic parametrized chain complex.

\paragraph{The generalized interleaving distance.}

Given a parametrized object $X\colon [0,\infty) \to \C$ and $\eps>0$, we define its \textbf{$\eps$-shift} as $X_\eps$, where $(X_\eps)^t := X^{t+\eps}$ and $(X_\eps)^{s\le t} := X^{s+\eps \le t+\eps}$. Given a parametrized morphism $\phi\colon X \to Y$, we define its \textbf{$\eps$-shift} as $\phi_\eps\colon X_\eps \to Y_\eps$, where $(\phi_\eps)^t := \phi^{t+\eps}$. 
Given two parametrized objects $X,Y \in \text{\normalfont P}\C$, an \textbf{$\eps$-interleaving} between $X$ and $Y$ is a pair $(\phi,\psi)$ of parametrized morphisms $\phi\colon X \to Y_\eps$ and $\psi\colon Y \to X_\eps$, such that $(\psi_\eps \circ \phi)^t = X^{t\le t+2\eps}$ and $(\phi_\eps \circ \psi)^t = Y^{t\le t+2\eps}$ for all $t \ge 0$.
If there exists an $\eps$-interleaving between $X$ and $Y$ we say that $X$ and $Y$ are \textbf{$\eps$-interleaved}. Note that if $\delta>\eps$ and $X$ and $Y$ are $\eps$-interleaved, than it follows that they are also $\delta$-interleaved. We say that $X$ and $Y$ are \textbf{interleaved} if there exists a real number $\eps\ge 0$ such that $X$ and $Y$ are $\eps$-interleaved.

Given two parametrized objects $X,Y \in \text{P}\C$, their \textbf{interleaving distance} is defined as
\[
d_I(X,Y) := \inf \{\eps >0 | \text{ $\exists$  $\eps$-interleaving between $X$ and $Y$}\},
\]
where we use the convention that $\inf \emptyset = \infty$. The interleaving distance is an extended pseudometric on $\text{\normalfont P}\C$, see \cite{BubenikMetricsGeneralizedPersMod14} for a proof. 

\paragraph{The generalized bottleneck distance.} 

Similar to the interleaving distance, also the bottleneck distance has been generalized \cite{IsometryThmGeneralized}. We give the basic definitions in large generality here but will later consider only two special cases: Multisets of intervals for the barcodes of parametrized vector spaces, and multisets of tagged intervals for the barcodes of factored chain complexes.

A \textbf{multiset} is a pair $\A=(A,m)$, where $A$ is a set and $m\colon A \to \N_{\ge1}$ is a function. For any element $I \in A$, the value $m(I)$ represents the multiplicity of $I$ in $\A$. Following \cite{BauerLesnick2014InducedMatchings}, we define the \textbf{representation} of $\A$ as the set
\[
\Rep(\A) := \{ (I,k) \in A\times \N_{\ge1} \mid m(I) \ge k \}. 
\]
Given a set $S$, we denote by $\Mult(S)$ the set of all multisets $(A,m)$ for which $A\subset S$.
We sometimes treat multisets on $S$ as if they were sets and leave it to the reader to make the necessary adjustments. Explicitly, this can usually be done by identifying all multisets with their representations.

Assume that $c\colon S\times S \to [0,\infty]$ is an extended pseudometric on $S$ and $W\colon S \to [0,\infty]$ is a function on $S$, such that $c$ and $W$ are \textbf{compatible}, which means that for all $s_1,s_2 \in S$ we have
\[
|W(s_1)-W(s_2)| \le c(s_1,s_2).
\]
Then we can define the corresponding bottleneck distance between multisets on $S$.
Given $\A,\B \in \Mult(S)$, and identifying $\A,\B$ with their representations, a \textbf{matching} between them is a subset $\M \subseteq \A \times \B$ such that for all $I \in \A$ there is at most one $J \in \B$ such that $(I,J) \in \M$ and for all $J' \in \B$ there is at most one $I' \in \A$ such that $(I',J') \in \M$. If $(I,J) \in \M$, then we say that $I$ and $J$ are \textbf{matched}, all other elements of $\A$ and $\B$ are called \textbf{unmatched}. The \textbf{cost} of the matching $\M$ is defined as
\[
\operatorname{cost}(\M) := \max \left\{ \sup_{(I,J) \in \M} c(I,J), \sup_{I \in \A \cup \B \text{ unmatched}} W(I) \right\}.
\]
Given $\eps>0$, we call $\M$ an \textbf{$\eps$-matching} if $\operatorname{cost}(\M)\le \eps$. The \textbf{bottleneck distance} between $\A$ and $\B$ is defined as
\begin{align*}
    d_B(\A,\B) &:= \inf \{\operatorname{cost}(\M) \mid \M \subseteq \A\times \B \text{ is a matching} \} \\
    &\hspace{3pt}= \inf \{\eps >0 \mid \text{$\exists$ $\eps$-matching between $\A$ and $\B$} \}.
\end{align*}

\begin{remark}\label{rem:bott-distance-ext-pseudometric}
    If we are given a set $S$, an extended pseudometric $c\colon S\times S \to [0,\infty]$, and a function $W\colon S \to [0,\infty]$, such that $c$ and $W$ are compatible, then the bottleneck distance $d_B$ is an extended pseudometric on $\Mult(S)$.
\end{remark}

\paragraph{Parametrized vector spaces.} 

Now we review the structure theorem and the isometry theorem from persistent homology in the language of parametrized vector spaces. This has two purposes. One is to motivate the analogous results for factored chain complexes and the other is that we want to use the results for parametrized vector spaces when proving the corresponding results for factored chain complexes.

In TDA literature, parametrized vector spaces are usually called \textbf{persistence modules}. Often another totally ordered set is used instead of $[0,\infty)$. Popular choices are $\R$, $\Z$, $\N$, or finite totally ordered sets. This choice of poset leads to small technical differences, but the overall flavour of the resulting theory is the same.

In this paper, by an \textbf{interval} we always mean a half-open interval of non-negative real numbers, i.e.~a subset $I\subseteq [0,\infty)$ of the form $I=[s,t)$ for some $0\le s < t\le \infty$. We denote the set of all intervals by $\I$. Now we define the interval functors in $\TPVect$.

\begin{definition}
    Given $0\le s < t \le \infty$, we define the \textbf{interval functor in $\TPVect$}, denoted $\Fi[s,t)$, to be the parametrized vector space whose vector spaces and internal maps are given by
    \[
    \Fi[s,t)^r := \begin{cases}
        \Fi, \text{ if } s\le r < t, \\
        0, \text{ otherwise},
    \end{cases}
    \qquad 
    \Fi[s,t)^{q\le r} := \begin{cases}
        \one_\Fi, \text{ if } s\le q\le r < t, \\
        0, \text{ otherwise}.
    \end{cases}
    \]
\end{definition}

We state the structure theorem in $\TPVect$. A proof can be found in \cite{OudotPersistenceTheoryBook15}.

\begin{theorem}[Structure theorem in $\TPVect$]\label{thm:str-thm-vect}
    Any tame parametrized vector space $V$ is isomorphic to a finite direct sum of interval functors in $\TPVect$, i.e.~there exists a unique finite multiset $\Barc =\Barc(V)\in \Mult(\I)$ of intervals such that 
    \[
    V \cong \bigoplus_{[s,t) \in \Barc} \Fi[s,t).
    \]
    The multiset $\Barc$ is called the \textbf{persistence barcode} of $V$.
\end{theorem}

Let us now recall the isometry theorem in $\TPVect$, which is another classical result in TDA. For this, we first introduce the interleaving distance and the bottleneck distance for tame parametrized vector spaces. The former is defined by applying the more general definition given above to the case of the category $\C=\Vect$. For the latter, we need to define a pseudometric and a weight function for multisets of intervals. 
Given two intervals $I=[s,t)$ and $I'=[s',t')$, we define the \textbf{cost of matching $I$ to $I'$} as $c(I,I') := \max\{|s-s'|,|t-t'|\}$. The \textbf{weight of $I$} is defined as $W(I):=\frac{t-s}{2}$.
One can check that $c$ defines an extended metric on multisets of intervals and that $c$ and $W$ are compatible. We thus get a bottleneck distance $d_B$ on multisets of intervals by \Cref{rem:bott-distance-ext-pseudometric}. The following theorem is a classical result in TDA, a proof can be found in \cite{OudotPersistenceTheoryBook15}.

\begin{theorem}[Isometry Theorem in $\TPVect$] \label{thm:iso-thm-tpvect}
    If $V$ and $W$ are tame parametrized vector spaces, then
    \[
    d_I(V,W) = d_B(\Barc(V),\Barc(W)).
    \]
\end{theorem}

\paragraph{Decomposition of chain complexes.}

If we have not explicitly assigned a symbol to the differential in a chain complex $C$, then we write $\partial_n^C$ for the differential in degree $n$. If the degree and/or the chain complex are clear from the context, then we sometimes write only $\partial^C$ or $\partial_n$ or even just $\partial$ instead.

Given $n\ge 1$, the \textbf{$n$-disk of $\Ch$} is the chain complex $D^n = (D^n_\bullet, \partial_\bullet)$ with
\[
    D^n_k = \begin{cases}
        \Fi, \text{ if } k=n,n-1, \\
        0, \text{ otherwise},
    \end{cases}
    \quad \text{and} \qquad
    \partial_k = \begin{cases}
        \one_\Fi, \text{ if } k=n, \\
        0, \text{ otherwise},
    \end{cases}
    \quad \text{for all } k \in \Z.
\]

Given $n\ge 0$, the \textbf{$n$-sphere of $\Ch$} is defined as the chain complex $S^n = (S^n_\bullet, \partial_\bullet)$ with
    \[
    S^n_k = \begin{cases}
        \Fi, \text{ if } k=n, \\
        0, \text{ otherwise},
    \end{cases}
    \quad \text{and} \qquad 
    \partial_k = 0, 
    \quad \text{ for all }k \in \Z.
    \]

It is known (see e.g. \cite{weibel1994introduction}) that every chain complex $X \in \Ch$ can be written as a finite direct sum of disks and spheres in a unique way up to isomorphism, since we are working with field coefficients. In particular, 
    \begin{itemize}
        \item $\#$ of $n$-spheres $=$ $\dim(H_n(X))$,
        \item $\#$ of $n$-disks $=$ $\dim(\operatorname{im}(\partial_n\colon X_n \to X_{n-1}))$.
    \end{itemize}

Given a chain complex $C_\bullet$ and a basis $\B=\{b_1,\ldots,b_r\}$ for $C_{k-1}$, we denote by $\langle\argdot,\argdot\rangle$ the scalar product on $C_{k-1}$ induced by the basis $\B$. If we are given an element $a \in C_k$ and a basis element $b_j \in \B$, writing $\langle\partial a,b_j\rangle\neq 0$ hence means that $\partial a = \lambda_1 b_1 + \cdots + \lambda_r b_r$ with $\lambda_j \neq 0$, i.e. $b_j$ appears in the boundary of $a$. 
Given a vector space $V$ and $x \in V$, we denote by $\Span(x)$ the linear subspace of $V$ generated by $x$. 

Later, we will need the following result, which already appears in the literature in different forms (see, e.g., \cite{HomologyComputation1998,GallaisCombinatorial2010}). 

\begin{lemma} \label{lem:chain-complex-simplification} 
    If $C_\bullet$ is a chain complex and $a \in C_n$ with $\partial a \neq 0$, then the sequence  $\overline{C}_\bullet$ of linear maps
    \begin{center}
    \begin{tikzcd}
        \cdots \ar[r] &C_{n+1} \ar[r] &C_n/\Span(a) \ar[r] &C_{n-1}/\Span(\partial a) \ar[r] &C_{n-2} \ar[r] &\cdots
    \end{tikzcd}
    \end{center}
    naturally induced by the differentials in $C_\bullet$ is a chain complex in $\Ch$.
    Moreover, there exists an epimorphism $q_a\colon C_\bullet \to \overline{C}_\bullet$,
    \begin{center}
    \begin{tikzcd}
        \cdots \ar[r] &C_{n+1} \ar[r] \ar[d] &C_n \ar[r] \ar[d] &C_{n-1} \ar[r] \ar[d] &C_{n-2} \ar[r] \ar[d] &\cdots \\
        \cdots \ar[r] &C_{n+1} \ar[r] &C_n/\Span(a) \ar[r] &C_{n-1}/\Span(\partial a) \ar[r] &C_{n-2} \ar[r] &\cdots,
    \end{tikzcd}
    \end{center}
    where the vertical maps are either the identity maps or quotient maps. Moreover, assume that we are given a basis $\B_k$ for each $C_k$, such that $a \in \B_n$. Further, we are given $b \in \B_{n-1}$ such that $\lambda:= \langle \partial a,b \rangle \neq 0$. 
    Then, defining
    \[
    \overline{\B}_n = \{ [a'] \mid a' \in \B_n \setminus \{a\} \}, \qquad
    \overline{\B}_{n-1} = \{ [b'] \mid b' \in \B_{n-1} \setminus \{b\} \}, \qquad \overline{\B}_k = \B_k \text{ for $k \neq n, n-1$},
    \]
    yields bases $\overline{\B}_k$ for the vector spaces $\overline{C}_k$. If, for all $k$, we denote by $M_k$ the matrix that represents $\partial^C_k\colon C_k \to C_{k-1}$ with respect to the bases $\B_k$ and $\B_{k-1}$ and by $\overline{M}_k$ the matrix that represents $\partial^{\overline{C}}_k\colon \overline{C}_k \to \overline{C}_{k-1}$ with respect to the bases $\overline{\B}_k$ and $\overline{\B}_{k-1}$, then:
    \begin{enumerate}[(i)]
        \item $\overline{M}_k=M_k$ for $k\ge n+2$ and for $k\le n-2$.\label{item:chain-complex-simplification-1}
        \item $\overline{M}_{n+1}$ is obtained from $M_{n+1}$ by deleting the row corresponding to $a\in \B_n$.\label{item:chain-complex-simplification-2}
        \item To obtain $\overline{M}_n$ from $M_n$, multiply the row of $b$ by ${\langle\partial a,b \rangle}^{-1}$, and then subtract from  the row of every $b'\in \B_{n-1}$,   the row of $b$ multiplied by $\langle \partial a,b'\rangle$. Finally, delete the column corresponding to $a\in \B_n$ and the row corresponding to $b\in \B_{n-1}$.\label{item:chain-complex-simplification-3}
        \item $\overline{M}_{n-1}$ is obtained from $M_{n-1}$ by deleting the column corresponding to $b \in \B_{n-1}$.\label{item:chain-complex-simplification-4}
    \end{enumerate}
\end{lemma}

Using the epimorphic chain map $q_a\colon C_\bullet \to \overline{C}_\bullet$ from \Cref{lem:chain-complex-simplification}, we get the following result, the content of which is not new as well (see, e.g. \cite[Section 1.4]{weibel1994introduction}).

\begin{lemma}\label{lem:chain-complex-ses-splits}
    Let $C_\bullet$ be a chain complex and let $\B_k$ be a basis of $C_k$, for all $k$. Let $a \in \B_n$ with $\partial a \neq 0$. Then, there is a short exact sequence of chain complexes
    \begin{center}
        \begin{tikzcd}
            0 \ar[r] &D^n \ar[r,"\iota"] &C_\bullet \ar[r,"q"] &\overline{C}_\bullet \ar[r] &0,
        \end{tikzcd}
    \end{center}
    where $q=q_a$ from \Cref{lem:chain-complex-simplification} and $\iota$ is defined by mapping the generator $1_n \in D^n_n$ to $a \in C_n$ and the generator $1_{n-1} \in D^n_{n-1}$ to $\partial a \in C_{n-1}$. Moreover, the sequence splits.
\end{lemma}

\begin{proof}
    It is straightforward to check that $\iota$ is a monomorphic chain map whose image is equal to $\ker(q)$, thus yielding a short exact sequence. To see that the sequence splits, we define a left inverse $\psi\colon C_\bullet \to D^n$ as follows. Define $\psi_k=0$ for all $k\neq n,n-1$. Choose $b \in \B_{n-1}$ such that $\lambda:= \langle \partial a,b \rangle \neq 0$ and define $\psi_{n-1}(b)= \lambda^{-1}\cdot 1_{n-1}$ and $\psi_{n-1}(b')=0$ for all other basis elements $b'\in \B_{n-1}$, and extend linearly to any $C_{n-1}$. For any $a'\in \B_n$, let $\lambda':= \langle \partial a',b\rangle$ and define $\psi_n(a') = \lambda' \cdot \lambda^{-1}\cdot 1_n$, extend linearly to all of $C_n$.
    Direct calculations show that $\psi$ is indeed a chain map and that $\psi\circ \iota$ is the identity on $D^n$. 
\end{proof}

\subsection{Dynamical systems}

Given a closed smooth manifold $M$, we denote by $\X(M)$ the set of smooth vector fields on $M$. By $\X^1(M)$ we denote the topological space with underlying set $\X(M)$ endowed with the Whitney $C^1$-topology. See \Cref{sec:gen-pos} for the definition of this topology. 

In this subsection we fix a vector field $v\in \X(M)$ and consider the dynamics of the flow generated by $v$. A standard reference for the definitions and results presented here is \cite{palis2012geometric}. We define everything with a subscript $v$, but this will be dropped when it is clear from the context which vector field we are considering.
We write $\phi_v\colon \R \times M \to M$ for the corresponding flow. 
Given any point $p\in M$, the \textbf{orbit} of $p$ (w.r.t.~$v$) is the set
\[
\Oo_v(p) := \phi_v(\R,p) = \{ \phi_v(t,p) \mid t \in \R \}.
\]
A \textbf{fixed point} (or \textbf{singular point}) of $v$ is a point $p\in M$ such that $\phi_v(t,p)=p$ for all $t\in \R$, i.e. $\Oo_v(p)=\{p\}$. This is equivalent to $v(p)=0$. We denote by $\Sing(v)$ the set of all singular points of $v$.

A point $p\in M$ is called \textbf{chain-recurrent} (w.r.t.~$v$) if for all $T>0$ and $\eps>0$ there exist $x_0,x_1,\ldots,x_m \in M$, with $x_0=x_m=p$, such that for all $i=0,\ldots,m-1$ there exists $t_i>T$ with $d_M(\phi(t_i,x_i),x_{i+1}) \le \eps$. Clearly all fixed points are chain-recurrent.

Given a fixed point $p$ of $v$,  its \textbf{stable manifold} and its \textbf{unstable manifold} are defined by
\begin{align*}
    W^s_v(p) &:= \{q \in M \mid \phi_v(t,q) \to p \text{ as } t \to \infty \}, \\
    W^u_v(p) &:= \{q \in M \mid \phi_v(t,q) \to p \text{ as } t \to -\infty \}.
\end{align*}

A fixed point $p \in M$ is called \textbf{hyperbolic} (w.r.t.~$v$) if there exist two subspaces $E^s_p,E^u_p \subseteq T_pM$ such that $T_pM = E^s_p \oplus E^u_p$ and there exist $0< \lambda < 1$ and $C>0$ such that for all $t\ge 0$ we have
\[
|D\phi^t_v(p) x| \le C \lambda^t |x|
\qquad \text{and} \qquad
|D\phi^{-t}_v(p)y| \le C \lambda^t |y|,
\]
where $x \in E^s_p$ and $y \in E^u_p$. The two subspaces $E^s_p$ and $E^u_p$ are uniquely determined, namely they are the tangent spaces of the stable and unstable manifold of $p$, respectively. The dimension of $E^u_p$ is called the \textbf{index} of $p$ and denoted by $\ind_v(p)$.
For a proof of the following theorem, see \cite{palis2012geometric} and the references therein.

\begin{theorem}
    The stable and unstable manifolds of hyperbolic fixed points are injectively immersed submanifolds of $M$.
\end{theorem}

\subsection{Structural stability and Morse-Smale vector fields}

Two injectively immersed submanifolds $K,N \subseteq M$ are said to \textbf{intersect transversally}, if for every point $p\in K \cap N$ we have $T_pM = T_pK + T_p N$. 
Note that if $\dim(K)+\dim(N) < \dim(M)$, then $K$ and $N$ intersect transversally if and only if $K\cap N = \emptyset$. 

\begin{definition}
    A vector field $v \in \X(M)$ is called a \textbf{gradient-like Morse-Smale vector field} if it satisfies the following conditions:
    \begin{enumerate}[$(i)$]
        \item The set of chain-recurrent points consists of finitely many fixed points, all of which are hyperbolic.
        \item The stable and unstable manifolds of fixed points intersect transversally.
    \end{enumerate}
    We denote by $\X_{gMS}(M)$ the set of gradient-like Morse-Smale vector fields on $M$.
\end{definition}

\begin{remark}
    Note that a gradient-like Morse-Smale vector field is not necessarily the gradient of a function. Some authors say that a vector field $v$ is gradient-like with respect to a function $f$ if $\Sing(v)=\Crit(f)$ and $Df(p)[v(p)]<0$ for all $p\in M\setminus \Sing(v)$. Note that such a function $f$ is not unique. With respect to any Riemannian metric on $M$ this is equivalent to $\langle v(p),-\nabla f(p) \rangle >0$, i.e. the vector fields $v$ and $-\nabla f$ point in the same half-space. Such a function indeed exists for every gradient-like Morse-Smale vector field (see \cite{Smale_GradientDS_61}), but $v$ is not guaranteed to coincide with the gradient of any such $f$.  
\end{remark}

Two vector fields $v,w \in \X(M)$ are \textbf{topologically equivalent} if there exists a homeomorphism $h\colon M\to M$ that maps the orbits of $v$ to the orbits of $w$, preserving their orientations. The homeomorphism $h$ is called a \textbf{topological equivalence} between $v$ and $w$.
A vector field $v \in \X(M)$ is \textbf{structurally stable}, if for every $\eps>0$ there exists a neighbourhood $\Nn$ of $v$ in $\X^1(M)$ such that for all $w \in \Nn$ there exists a topological equivalence $\varphi\colon M \to M$ between $v$ and $w$ such that $d_M(p,\varphi(p)) \le \eps$ for all $p \in M$.

The following result follows from \cite{RobinsonStrStab74}.

\begin{theorem}\label{thm:str-stab}
    Gradient-like Morse-Smale vector fields are structurally stable.
\end{theorem}

\subsection{The smooth Morse complex}

The main theorem of Morse theory states that if we have a Morse function $f\colon M \to \R$, then $M$ is homotopy equivalent to a CW complex with one $k$-cell for each critical point of $f$ of index $k$ \cite{BanyagaLecturesOnMorseHomology}. The proof uses the flow of the negative gradient vector field $-\nabla f$. The idea of Morse homology is to make the connection between the flowlines of $-\nabla f$ and the topology of $M$ more explicit by constructing a chain complex $C_\bullet$ whose vector space in degree $k$ is generated by the critical points of $f$ (i.e.~zeroes of $-\nabla f$ of index $k$) and whose differential can be computed by counting certain flowlines of $-\nabla f$, such that the homology of that chain complex is the homology of $M$ \cite{BanyagaLecturesOnMorseHomology}. Note that what we call the Morse complex here is sometimes called Morse-Smale-Witten complex (e.g. in \cite{BanyagaLecturesOnMorseHomology}) or Thom-Smale complex (e.g. in \cite{GallaisCombinatorial2010}). 
The construction can be extended to gradient-like vector fields, as we show in the following. We give this definition only in the case of $\Fi=\Z/2\Z$, in order to reduce the amount of technicalities. Everything we do is also possible with coefficients in an arbitrary field.
If $v$ is a gradient-like Morse-Smale vector field on $M$, we denote by $\Sing_k(v)$ the set of singular points of $v$ of index $k$.

\begin{definition}
    If $v \in \X(M)$ is a gradient-like Morse-Smale vector field, then the associated \textbf{smooth Morse complex} $\MC_\bullet(v) \in \Ch$ (Morse complex, for brevity) is defined as follows:
    \begin{enumerate}
        \item For $k\ge0$, $\MC_k(v)$ is the free $\Fi$-vector space generated by the singular points of index $k$ of $v$.
        \item The boundary operator $\partial\colon \MC_k(v) \to \MC_{k-1}(v)$ is defined by counting the flowlines between singular points of adjecent index (mod 2). Explicitly, for all $p \in \Sing_k(v)$ define
        \[
        \partial(p) := \sum_{q \in \Sing_{k-1}(v)} \alpha(p,q) \cdot q,
        \]
        where $\alpha(p,q)$ is the number of flowlines from $p$ to $q$ (mod 2), i.e.~the number of connected components of $W^u(p)\cap W^s(q)$ (mod 2).
    \end{enumerate}
\end{definition}

Since a Morse-Smale vector field $v$ has only finitely many singular points, $\MC_k(v)$ is a finite-dimensional vector space for all $k$. Also $\MC_k(v)=0$ for $k<0$ and $k>\dim(M)$. In order to see that the boundary operator $\partial$ is well-defined, observe that, by compactness, for any two critical points $p$ and $q$ with $\ind_v(p)=\ind_v(q)+1$, there can exist only finitely many flowlines from $p$ to $q$.

\begin{lemma}\label{lem:top-equiv-induces-iso}
    If $v,w \in \X_{gMS}(M)$, then a topological equivalence $h\colon M \to M$ between $v$ and $w$ induces an isomorphism between $\MC_\bullet(v)$ and $\MC_\bullet(w)$.
\end{lemma}

\begin{proof}
    A topological equivalence $h$ between $v$ and $w$ establishes a bijection between the singular points of $v$ and the singular points of $w$, respecting the indices. This induces bijections $\Sing_k(v) \to \Sing_k(w)$ and thus isomorphisms $\MC_k(v) \to \MC_k(w)$ for all $k$. Moreover, $h$ maps the stable and unstable manifolds of any singular point $p$ of $v$ to the stable and unstable manifolds of the corresponding singular point $h(p)$ of $w$. Therefore $\alpha(p,q)=\alpha(h(p),h(q))$ for any pair $(p,q) \in \Sing_k(v)\times \Sing_{k-1}(v)$ and thus the isomorphisms $\MC_k(v) \to \MC_k(w)$ commute with the boundary operators in $\MC_\bullet(v)$ and $\MC_\bullet(w)$, so we get an isomorphism of chain complexes. 
\end{proof}

The usual result that the differential squares to zero in the Morse complex of a Morse-Smale function can be transferred to the Morse complex for gradient Morse-Smale vector fields.

\begin{theorem}
    Let $v\in \X(M)$ be a gradient-like Morse-Smale vector field. Then $(\MC_\bullet(v),\partial)$ is a chain complex, i.e.~$\partial^2=0$. Moreover, the homology of this chain complex is isomorphic to the singular homology of $M$.
\end{theorem}

\begin{proof}
    If $v$ is the gradient of a Morse-Smale function $f\colon M \to \R$ with respect to some Riemannian metric on $M$, then we refer to \cite{BanyagaLecturesOnMorseHomology} for a proof. In order to reduce to that case when $v$ is gradient-like, note that by combining Lemma 2 from \cite{Newhouse_SimpleArc76} and Theorem B from \cite{Smale_GradientDS_61}, there exists a function $f\colon M\to \R$ which is Morse-Smale with respect to some Riemannian metric and $v$ is topologically equivalent to $-\nabla f$. Since the result holds for $-\nabla f$, it also holds for $v$ by \Cref{lem:top-equiv-induces-iso}.
\end{proof}

Later we will parametrize this chain complex. In other words, given a gradient-like Morse-Smale vector field $v \in \X(M)$, we will construct a factored chain complex $X(v)$, such that $X(v)^0 = \MC_\bullet(v)$. We will do this by applying algebraic simplifications to the Morse complex of $v$.

\subsection{Combinatorial vector fields and the combinatorial Morse complex}

Let us start by recalling the definition of a combinatorial vector field, first introduced in \cite{Forman_CombVectFieldsDynSyst}.
Given a simplicial complex $K$, a \textbf{combinatorial vector field} on $K$ is a function $V\colon K \to K \cup \{0\}$ such that for all $\sigma \in K$: either $V(\sigma)=0$ or $\dim V(\sigma)=\dim \sigma +1$ and $\sigma$ is a face of $V(\sigma)$, $V(V(\sigma))=0$, and $|V^{-1}(\sigma)| \le 1$ for all $\sigma \in K$. %
Cells $\sigma \in K$ with $V(\sigma)=0$ and $V^{-1}(\sigma)=\emptyset$ are called \textbf{critical cells} for $V$. We denote by $\overline{\X}(K)$ the set of all combinatorial vector fields on $K$.
If $V$ is a combinatorial vector field on $K$, then, following \cite{GallaisCombinatorial2010}, a \textbf{$V$-path of dimension $k$} is a sequence $\sigma_0, \sigma_1, \ldots, \sigma_r$ of $k$-simplices, such that $\sigma_i \neq \sigma_{i+1}$ and $\sigma_{i+1}$ is a hyperface of $V(\sigma_i)$ for all $i=0,\ldots,r-1$. If $r=0$, then this $V$-path is called \textbf{stationary} and if $\sigma_r=\sigma_0$, then it is called \textbf{closed}. 
See \cite{GallaisCombinatorial2010} for more details. 
A combinatorial vector field $V$ with no closed non-stationary $V$-paths is called \textbf{gradient-like}. We denote by $\overline{\X}_g(K)$ the set of all gradient-like combinatorial vector fields on $K$.

\begin{definition}\label{def:comb-Morse-complex}
    If $V \in \overline{\X}_g(K)$ is a gradient-like combinatorial vector field on a simplicial complex $K$, then the associated \textbf{combinatorial Morse complex} $\overline{\MC}_\bullet(V) \in \Ch$ is defined as follows:
    \begin{enumerate}
        \item For $k\ge0$, $\overline{\MC}_k(V)$ is the $\Fi$-vector space generated by the critical $k$-cells of $V$.
        \item The boundary operator $\partial\colon \overline{\MC}_k(V) \to \overline{\MC}_{k-1}(V)$ is defined by counting (mod 2) the number of $V$-paths between each pair of critical cells of indices $k$ and $k-1$.
    \end{enumerate}
\end{definition}

We refer to \cite{GallaisCombinatorial2010} for a proof that this is indeed a chain complex (called under the different name of Thom-Smale complex) and that the homology of this chain complex is isomorphic to the simplicial homology of $K$.

\subsection{Relating the smooth and combinatorial Morse complexes} 

If $M$ is a smooth manifold, then a \textbf{triangulation} of $M$ is a pair $(M',\phi)$, where $M'$ is a simplicial complex and $\phi\colon |M'| \to M$ is a homeomorphism from the geometric realization of $M'$ to $M$. A \textbf{triangulated manifold} is a triple $(M,M',\phi)$, where $M$ is a smooth manifold and $(M',\phi)$ is a triangulation of $M$. We say \textbf{triangulated Riemannian manifold} if $M$ is in addition a Riemannian manifold.

The following result is proven by Gallais in \cite{GallaisCombinatorial2010}, with a small gap filled by Benedetti.

\begin{theorem}\label{thm:Gallais}
    Let $M$ be a smooth closed oriented Riemannian manifold and let $v \in \X_{gMS}(M)$. Then there exists a triangulation $(M',\phi)$ of $M$ and a gradient-like combinatorial vector field $V \in \overline{\X}_g(M')$ such that
    \begin{enumerate}[$(i)$]
        \item  For every $k$ there exists a bijection between the singular points of $v$ of index $k$ and the critical cells of $V$ of dimension $k$. If $p \in \Sing(v)$ corresponds to $\sigma \in \Crit(V)$, then $p$ lies in the geometric realization of the corresponding critical cell, i.e.~$p \in \phi(|\sigma|)$. 
        \label{item:gallais-1}
        \item For a pair $(p,q)$ of singular points of $v$ of index $k+1$ and $k$ and the corresponding critical cells $(\tau,\sigma)$ of $V$ we have a bijection between flow lines from $p$ to $q$ and $V$-paths from a hyperface of $\tau$ to $\sigma$. \label{item:gallais-2}
        \item The bijection from $(\ref{item:gallais-1})$ induces an isomorphism between $\MC_\bullet(v)$ and $\overline{\MC}_\bullet(V)$.
        \label{item:gallais-3}
    \end{enumerate}
\end{theorem}

\paragraph{Barycentric subdivision.}

Given a simplicial complex $K$, we denote by $\Delta(K)$ its barycentric subdivision. The $k$-simplices of $\Delta(K)$ correspond to chains of length $k+1$ $\sigma_0 \subseteq \cdots \subseteq \sigma_k$ of simplices in $K$. We write $\Delta^n(K)$ for the $n$-th iteration.

Assume that moreover we are given $V \in \overline{\X}(K)$ and for every critical $k$-cell $\sigma$ of $V$ we have chosen a $k$-cell of $\Delta(K)$ that lies in $\sigma$. Then it is shown by Zhukova in \cite{ZhukovaBarycentric2018} that $V$ and these choices of cells induce a combinatorial vector field on $\Delta(K)$ in a canonical way. We denote this combinatorial vector field by $\Delta(V) \in \overline{\X}(\Delta(K))$, hiding the dependence on the choices of cells from the notation.
Equivalently, the choice upon which $\Delta(V)$ depends amounts to choosing an ordering of the vertices of each critical simplex of $V$.
The following result from \cite{ZhukovaBarycentric2018} lists some of the properties that $\Delta(V)$ has.

\begin{theorem}\label{thm:Zhukova}
    Let $K$ be a simplicial complex and $V \in \overline{\X}_g(K)$. Assume we have chosen an ordering of the vertices of every critical cell of $V$ (and thus $\Delta(V)$ is defined). Then
    \begin{enumerate}[$(i)$]
        \item $\Delta(V) \in \overline{\X}_g(\Delta(K))$, i.e.~there are no closed non-stationary $\Delta(V)$-paths.
        \item The critical simplices of $\Delta(V)$ are exactly the chosen ones. This yields a bijection between the critical $k$-cells of $V$ and the critical $k$-cells of $\Delta(V)$ for all $k$. \label{item:bijection-crit-cells}
        \item There is a one-to-one correspondence between gradient paths (see \cite{ZhukovaBarycentric2018} for the definition) of $V$ and of $\Delta(V)$, respecting the bijections from $(\ref{item:bijection-crit-cells})$. \label{item:corr-gradient-paths}
    \end{enumerate}
\end{theorem}

It follows from $(\ref{item:corr-gradient-paths})$ that the bijections from $(\ref{item:bijection-crit-cells})$ yield an isomorphism $\overline{\MC}_\bullet(V) \cong \overline{\MC}_\bullet(\Delta(V))$.

\section{Factored chain complexes}\label{sec:tepch}

In this section we develop the general theory of tame epimorphic parametrized chain complexes, i.e. factored chain complexes. Many of the results are analogous to the case of tame parametrized vector spaces, i.e.~persistence modules.

We start by a classification of the chain maps between disks and spheres in $\Ch$, which are the simplest chain complexes. This will help us describe the simplest possible parametrized chain complexes and parametrized chain maps between them. Note that in the category $\Ch$, an epimorphism is a chain map that is surjective in every degree. The result follows from a basic calculation and we omit the proof.

\begin{proposition} \label{prop:non-zero-chain-maps-disks-spheres}
    Any non-zero chain map between disks and/or spheres in $\Ch$ is a scalar multiple of one of the following maps:
    \begin{enumerate}[$(i)$]
        \item The identity $\one_{D^n} \colon D^n \to D^n$. This map is an epimorphism.
        \item The identity $\one_{S^n} \colon S^n \to S^n$. This map is an epimorphism.
        \item The inclusion $\iota^n\colon S^n \hookrightarrow D^{n+1}$. This map is not an epimorphism.
        \item The chain map $\Phi^n\colon D^n \to D^{n+1}$, which is the identity in degree $n$ and zero in all other degrees. This map is not an epimorphism.
        \item The chain map $\Psi^n\colon D^n \to S^n$, which is the identity in degree $n$ and zero in all other degrees. This map is an epimorphism.
    \end{enumerate}
\end{proposition}

This is what the chain map $\Psi^n\colon D^n \to S^n$ looks like:

\begin{center}
    \begin{tikzcd}
        D^n \ar[d,"\Psi^n"]
        &\cdots \ar[r] 
        &0 \ar[r] \ar[d]
        &\Fi \ar[r] \ar[d]
        &\Fi \ar[r] \ar[d]
        &0 \ar[r] \ar[d] 
        &\cdots \\
        S^n
        &\cdots \ar[r] 
        &0 \ar[r]
        &\Fi \ar[r]
        &0 \ar[r]
        &0 \ar[r] 
        &\cdots 
    \end{tikzcd}
\end{center}
Linear maps $\Fi \to \Fi$ in any diagram are always assumed to be the identity map, if not stated otherwise.

\subsection[Interval functors]{Interval functors in $\TEPCh$}

In this subsection, we describe the simplest possible factored chain complexes, which turn out to be the building block for all other factored chain complexes. Analogously to how the simplest parametrized vector spaces are characterized by intervals, these factored chain complexes are characterized by tagged intervals, which are intervals with an additional choice of one point. This point marks the time of collapsing from a disk to a sphere.

\begin{definition}\label{def:interval}
    A \textbf{tagged interval} is a tuple consisting of a real interval $[0,t)$, where $0< t\le \infty$, together with a distinguished point $s\in [0,t]$. We denote this tagged interval also by $[0,s,t)$. We write $\J$ for the set of all tagged intervals.
\end{definition}

\begin{definition}
    Let $n\in \N$ and $[0,s,t) \in \J$. In the case $n=0$ we additionally assume that $s=0$. Then we define the \textbf{interval functor $\I^n[0,s,t)$ in $\TEPCh$} by 
    \[
    (\I^n[0,s,t))^r = \begin{cases}
        D^n, \text{ if } 0\le r <s,\\
        S^n, \text{ if } s \le r <t,\\
        0, \text{ otherwise},
    \end{cases}
    \quad \text{and} \quad 
    (\I^n[0,s,t))^{q\le r} = \begin{cases}
        \one_{D^n}, \text{ if } 0\le q \le r <s,\\
        \Psi^n, \text{ if } 0\le q < s \le r <t,\\
        \one_{S^n}, \text{ if } s\le q \le r <t,\\
        0, \text{ otherwise}.
    \end{cases}
    \] 
\end{definition}

Note that the interval functor $\I^n[0,s,t)$ is non-zero on $[0,t)$ in degree $n$ and on $[0,s)$ in degree $n-1$, hence one may think of the tagged interval $[0,s,t)$ as representing the pair $[0,s) \subseteq [0,t)$.

The following diagram shows a diagrammatic depiction of the interval functor $\I^n[0,s,t)$ in $\TEPCh$, where chain complexes are drawn as columns.

\begin{center}
\begin{tikzcd}[row sep=small]
    &0 &&&s&&&t\\
    &\vdots \ar[d] &&\vdots \ar[d] &\vdots \ar[d] &&\vdots \ar[d] &\vdots \ar[d] &\\
    &0 \ar[d] \ar[r] &\cdots \ar[r] &0 \ar[d] \ar[r] &0 \ar[d] \ar[r] &\cdots \ar[r] &0 \ar[d] \ar[r] &0 \ar[d] \ar[r] &\cdots \\
    n &\Fi \ar[d] \ar[r] &\cdots \ar[r] &\Fi \ar[d] \ar[r] &\Fi \ar[d] \ar[r] &\cdots \ar[r] &\Fi \ar[d] \ar[r] &0 \ar[d] \ar[r] &\cdots \\
    n-1 &\Fi \ar[d] \ar[r] &\cdots \ar[r] &\Fi \ar[d] \ar[r] &0 \ar[d] \ar[r] &\cdots \ar[r] &0 \ar[d] \ar[r] &0 \ar[d] \ar[r] &\cdots \\
    &0 \ar[d] \ar[r] &\cdots \ar[r] &0 \ar[d] \ar[r] &0 \ar[d] \ar[r] &\cdots \ar[r] &0 \ar[d] \ar[r] &0 \ar[d] \ar[r] &\cdots \\
    &\vdots &&\vdots &\vdots &&\vdots &\vdots & \\
\end{tikzcd}
\end{center}
For a pictorial representation of $\I^2[0,0,\infty)$, $\I^2[0,1,1)$ and $\I^0[0,0,\infty)$, see \Cref{fig:vector-field-barcode} (right).

We now prove some results describing parametrized chain maps between different interval functors in $\TEPCh$. This will later be useful for the structure theorem. We introduce a notation that will be convenient in many proofs and make statements more readable: For $X \in \TPCh$, $a \in X^s_n$ and $t\ge s$, we set $a^t := X^{s\le t}_n(a) \in X^t_n$. Moreover, we use the convention that $X^\infty_\bullet$ denotes the zero chain complex, so that $X^{s\le \infty}_n$ is the zero map for all $0\le s \le \infty$ and $n \in \N$.

\begin{lemma}\label{lem:non-zero-parmaps}
    Let $[0,s,t) \in \J$, let $1_n \in \I^n[0,s,t)^0_n$ be a generator of $\Fi$ and let $X \in \PCh$ be any parametrized chain complex. Then,
    \[
    \Hom(\I^n[0,s,t),X) \cong \ker(X^{0\le t}_n) \cap \ker(X^{0\le s}_{n-1}\circ \partial^0_n) \subseteq X^0_n,
    \]
    where the isomorphism is given by $\varphi \mapsto \varphi^0_n(1_n)$.
\end{lemma}

\begin{proof}
    Assigning $\varphi \mapsto \varphi^0_n(1_n)$ yields a linear map $\Hom(\I^n[0,s,t),X) \to X^0_n$, so it suffices to check that this map is injective and has image $\ker(X^{0\le t}_n) \cap \ker(X^{0\le s}_{n-1}\circ \partial^0_n)$.

    For injectivity, note that every non-zero element $x \in \I^n[0,s,t)^r_k$, with $r\ge0$, is a scalar multiple of either $1^r_n$ or $\partial 1^r_n$ under the internal maps. Therefore, if $\varphi^0_n(1_n)=0$, then also $\varphi=0$. 
    
    To show that the image is contained in $\ker(X^{0\le t}_n) \cap \ker(X^{0\le s}_{n-1}\circ \partial^0_n)$, note that, in $\I^n[0,s,t)$, $\partial 1^s_n=0$ and $1^t_n=0$, thus if $\varphi\colon \I^n[0,s,t) \to X$ is any parametrized chain map with $\varphi^0_n(1_n)=a$, then 
    \[
    X^{0\le s}_{n-1}(\partial^0_n a) = X^{0\le s}_{n-1}(\partial^0_n \varphi^0_n(1_n))= X^{0\le s}_{n-1}(\varphi^0_{n-1}(\partial 1_n)) = \varphi^s_{n-1} \left( \partial 1^s_n \right) = \varphi^s_{n-1}(0) = 0
    \]
    and 
    \[
    X^{0\le t}_n(a) = X^{0\le t}_n(\varphi^0_n(1_n)) = \varphi^t_n \left( 1^t_n \right) = \varphi^t_n(0) = 0,
    \]
    hence $a \in \ker(X^{0\le t}_n) \cap \ker(X^{0\le s}_{n-1}\circ \partial^0_n)$.
    
    On the other hand, if $a \in \ker(X^{0\le t}_n) \cap \ker(X^{0\le s}_{n-1}\circ \partial^0_n)$, then we can define a parametrized chain map $\varphi\colon \I^n[0,s,t) \to X$ explicitly by defining, for every $r\ge 0$,
    \[
    \varphi^r_n(1^r_n) := a^r
    \quad \text{and} \quad
    \varphi^r_{n-1}(\partial 1^r_n) := \partial^r_n a^r,
    \]
    and then extending linearly to all of $\I^n[0,s,t)$. This is enough since the elements $1^r_n$ and $\partial 1^r_n$ are generators. This is well-defined because by assumption $\partial^s_n(a^s)=0$ and $a^t=0$, and thus yields $\varphi \in \Hom(\I^n[0,s,t),X)$ with $\varphi^0_n(1_n) =a$.
\end{proof}

\begin{lemma}\label{lem:non-zero-morphism}
    Let $n,m \in \N$ and $[0,s,t), [0,s',t') \in \J$. Then
    \[
    \Hom(\I^n[0,s,t), \I^m[0,s',t')) \cong \begin{cases}
        \Fi, &\text{if } m=n \text{ and } s\ge s',\ t\ge t', \\
        \Fi, &\text{if } m=n+1 \text{ and } t\ge s' >0, \\
        0, &\text{otherwise.}
    \end{cases}
    \]
\end{lemma}

\begin{proof}
    For $X=\I^m[0,s',t')$ we have 
    \[
    \ker(X^{0\le t}_n) \cong \begin{cases}
        \Fi, &\text{if } m=n \text{ and } t\ge t', \\
        \Fi, &\text{if } m=n+1 \text{ and } t\ge s' >0, \\
        0, &\text{otherwise,}
    \end{cases}
    \quad \text{and} \quad
    \ker(X^{0\le s}_{n-1}\circ \partial^0_n) \cong \begin{cases}
        \Fi, &\text{if } m=n \text{ and } s\ge s', \\
        \Fi, &\text{if } m=n+1 \text{ and } s' >0, \\
        0, &\text{otherwise.}
    \end{cases}
    \]
    By using \Cref{lem:non-zero-parmaps} and intersecting the two kernels, the result follows.
\end{proof}

\begin{proposition}
    Interval functors $\I^n[0,s,t)$ are indecomposable in $\TEPCh$.
\end{proposition}

\begin{proof}
    Consider a decomposable factored chain complex $X \in \TEPCh$, i.e.~$X=Y\oplus Z$ for $Y,Z \in \TEPCh$ nonzero. If we denote by $\phi$ and $\psi$ the projections on the first and second summand, then these are two nonzero elements of $\End(X)$ none of which is a scalar multiple of the other, so $\End(X) \ncong \Fi$. However, it follows from \Cref{lem:non-zero-morphism} that $\End(\I^n[0,s,t)) \cong \Fi$, so $\I^n[0,s,t)$ must be indecomposable.
\end{proof}

\subsection{Parametrized vector spaces induced from parametrized chain complexes}\label{sec:par-vec-ind-from-par-chain}

Given a parametrized chain complex $X$, we present two important ways how we can construct parametrized vector spaces from $X$. One is by taking homology and the other considers the image of the differential in $X$. If $X$ is tame and epimorphic, then the two of them together in some sense contain all the information about $X$, as we will see later.

Given a chain complex $C_\bullet$, we denote by $H_n(C_\bullet)$ the homology of $C_\bullet$ in degree $n$ with coefficients in $\Fi$.

\begin{definition}
     If $X\colon [0,\infty) \to \Ch$ is a parametrized chain complex, its \textbf{$n$-homology} is the parametrized vector space $H_n(X)\colon [0,\infty) \to \Vect$, defined by
    \begin{itemize}
        \item $H_n(X)^t := H_n(X^t)$, 
        \item $H_n(X)^{s\le t} := H_n(X^{s\le t})$.
    \end{itemize}
    Given a parametrized chain map $\varphi\colon X \to Y$, we define the parametrized linear map $H_n(\varphi)\colon H_n(X) \to H_n(Y)$ by 
    \begin{itemize}
        \item $H_n(\varphi)^t := H_n(\varphi^t)$.
    \end{itemize}
\end{definition}

\begin{definition}
     If $X\colon [0,\infty) \to \Ch$ is a parametrized chain complex, its \textbf{$n$-boundary} is the parametrized vector space $\partial_n(X)\colon [0,\infty) \to \Vect$, defined by
    \begin{itemize}
        \item $\partial_n(X)^t := \im(\partial_n^{X^t}) \subseteq X^t_{n-1}$, 
        \item $\partial_n(X)^{s\le t} := X^{s\le t}_{n-1}|_{\partial_n(X)^s}$.
    \end{itemize}
    Given a parametrized chain map $\varphi\colon X \to Y$, we define the parametrized linear map $\partial_n(\varphi)\colon \partial_n(X) \to \partial_n(Y)$ by 
    \begin{itemize}
        \item $\partial_n(\varphi)^t := \varphi^t|_{\partial_n(X)^t}$.
    \end{itemize}
\end{definition}

\begin{remark}
    There are functors $H_n,\partial_n\colon \Ch \to \Vect$, such that for all $X \in \PCh$ we have $H_n(X) = H_n\circ X$ and $\partial_n(X) = \partial_n \circ X$. From this it follows, that $H_n(X),\partial_n(X) \in \PVect$ and moreover, the assignments $X \mapsto H_n(X)$ and $X \mapsto \partial_n(X)$ are functors $\PCh \to \PVect$. Explicitly, this \textbf{functoriality} means that $H_n$ and $\partial_n$ have the following properties. 
    \begin{enumerate}[$(i)$]
        \item If $X \in \PCh$, then $H_n(X), \partial_n(X) \in \PVect$.
        \item If $\varphi\colon X \to Y$ is a parametrized chain map, then $H_n(\varphi)\colon H_n(X) \to H_n(Y)$ and $\partial_n(\varphi)\colon \partial_n(X) \to \partial_n(Y)$ are parametrized linear maps.
        \item Identity morphisms get mapped to identity morphisms, i.e.~$H_n(\one_X)=\one_{H_n(X)}$ and $\partial_n(\one_X)=\one_{\partial_n(X)}$ for all $X \in \PCh$.
        \item If $\psi\colon Y \to Z$ is another parametrized chain map, then $H_n(\psi\circ\varphi) = H_n(\psi)\circ H_n(\varphi)$ and $\partial_n(\psi\circ\varphi) = \partial_n(\psi)\circ \partial_n(\varphi)$.
        \item If $X \cong Y$ in $\PCh$, then $H_n(X)\cong H_n(Y)$ and $\partial_n(X) \cong \partial_n(Y)$ in $\PVect$.
    \end{enumerate}
\end{remark}

\subsection[The structure theorem for factored chain complexes]{The structure theorem in $\TEPCh$}\label{sec:structure-thm}

In this subsection we study decompositions of factored chain complexes and prove the structure theorem.

Let $X$ be a parametrized chain complex. A \textbf{parametrized subcomplex} of $X$ is a parametrized chain complex $Y$ such that $Y^t$ is a subcomplex of $X^t$ for all $t$ and the internal maps in $Y$ are given by the restrictions of the internal maps in $X$. We write $Y \subseteq X$.
Given parametrized chain complexes $X,Y$, their \textbf{direct sum} $X\oplus Y$ is the parametrized chain complex defined by $(X\oplus Y)^t := X^t \oplus Y^t$ and $(X\oplus Y)^{s\le t} := X^{s\le t} \oplus Y^{s\le t}$. The direct sum between any finite number of parametrized chain complexes is defined analogously. 

One can check that if $X,Y$ are parametrized chain complexes, then $X\oplus Y$ is again a parametrized chain complex. Since the direct sum of epimorphisms is an epimorphism, it follows that the direct sum of epimorphic parametrized chain complexes is again epimorphic. Also if $X$ and $Y$ are both tame, then also $X\oplus Y$ is tame. Note that the set where $X\oplus Y$ fails to be left-constant is the union of the corresponding sets for $X$ and $Y$. All of these arguments continue to hold when we pass to direct sums of arbitrary (but still finite) size. To summarize, we can say that the categories $\PCh$, $\EPCh$, and $\TEPCh$ are closed under finite direct sums. 

\begin{lemma}\label{lem:par-subcomp}
    Let $X$ be a parametrized chain complex and let $Y \subseteq X$ with $Y$ epimorphic.
    \begin{enumerate}[$(i)$]
        \item\label{item:par-subcomp-1} If $X^{s\le t}$ is an isomorphism, then also $Y^{s\le t}$ is an isomorphism.
        \item\label{item:par-subcomp-2} If $X$ is right-constant at $t \in [0,\infty)$, then $Y$ is also right-constant at $t$.
        \item\label{item:par-subcomp-3} If $X$ is left-constant at $t \in [0,\infty)$, then $Y$ is also left-constant at $t$.
        \item\label{item:par-subcomp-4} If $X\in \TEPCh$, then also $Y \in \TEPCh$. 
    \end{enumerate}
\end{lemma}

\begin{proof}
The statement $(\ref{item:par-subcomp-1})$ follows from the fact that the restriction of an injective chain map to a subcomplex is again injective. The statements $(\ref{item:par-subcomp-2})$ and $(\ref{item:par-subcomp-3})$ follow from $(\ref{item:par-subcomp-1})$. In order to show $(\ref{item:par-subcomp-4})$, we need to show that under the given assumptions, $Y$ is tame. We know that $X$ is tame, i.e.~right-constant everywhere and left-constant everywhere but on a finite set of points. By $(\ref{item:par-subcomp-2})$, also $Y$ is right-constant everywhere and by~$(\ref{item:par-subcomp-3})$, $Y$ can fail to be left-constant only at points where also $X$ fails to be left-constant, thus also only on a finite set of points. It follows that also $Y$ is tame and thus $Y \in \TEPCh$. 
\end{proof}

Now we prove some properties of the functors $H_n,\partial_n\colon \TEPCh \to \TPVect$.
In words, we show that both $H_n$ and $\partial_n$ preserve tameness and direct sums, while $\partial_n$ additionally preserves the property of being epimorphic. Applying $H_n$ or $\partial_n$ to interval functors in $\TEPCh$ yields interval functors in $\TPVect$.

\begin{proposition}\label{prop:par-hom-im}
    Let $X$ be a parametrized chain complex.
    \begin{enumerate}[$(i)$]
        \item\label{item:par-hom-im-1} If $X$ is tame, then both $H_n(X)$ and $\partial_n(X)$ are tame.
        \item\label{item:par-hom-im-2} If $X$ is epimorphic, then $\partial_n(X)$ is epimorphic.
        \item\label{item:par-hom-im-3} If $X = Y\oplus Z$, then $H_n(X) = H_n(Y) \oplus H_n(Z)$ and $\partial_n(X) = \partial_n(Y) \oplus \partial_n(Z)$.
        \item\label{item:par-hom-im-4} Interval functors in $\TEPCh$ induce interval functors in $\TPVect$ in the following way, where we use the convention that $\Fi[t,t)=0$ for any $t$.
        \[
        H_k(\I^n[0,s,t)) = \begin{cases}
            \Fi[s,t)    &\text{if } k=n,\\
            0           &\text{if } k \neq n,
        \end{cases}
        \qquad \text{and} \qquad
        \partial_k({\I^n[0,s,t)}) = \begin{cases}
            \Fi[0,s)    &\text{if } k=n,\\
            0           &\text{if } k \neq n.
        \end{cases}
        \]
    \end{enumerate}
\end{proposition}

\begin{proof}
    To see that $(\ref{item:par-hom-im-1})$ holds, note that if $X^{s\le t}$ is an isomorphism of chain complexes, then both $H_n(X^{s\le t})$ and the restriction of $X^{s\le t}_{n-1}$ to $\partial_n(X^s_n)$ are isomorphisms of vector spaces.

    To show $(\ref{item:par-hom-im-2})$, let $b \in \partial_n(X^t_n) \subseteq X^t_{n-1}$. Then $b=\partial a$ for some $a \in X^t_n$. Since $X^{s\le t}$ is an epimorphism, there exists $a' \in X^s_n$ with $X^{s\le t}_n(a')=a$. Thus we have $b= X^{s\le t}_{n-1}(\partial a')$. This shows that the internal map $\partial_n(X)^{s\le t}$ is an epimorphism for any $s\le t \in [0,\infty)$.

    It follows from standard arguments in homological algebra that homology and boundary both preserve direct sums. This proves $(\ref{item:par-hom-im-3})$.

    In order to show $(\ref{item:par-hom-im-4})$, note that for the $n$-disk $D^n$ and the $n$-sphere $S^n$ in $\Ch$ it holds that 
    \[
    H_k(D^n)=0, \text{ for all } k,
    \quad 
    H_k(S^n) = \begin{cases}
        \Fi, \text{ if } k=n,\\
        0, \text{ if } k\neq n,
    \end{cases}
    \quad
    \partial_k(D^n) = \begin{cases}
        \Fi, \text{ if } k=n,\\
        0, \text{ if } k\neq n,
    \end{cases}
    \quad 
    \partial_k(S^n)=0, \text{ for all } k.
    \]
    From this and the fact that both homology and the boundary map the identity to the identity, $(\ref{item:par-hom-im-4})$ follows.
\end{proof}

We introduce a new notation for the parametrized subcomplex induced by a single element.
Note that by applying \Cref{lem:non-zero-parmaps} in the case $s=t=\infty$, we get an isomorphism
\[
\Hom(\I^n[0,\infty,\infty),X) \overset{\cong}{\longrightarrow} X^0_n.
\]
Denote the inverse of this isomorphism by $\Phi\colon X^0_n \to \Hom(\I^n[0,\infty,\infty),X)$.

\begin{definition}
    For $X \in \PCh$ and $a \in X^0_n$, the parametrized subcomplex $\langle a \rangle$ of $X$ is the image of the parametrized chain map $\Phi(a)\colon \I^n[0,\infty,\infty) \to X$.
\end{definition}

Explicitly, this means that for $t \in [0,\infty)$, the parametrized chain complex $\langle a \rangle^t$ is given by
\[
\langle a \rangle^t_n = \Span(X_n^{0\le t}(a)), \qquad 
\langle a \rangle^t_{n-1} = \Span(X_{n-1}^{0\le t}(\partial a)), \qquad
\langle a \rangle^t_k = 0 \text{ if } k\neq n,{n-1}.
\]
The differential in degree $n$ is given by the restriction of the differential $\partial^t_n\colon X^t_n \to X^t_{n-1}$, all the other differentials are zero. The internal maps $\langle a \rangle^{s \le t}$ are given by the restrictions of the internal maps of $X$. Note that the dependence of $\langle a \rangle$ from $X$ is hidden from the notation.
Alternatively, $\langle a \rangle$ can be characterized as the smallest parametrized subcomplex of $X$ containing $a$.

Given a parametrized chain complex $X$ and an element $a \in X^0_n$, we define its \textbf{death time} as $d(a) := \inf \{s \ge 0 \mid X^{0\le s}_n(a)=0\}$. We also say that $a$ \textbf{dies at time} $d(a)$. We start by observing some properties of death times that will be useful later. In the following we adopt the convention that $\max\{\lambda,\infty\}=\infty$ for any number $\lambda$, and $\max(\emptyset)=0$.

\begin{lemma}\label{lem:death-time-linear-combination}
    For $X \in \PCh$, elements $a_1,\ldots,a_k \in X^0_n$, and scalars $\lambda_1,\ldots,\lambda_k \in \Fi$, we have 
    \begin{enumerate}[(i)]
        \item\label{item:death-time-lin-comb-1} $d(\lambda_1 a_1 +\cdots+ \lambda_k a_k) \le \max \{ d(a_i) \mid \lambda_i \neq 0\}$.
        \item\label{item:death-time-lin-comb-2} If $d(a_1),\ldots,d(a_k)$ are distinct, then $d(\lambda_1 a_1 +\cdots+ \lambda_k a_k) = \max \{ d(a_i) \mid \lambda_i \neq 0\}$.
        \item\label{item:death-time-lin-comb-3} If $d(a_1),\ldots,d(a_k)$ are distinct and none of them is zero, then $a_1,\ldots,a_k$ are linearly independent.
    \end{enumerate}
\end{lemma}

\begin{proof}
    The inequality in $(\ref{item:death-time-lin-comb-1})$ holds simply because if $s\ge d(a_i)$ for all $i$, then $X^{0\le s}_n(\lambda_1 a_1 +\cdots+ \lambda_k a_k) = \lambda_1 a_1^s +\cdots+ \lambda_k a_k^s =0$.

    For $(\ref{item:death-time-lin-comb-2})$, assume now that the death times are all distinct. Without loss of generality assume that $d(a_1)< \cdots < d(a_k)$, with $\lambda_k \neq 0$. Then, for any $d(a_{k-1}) < s < d(a_k)$, we have $X^{0\le s}_n(\lambda_1 a_1 +\cdots+ \lambda_k a_k) = \lambda_k a_k^s \neq 0$. This shows that $d(\lambda_1 a_1 +\cdots+ \lambda_k a_k) \ge d(a_k)$.

    Finally, for $(\ref{item:death-time-lin-comb-3})$, assume by contradiction that $a_1,\ldots,a_k$ are linearly dependent, i.e. there exists $i$ such that we can write $a_i$ as a linear combination of the other $a_j$. Then, by $(\ref{item:death-time-lin-comb-2})$, the death time of $a_i$ is among the death times of the other $a_j$, contradicting the fact that they are $k$ distinct values.
\end{proof}

\begin{proposition}\label{prop:par-span}
    Let $X \in \PCh$ and let $0\neq a \in X^0_n$. If $X$ is right-constant at $s:=d(\partial a)$ and at $t:=d(a)$, then $\langle a \rangle$ is isomorphic to $\I^n[0,s,t)$, with $t>0$. In particular, $\langle a \rangle$ is a factored chain complex.
\end{proposition}

\begin{proof}
    Since $X$ is right-constant at $s$ and $t$, it follows that $X^{0\le s}_{n-1}(\partial a)=0$ and $X^{0\le t}_n(a)=0$. Therefore, by \Cref{lem:non-zero-parmaps}, there exists a parametrized chain map $\varphi\colon \I^n[0,s,t) \to X$ with $\varphi^0_n(1_n)=a$. One can check that the image of $\varphi$ is precisely $\langle a \rangle$ and that $\varphi^r_k$ is injective for all $r$ and $k$. Therefore, it follows that $\varphi$ is an isomorphism from $\I^n[0,s,t)$ to $\langle a \rangle$. Note that $t>0$, because otherwise $a=0$ would follow, since $X$ is right-constant at $t$.
\end{proof}

We want to formulate and prove a structure theorem for factored chain complexes. More precisely, we want to prove that every factored chain complex is isomorphic to a direct sum of interval functors in $\TEPCh$. We begin with a statement from linear algebra that will turn out to be useful for this.

\begin{proposition}\label{prop:direct-sum-image}
    Let $f\colon V \twoheadrightarrow W$ be a surjective linear map and $A,B \subseteq V$ linear subspaces such that $V=A\oplus B$ and $\ker(f) \subseteq B$. Then $W = f(A) \oplus f(B)$.
\end{proposition}

\begin{proof}
    Denote by $\Tilde{f}\colon V/\ker(f) \to W$ the isomorphism induced by $f$. Note that the decomposition $V=A\oplus B$ implies that $V/\ker(f) = \frac{A}{A\cap \ker(f)} \oplus \frac{B}{B\cap \ker(f)} = A \oplus B/\ker(f)$. Therefore,
    \[
    W = \Tilde{f}(V/\ker(f)) = \Tilde{f}(A\oplus B/\ker(f)) = \Tilde{f}(A) \oplus \Tilde{f}(B/\ker(f)) = f(A)\oplus f(B). \qedhere
    \]
\end{proof}

Next we state a result that says that we can split off an interval functor from any factored chain complex. This is the heart of the existence part of the structure theorem. As before, for any $r\ge 0$, we denote the differentials in the chain complex $X^r$ by $\partial^r_k\colon X^r_k \to X^r_{k-1}$. Moreover we write $a^r := X^{0\le r}_n(a)$ and $\partial a^r := X^{0\le r}_{n-1}(\partial^0_n a) = \partial^r_n(a^r)$. 

\begin{lemma}\label{lem:split-off-interval}
    Let $X \in \TEPCh$ and $a \in X^0_n \setminus \im(\partial_{n+1})$ satisfying the condition
    \[
    d(a) = \min \{d(a') \mid a' \in X^0_n,\ d(\partial a') = d(\partial a) \}.
    \] 
    Then there exists a factored subcomplex $Y \subseteq X$ such that $X \cong \langle a \rangle \oplus Y$. 
\end{lemma}

\begin{proof}
    Since $a \notin \im(\partial_{n+1})$, we have $a\neq 0$ and thus, by \Cref{prop:par-span}, there exists a tagged interval $[0,s_1,t_1) \in \J$, such that $\langle a \rangle \cong \I^n[0,s_1,t_1)$. By tameness of $X$, if $s_1>0$, there exist $s_0 < s_1$ and $t_0 < t_1$, such that for all $s_0\le s < s_1$ and for all $t_0 \le t < t_1$, the internal chain maps $X^{s_0\le s}$ and $X^{t_0\le t}$ are isomorphisms. In the case $s_1=0$, we do not choose $s_0$ and in the case $s_1=t_1$, we may choose $s_0=t_0$. We construct a direct summand $Y$ of $X$, complementary to $\langle a \rangle$, in five steps.

    \begin{enumerate}
        \item We define $Y$ to be all of $X$ in all degrees where $\langle a \rangle$ is zero. Explicitly: For $k\neq n,n-1$ and any $r\ge 0$, we define $Y^r_k := X^r_k$.

        \item\label{item:direct-summand-step-2} 
        Let $Y^0_n \subseteq X^0_n$ be a linear subspace such that 
        \[
        \ker(X^{0\le t_0}_n) \subseteq Y^0_n,
        \qquad \qquad
        \ker(\partial^{s_0}_n \circ X^{0\le s_0}_n) \subseteq Y^0_n,
        \qquad \text{and} \qquad
        X^0_n = \langle a \rangle^0_n \oplus Y^0_n
        \]
        (in the case $s_1=0$ we simply drop the second condition). In terms of death times, the first two conditions are saying that any element in $X^0_n$ that dies before $a$ or whose boundary dies before $\partial a$ has to be contained in $Y^0_n$. Of course we need to argue why such a subspace exists. Note that $\langle a \rangle^0_n \cap (\ker(X^{0\le t_0}_n) + \ker(\partial^{s_0}_n \circ X^{0\le s_0}_n)) =\{0\}$, since otherwise we could write $a=b_1+b_2$ with $b_1^{t_0}=0$ and $\partial b_2^{s_0} =0$. In particular $d(b_1)<t_1=d(a)$ and $d(\partial b_2)<s_1=d(\partial a)$. This would imply that $d(\partial b_1)=d(\partial a - \partial b_2) = \max\{d(\partial a),d(\partial b_2)\} = d(\partial a)$ by \Cref{lem:death-time-linear-combination}, hence violating the condition that the death time of $a$ is minimal among all elements whose boundaries have the same death time as $\partial a$. By standard linear algebra arguments there exists a subspace $Y^0_n \subseteq X^0_n$ that contains $\ker(X^{0\le t_0}_n) + \ker(\partial^{s_0}_n \circ X^{0\le s_0}_n)$ and is a complement of $\langle a \rangle^0_n$ in $X^0_n$.

        \item For any $r\ge 0$, we define $Y^r_n := X^{0\le r}_n(Y^0_n)$. It follows from the previous step and \Cref{prop:direct-sum-image} that $X^r_n = \langle a \rangle^r_n \oplus Y^r_n$ for all $r<t_1$. For larger values of $r$, it follows from $a^r=0$ and surjectivity of $X^{0\le r}_n$ that $Y^r_n = X^r_n$.

        \item\label{item:direct-summand-step-4} In the case $s_1=0$, we define $Y^0_{n-1}:=X^0_{n-1}$. Otherwise, let $Y^0_{n-1} \subseteq X^0_{n-1}$ be a linear subspace such that 
        \[
        \partial^0_n(Y^0_n) \subseteq Y^0_{n-1},
        \qquad \qquad
        \ker(X^{0\le s_0}_{n-1}) \subseteq Y^0_{n-1},
        \qquad \text{and} \qquad
        X^0_{n-1} = \langle a \rangle^0_{n-1} \oplus Y^0_{n-1}.
        \]
        Again, for such a subspace to exist it is sufficient to show that $\langle a \rangle^0_{n-1} \cap (\partial^0_n(Y^0_n) + \ker(X^{0\le s_0}_{n-1})) =\{0\}$. If that were not the case, then $\partial a = \partial b + c$ for some $b \in Y^0_n$ and $c \in X^0_{n-1}$ with $c^{s_0}=0$. In particular $d(c)< d(\partial a)$. Defining $a':=a-b$ would then yield an element of $X^0_n$ with $d(\partial a')=d(\partial a - \partial b) = d(c)< d(\partial a)$. This would mean that $(\partial a')^{s_0}=0$, thus $a' \in Y^0_n$. But, since $b \in Y^0_n$, this would imply that $a=a'-b \in Y^0_n$, contradicting the fact that $\langle a \rangle^0_n \cap Y^0_n =0$. As in step \ref{item:direct-summand-step-2} we can now construct $Y^0_{n-1} \subseteq X^0_{n-1}$ by standard linear algebra arguments.

        \item For any $r\ge 0$, we define $Y^r_{n-1} := X^{0\le r}_{n-1}(Y^0_{n-1})$. If follows from the previous step and \Cref{prop:direct-sum-image} that $X^r_{n-1
        }= \langle a \rangle^r_{n-1
        }\oplus Y^r_{n-1
        }$ for all $r<s_1$. For larger values of $r$, it follows from $\partial a^r=0$ and surjectivity of $X^{0\le r}_{n-1}$ that $Y^r_{n-1} = X^r_{n-1}$.
    \end{enumerate}
    So far we have constructed linear subspaces $Y^r_k \subseteq X^r_k$ for all $r,k$ with the property that $X^r_k = \langle a \rangle^r_k \oplus Y^r_k$. The internal maps of $Y$, both vertical and horizontal, are defined to be the restrictions of the corresponding maps in $X$. It remains to check that $Y$ is indeed a tame epimorphic parametrized subcomplex of $X$. 
    
    Let us first check that $Y^0$ is a subcomplex of $X^0$, i.e.~$\partial^0_k(Y^0_k) \subseteq Y^0_{k-1}$ for all $k$. We only need to check those degrees where $Y^0_{k-1} \neq X^0_{k-1}$, so we need to check that $\partial^0_{n+1}(Y^0_{n+1}) \subseteq Y^0_n$ and $\partial^0_n(Y^0_n) \subseteq Y^0_{n-1}$. The latter is true by construction of $Y^0_{n-1}$ in step \ref{item:direct-summand-step-4} above. If $s_1=0$, then the former holds because in that case $Y^0_{n-1}=X^0_{n-1}$. Otherwise, it holds because
    \[
    \partial^0_{n+1}(Y^0_{n+1})
    = \partial^0_{n+1}(X^0_{n+1})
    \subseteq \ker(\partial^0_n)
    \subseteq \ker(X^{0\le s_0}_{n-1} \circ \partial^0_n)
    = \ker(\partial^{s_0}_n \circ X^{0\le s_0}_n)
    \subseteq Y^0_n,
    \]
    where the last inclusion holds by construction of $Y^0_n$ in step \ref{item:direct-summand-step-2} above. To show that $Y^r$ is a subcomplex of $X^r$ for any $r\ge 0$, first note that the restriction of $X^{0\le r}_k$ to $Y^0_k$ yields a surjective map $Y^0_k \twoheadrightarrow Y^r_k$, by construction of $Y^r_k$. We can use this fact together with the knowldege that $Y^0$ is a subcomplex of $X^0$ to show that $Y^r$ is a subcomplex of $X^r$. Thus we have shown that $Y$ is an epimorphic parametrized subcomplex of $X$. Tameness of $Y$ follows from \Cref{lem:par-subcomp} $(\ref{item:par-subcomp-4})$.
\end{proof}

Now we are ready to prove the structure theorem in $\TEPCh$.

\begin{theorem}[Structure theorem in $\TEPCh$]\label{thm:structure-theorem}
    Any factored chain complex $X$ is isomorphic to a finite direct sum of interval functors in $\TEPCh$, i.e.~for every $n\in \N$, there exists a unique finite multiset $\tBarc_n = \tBarc_n(X) \in \Mult(\J)$ of tagged intervals, such that 
    \[
    X \cong \bigoplus_{n \in \N} \bigoplus_{[0,s,t) \in \tBarc_n} \I^n[0,s,t).
    \]
\end{theorem}

The collection of multisets $\tBarc(X)=\left(\tBarc_n(X)\right)_{n\in \N}$ is called the \textbf{ tagged barcode} of $X$. Sometimes we will call $\tBarc_n(X)$ the tagged barcode in degree $n$ (or $n$-th tagged barcode).

\begin{proof}[Proof of existence]
    The existence of such a decomposition can be shown by iterated use of \Cref{lem:split-off-interval}. Formally, we do a proof by induction over the number $m(X) := \sum_k \dim(X^0_k)$. 
    
    If $m(X)=0$, then $X=0$ and the statement holds for $X$, since then, $X$ is isomorphic to the empty direct sum, i.e.~$\tBarc_n(X)=\emptyset$ for all $n$. 

    Assume now that $m(X)=m>0$ and that we have already proven the statement for all $Z\in \TEPCh$ with $m(Z)< m$. Since $X\neq 0$, there must exist some $n\ge0$ and $a \in X^0_n \setminus \im(\partial_{n+1})$ whose death time is minimal among all elements whose boundary has the same death time as $\partial a$. By tameness of $X$, there can be only finitely many different death times among the elements of $X^0_n$, thus an element with minimal death time always exists (since $X^0_n$ is finite-dimensional, this actually also follows from \Cref{lem:death-time-linear-combination} $(\ref{item:death-time-lin-comb-3})$ even without the tameness assumption). By \Cref{lem:split-off-interval}, we have $X=\langle a \rangle \oplus Y$ for some $Y \in \TEPCh$. Clearly $m(Y)<m(X)$ and thus $Y$ can be written as a finite direct sum of interval functors in $\TEPCh$. Furthermore, $\langle a \rangle$ is isomorphic to an interval functor in $\TEPCh$ by \Cref{prop:par-span}, thus also $X$ is isomorphic to a finite direct sum of interval functors in $\TEPCh$.
\end{proof}

\begin{proof}[Proof of uniqueness]
    Assume that there are two collections of multisets $\tBarc_k$ and $\tBarc'_k$, $k\in \N$, such that 
    \[
    X:= \bigoplus_{k \in \N} \bigoplus_{[0,s,t) \in \tBarc_k} \I^k[0,s,t) \cong \bigoplus_{k \in \N} \bigoplus_{[0,s,t) \in \tBarc'_k} \I^k[0,s,t) =:Y.
    \]
    Fix a number $n\in \N$. By functoriality of $H_n(\square)$ and $\partial_n(\square)$, this implies that $H_n(X) \cong H_n(Y)$ and $\partial_n(X) \cong \partial_n(Y)$. By \Cref{prop:par-hom-im} $(\ref{item:par-hom-im-3})$ and $(\ref{item:par-hom-im-4})$, we have 
    \[
    H_n(X)= \bigoplus_{\substack{[s,t) \in \tBarc_n, \\ s<t}} \Fi[s,t)
    \qquad \text{and} \qquad
    \partial_n(X) = \bigoplus_{\substack{[s,t) \in \tBarc_n, \\ s>0}}\Fi[0,s).
    \]
    This implies that we can obtain the barcodes of $H_n(X)$ and $\partial_n(X)$ by adding, for every tagged interval $[0,s,t) \in \tBarc_n$, an interval of the form $[0,s)$ to $\Barc(\partial_n(X))$ and, if $s\neq t$, an interval of the form $[s,t)$ to $\Barc(H_n(X))$.
    Applying the definition that a multiset is a function, where the function value is interpreted as the multiplicity of a (tagged) interval, the barcodes of $H_n(X)$ and $\partial_n(X)$ are given by
    \[
    \Barc(H_n(X))([s,t)) = \begin{cases}
        \tBarc_n([0,s,t)), \text{ if } s<t, \\
        0, \text{ if } s=t,
    \end{cases} 
    \text{and} \quad 
    \Barc(\partial_n(X))([r,s)) = \begin{cases}
        \sum_{s\le t} \tBarc_n([0,s,t)), \text{ if } r=0, \\
        0, \text{ if } r>0.
    \end{cases}
    \]
    We can apply the same reasoning to $Y$ to get the analogous expressions for $\Barc(H_n(Y))$ and $\Barc(\partial_n(Y))$, just with $\tBarc_n$ replaced by $\tBarc_n'$. By the uniqueness in \Cref{thm:str-thm-vect} we have $\Barc(H_n(X))=\Barc(H_n(Y))$ and $\Barc(\partial_n(X))=\Barc(\partial_n(Y))$. The first equality implies that 
    \[
    \tBarc_n([0,s,t)) = \tBarc_n'([0,s,t)) \text{ for all } s<t \in [0,\infty],
    \]
    so it only remains to show that $\tBarc_n$ and $\tBarc_n'$ also agree on tagged intervals of the form $[0,s,s)$. For this, consider
    \begin{align*}
        \Barc(\partial_n(X))([0,s)) = \sum_{s\le t} \tBarc_n([0,s,t)) &= \tBarc_n([0,s,s)) + \sum_{s<t} \tBarc_n([0,s,t))\\
        \rotatebox{90}{=} \hspace{4.4cm} &\\
        \Barc(\partial_n(Y))([0,s)) = \sum_{s\le t} \tBarc_n'([0,s,t)) &= \tBarc_n'([0,s,s)) + \sum_{s<t} \tBarc_n'([0,s,t)).
    \end{align*}
    Since we have seen before that $\tBarc_n$ and $\tBarc_n'$ agree on tagged intervals $[0,s,t)$ with $s<t$, it follows from this chain of equations that also $\tBarc_n([0,s,s))=\tBarc_n'([0,s,s))$. Thus we have shown that $\tBarc_n=\tBarc_n'$ and since $n$ was chosen arbitrarily, this concludes the proof.
\end{proof}

The following corollary gives sufficient conditions for the existence of certain elements in the tagged barcode of a factored chain complex.

\begin{corollary}\label{cor:barcode}
    Let $X \in \TEPCh$ be such that all of its internal chain maps induce isomorphisms in homology in all degrees and let $s<t \in [0,\infty)$ be such that 
    \begin{enumerate}[$(i)$]
        \item for all $s< s' <t$ the internal chain map $X^{s\le s'}$ is an isomorphism, \label{item:barcode-1}
        \item $\dim(X^t_k) = \dim(X^s_k)-1$ for $k=n,n-1$ and $\dim(X^t_k) = \dim(X^s_k)$ for all other values of $k$. \label{item:barcode-2}
    \end{enumerate}
    Then the tagged interval $[0,t,t)$ appears with multiplicity 1 in the tagged barcode $\tBarc_n(X)$ and does not appear (i.e. has multiplicity 0) in all other tagged barcodes $\tBarc_m(X)$, with $m\neq n$. Moreover, for any $s<t'<t$, the tagged interval $[0,t',t')$ does not appear in the tagged barcode of $X$.
\end{corollary}

\begin{proof}
    By \Cref{thm:structure-theorem}, $X$ is a direct sum of interval functors in $\TEPCh$. Since all internal chain maps of $X$ induce isomorphisms in homology, only tagged intervals of the form $[0,t',t')$, for $0<t'\le \infty$, can appear in the tagged barcode of $X$. By $(\ref{item:barcode-2})$, in the decomposition of $X$ there must be precisely one interval functor $\I^n[0,t',t')$, and no interval functor $\I^m[0,t',t')$ for $m\neq n$, with $s<t'\le t$. Condition $(\ref{item:barcode-1})$ prevents the case $t'<t$, so the only remaining possibility is what was claimed.
\end{proof}

\subsection[The isometry theorem for factored chain complexes]{The isometry theorem in $\TEPCh$}\label{sec:iso-thm-tepch}

\paragraph{The interleaving distance.}

The interleaving distance in $\TEPCh$ is precisely the one that we introduced in larger generality in \Cref{sec:par-obj}, applied to the category $\C=\Ch$. Here we give some explicit formulae for computing the interleaving distance between interval functors in $\TEPCh$.

\begin{remark}\label{rem:shift-of-interval}
    It follows from the definition of $\I^n[0,s,t)$ and the $\eps$-shift that for any $[0,s,t) \in \J$ and $\eps \ge 0$ we have
    \[
    \I^n[0,s,t)_\eps = 
        \I^n[0,\max\{s-\eps,0\},\max\{t-\eps,0\}),
    \]
    where $\I^n[0,0,0) := 0$.
\end{remark}

\begin{lemma}\label{lem:intervals-tepch-eps-interleaved}
    Given $\eps>0$ and two tagged intervals $J=[0,s,t)$ and $J'=[0,s',t')$, the corresponding interval functors in $\TEPCh$, i.e.~$\I^n[0,s,t)$ and $\I^n[0,s',t')$, are $\eps$-interleaved if and only if
    \[
    \eps \ge \min \left\{ \max\left\{|s-s'|,|t-t'|\right\}, \max\left\{\frac{t}{2}, \frac{t'}{2} \right\} \right\},
    \]
    where we use the convention that $|\infty-\infty|=0$.
\end{lemma}

\begin{proof}
    If $\frac{t}{2},\frac{t'}{2} \le \eps$, then the internal maps of length $2\eps$ are zero in both $\I^n[0,s,t)$ and $\I^n[0,s',t')$. Therefore, the zero maps are an $\eps$-interleaving between $\I^n[0,s,t)$ and $\I^n[0,s',t')$.

    If $|s-s'|,|t-t'| \le \eps$, then, by \Cref{rem:shift-of-interval} and \Cref{lem:non-zero-morphism}, there exist non-zero parametrized chain maps $\phi\colon \I^n[0,s,t) \to \I^n[0,s',t')_\eps$ and $\psi\colon \I^n[0,s',t') \to \I^n[0,s,t)_\eps$. More explicitly, we let $\phi$ and $\psi$ be the maps constructed in the proof of \Cref{lem:non-zero-morphism}. One can then check that the compositions $\psi_\eps \circ \phi$ and $\phi_\eps\circ \psi$ agree with the internal maps of length $2\eps$ in $\I^n[0,s,t)$ and $\I^n[0,s',t')$. 

    Let us now assume that $\eps < \min \left\{ \max\left\{|s-s'|,|t-t'|\right\}, \max\left\{\frac{t}{2}, \frac{t'}{2} \right\} \right\}$ and show that in this case there exists no $\eps$-interleaving. Since $\eps < \max\left\{\frac{t}{2},\frac{t'}{2}\right\}$, we can assume without loss of generality that $\eps < \frac{t}{2}$. Therefore, there exist internal maps in $\I^n[0,s,t)$ of length $2\eps$ which are non-zero. Therefore, any $\eps$-interleaving must consist of non-zero morphisms. However, since $\eps < \max\left\{|s-s'|,|t-t'| \right\}$, it follows from \Cref{rem:shift-of-interval} and \Cref{lem:non-zero-morphism} that no non-zero morphism with a shift of $\eps$ can exist in one of the two directions.
\end{proof}

\begin{corollary}\label{cor:intervals-tepch-interleaved}
    Two interval functors $\I^n[0,s,t)$ and $\I^n[0,s',t')$ in $\TEPCh$ are interleaved if and only if one of the following cases holds:
    \begin{enumerate}[$(i)$]
        \item $t,t'<\infty$,\label{item:type-1}
        \item $t=t'=\infty$ and $s,s'<\infty$,\label{item:type-2}
        \item $t=t'=s=s'=\infty$.\label{item:type-3}
    \end{enumerate}
    More generally, if $X,Y \in \TEPCh$ are interleaved, then it follows that the tagged barcodes of $X$ and $Y$ in each degree have the same amount of intervals of the form $[0,s,\infty)$ with $s<\infty$ and of the form $[0,\infty,\infty)$.
\end{corollary}

\begin{proof}
    By \Cref{lem:intervals-tepch-eps-interleaved}, the two interval functors in $\TEPCh$ are interleaved if and only if $\max\{|s-s'|,|t-t'|\}<\infty$ or $\max\left\{\frac{t}{2},\frac{t'}{2}\right\}<\infty$. Note that the former is equivalent to saying that $s$ and $s'$ are either both finite or both infinite and $t$ and $t'$ are either both finite or both infinite. The statement  $\max\left\{\frac{t}{2},\frac{t'}{2}\right\}<\infty$ is equivalent to saying that $t$ and $t'$ are both finite, so it is a stronger statement and we can ignore it. Thus we get the three cases described in the statement: Either $t,t',s,s'$ are all finite, or $t,t'$ are infinite and $s,s'$ are finite, or $t,t',s,s'$ are all infinite.
    
    In order to prove the more general statement, note that by tameness there exists $T>0$ such that for all $T\le s \le t$, the internal maps $X^{s\le t}$ and $Y^{s\le t}$ are isomorphisms. From this it follows that the chain complexes $X^T$ and $Y^T$ are isomorphic and must thus have the same number of disks and spheres in their decomposition. The number of $n$-disks in the decomposition of $X^T$ is equal to the number of copies of $\I^n[0,\infty,\infty)$ in the decomposition of $X$ and the number of $n$-spheres in the decomposition of $X^T$ is equal to the number of copies of $\I^n[0,s,\infty)$ for any $s<\infty$ in the decomposition of $X$. The same holds for $Y^T$ and $Y$, hence the result follows.
\end{proof}

\paragraph{The bottleneck distance.}\label{sec:tagged-barcodes-bottleneck-dist}

Similar to the interleaving distance, the bottleneck distance is a special case of the generalized bottleneck distance described in \Cref{sec:par-obj}. We only need to give a pseudometric and a weight function on the set $\J$ of tagged intervals.

\begin{definition}
    Given two tagged intervals $J=[0,s,t)$ and $J'=[0,s',t')$, we define the \textbf{cost of matching $J$ to $J'$} as $c(J,J') := \max\{|s-s'|,|t-t'|\}$. The \textbf{weight of $J$} is defined as $W(J):=\frac{t}{2}$.
\end{definition}

One can check that $c$ is an extended metric on $\J$ and that $c$ and $W$ are compatible. Therefore, by \Cref{rem:bott-distance-ext-pseudometric}, the bottleneck distance $d_B$ is an extended pseudometric on $\Mult(\J)$. 

Note that the bottleneck distance of two multisets of tagged intervals from $\J$ is finite if and only if they contain the same number of tagged intervals of the form $[0,\infty,\infty)$ and of the form $[0,s,\infty)$ with $s<\infty$, respectively. This is analogous to the \Cref{cor:intervals-tepch-interleaved} about interleaving distances between interval functors in $\TEPCh$. 

\paragraph{The isometry theorem.}

We are now ready to state the isometry theorem and prove one of the two inequalities.

\begin{theorem}[Isometry Theorem in $\TEPCh$] \label{thm:iso-thm-tepch}
    If $X$ and $Y$ are factored chain complexes, then
    \[
    d_I(X,Y) = \max_{n\in \N} d_B(\tBarc_n(X),\tBarc_n(Y)).
    \]
\end{theorem}

\begin{proof}[Proof of the inequality $\le$]
    It suffices to prove that if for some $\eps>0$ there exist $\eps$-matchings between $\tBarc_n(X)$ and $\tBarc_n(Y)$ for all $n$, then there exists an $\eps$-interleaving between $X$ and $Y$. By \Cref{thm:structure-theorem}, we have that
    \[
    X \cong \bigoplus_{n\in \N} \left( \bigoplus_{[0,s,t) \in \tBarc_n(X)} \I^n[0,s,t) \right)
    \qquad \text{and} \qquad
    Y \cong \bigoplus_{n\in \N} \left( \bigoplus_{[0,s',t')\in \tBarc_n(Y)} \I^n[0,s',t') \right).
    \]
    Fix a number $\eps>0$ and assume that there exists an $\eps$-matching between $\tBarc_n(X)$ and $\tBarc_n(Y)$ for each $n$. From this, we want to construct an $\eps$-interleaving between $X$ and $Y$. If $[0,s,t) \in \tBarc_n(X)$ and $[0,s',t') \in \tBarc_n(Y)$ are matched, then $\max\{|s-s'|,|t-t'|\} \le \eps$. Therefore $\I^n[0,s,t)$ and $\I^n[0,s',t')$ are $\eps$-interleaved by \Cref{lem:intervals-tepch-eps-interleaved}. For intervals $[0,s,t) \in \tBarc_n(X)$ and $[0,s',t') \in \tBarc_n(Y)$ that are unmatched we have $t,t' \le 2\eps$, and therefore the internal maps of length $2\eps$ in $I^n[0,s,t)$ and $\I^n[0,s',t')$ are zero, thus they are $\eps$-interleaved with the zero parametrized chain complex.

    We now construct an $\eps$-interleaving between $X$ and $Y$. Given a interval $[0,s,t) \in \tBarc_n(X)$ that is matched with $[0,s',t') \in \tBarc_n(Y)$, we define $\phi$ on the summand $\I^n[0,s,t)$ to be the interleaving constructed in the proof of \Cref{lem:intervals-tepch-eps-interleaved}. For any unmatched interval $[0,s,t) \in \tBarc_n(X)$ we define $\phi$ to be zero on $\I^n[0,s,t)$. By extending linearly, this defines a parametrized chain map $\phi\colon X \to Y_\eps$. We define $\psi\colon Y \to X_\eps$ analogously. It follows that $(\phi,\psi)$ is an $\eps$-interleaving between $X$ and $Y$.
\end{proof}

For the other inequality, we mimic the proof given in \cite{Bjerkevik_StabilityDPS_21} as suggested by previous attempts \cite{StabFiltCh}. We begin by setting up the notation.

\begin{definition}
    Given $X \in \TEPCh$ and $\eps>0$, we define $\tBarc_n^\eps(X):=\{I \in \tBarc_n(X) \mid W(I)> \eps\}$. 
    Given $X,Y \in \TEPCh$, we define
    for $I\in \tBarc_n^\eps(X)$ and for $A\subseteq \tBarc_n^\eps(X)$
    \[
    \mu_\eps(I) := \{J \in \tBarc_n(Y) \mid c(I,J) \le \eps \}
    \quad \text{ and } \quad
    \mu_\eps(A) := \bigcup_{I \in A}\mu(I).
    \]
    For $J \in \tBarc_n^\eps(Y)$ and $A \subseteq \tBarc_n^\eps(Y)$ the sets $\mu_\eps(J)\subseteq \tBarc_n(X)$ and $\mu_\eps(A)\subseteq \tBarc_n(X)$ are defined analogously.
\end{definition}

The proof of \Cref{thm:iso-thm-tepch} uses Hall's theorem. For this we recall some notions from graph theory. Given a bipartite graph $H$ with vertex set $L\sqcup R$, for any subset $A\subseteq L$ we define $\mu(A)$ to be the set of those vertices in $R$ that have an edge to at least one vertex in $A$. For $A\subseteq R$ we define $\mu(A)$ analogously. A \textbf{partial matching} $\M$ on $H$ is a set of edges of $H$ such that no two edges share a common vertex. We say that a matching $\M$ \textbf{covers} a set $A\subseteq L$ (or $A\subseteq R$) if every vertex of $A$ is contained in an edge of $\M$. Now we can state Hall's marriage theorem.

\begin{theorem}[Hall's marriage theorem \cite{HallMarriageTheorem}]\label{thm:hall-marriage}
    Let $H$ be a finite bipartite graph on the vertex set $L \sqcup R$. Then there exists a partial matching on $H$ that covers $L$ if and only if, for any subset $A\subseteq L$, we have $|\mu(A)| \ge |A|$.
\end{theorem}

We will need a slightly stronger version of this theorem, which we formulate in the following corollary.

\begin{corollary}\label{cor:Hall-marriage}
    Let $H$ be a finite bipartite graph on the vertex set $L \sqcup R$ and assume we are given two subsets $L^*\subseteq L$ and $R^*\subseteq R$. Then there exists a partial matching on $H$ that covers both $L^*$ and $R^*$ if and only if, for any subset $A\subseteq L^*$ or $A\subseteq R^*$, we have $|\mu(A)| \ge |A|$.
\end{corollary}

\begin{proof}
    As in the original Hall's theorem, it is not hard to see that if there exists a partial matching covering $L^*$ and $R^*$, then this condition is satisfied. It thus remains to show the other implication. Denote by $H_1$ the subgraph of $H$ resulting from restricting the vertex set to $L^*\sqcup R$ and denote by $H_2$ the subgraph of $H$ resulting from restricting the vertex set to $L\sqcup R^*$. By \Cref{thm:hall-marriage} there exists a matching $\M_1$ on $H_1$ covering $L^*$ and a matching $\M_2$ on $H_2$ covering $R^*$. We now view both $\M_1$ and $\M_2$ as partial matchings on $H$ and we want to combine them into a partial matching that covers both $L^*$ and $R^*$.

    We start with $\M_1$. This covers all vertices of $L^*$, but it may not cover all vertices of $R^*$. If we are given such a vertex $x\in R^*$, it must be covered in $\M_2$, so we can start drawing a path through $H$, starting at $x$ and alternating between edges from $\M_1$ and $\M_2$. By finiteness, this path must end at some point, and it can only end outside of $L^*$ and $R^*$. The reason for this is that when we follow the path up to a vertex $y \in L^*$, then we must have arrived there by an edge from $\M_2$. However, this vertex $y$ is covered in the matching $\M_1$, so the path cannot end there. An analogous argument shows why it cannot end in $R^*$. We can thus replace all the edges from $\M_1$ along this path with the edges from $\M_2$ along the path and get a matching which still covers $L^*$ and additionally covers the vertex $y$. Repeating this argument eventually yields a partial matching that covers both $L^*$ and $R^*$.
\end{proof}

The following result indicates how we are going to apply Hall's theorem.

\begin{lemma}\label{lem:apply-hall}
    If $X,Y \in \TEPCh$ are $\eps$-interleaved, then for any subset $A \subseteq \tBarc_n^\eps(X)$ or $A \subseteq \tBarc_n^\eps(Y)$, we have $|\mu_\eps(A)| \ge |A|$.
\end{lemma}

Before we prove \Cref{lem:apply-hall}, we explain how it helps us finish the proof of the isometry theorem. 

\begin{proof}[Proof of the inequality $\ge$ in \Cref{thm:iso-thm-tepch}]
    Let $\eps>0$ and assume that $X$ and $Y$ are $\eps$-interleaved. We need to show that for any $n\in \N$, there exists an $\eps$-matching between $\tBarc_n(X)$ and $\tBarc_n(Y)$. We construct a finite bipartite graph $H$ as follows. Let the vertex set of $H$ be given by $\tBarc_n(X) \sqcup \tBarc_n(Y)$. Connect $I \in \tBarc_n(X)$ and $J\in \tBarc_n(Y)$ with an edge if and only if $c(I,J)\le \eps$. Thus in the graph $H$, it holds that for any subset $A\subseteq \tBarc_n(X)$ or $A\subseteq \tBarc_n(Y)$ we have $\mu(A)=\mu_\eps(A)$. Therefore, by \Cref{lem:apply-hall}, the conditions for \Cref{cor:Hall-marriage} are satisfied and there exists a partial matching in $H$ covering both $\tBarc_n^\eps(X)$ and $\tBarc_n^\eps(Y)$. Such a matching corresponds exactly to an $\eps$-matching of the tagged barcodes $\tBarc_n(X)$ and $\tBarc_n(Y)$, thus we are done.
\end{proof}

It thus remains only to prove \Cref{lem:apply-hall}. For this we define $\alpha([0,s,t)) := s+t$ and we first show the following result.

\begin{lemma}\label{lem:alpha-and-interleaving}
    Let $I_1=[0,s_1,t_1), I_2=[0,s_2,t_2), I_3=[0,s_3,t_3) \in \J$ with $\alpha(I_1) \le \alpha(I_3)$ and let 
    \[
    f\colon \I^n[0,s_1,t_1) \to \I^n[0,s_2,t_2)_\eps
    \quad \text{ and } \quad
    g\colon \I^n[0,s_2,t_2) \to \I^n[0,s_3,t_3)_\eps
    \]
    be parametrized chain maps which are both non-zero. Then $c(I_1,I_2) \le \eps$ or $c(I_2,I_3) \le \eps$.
\end{lemma}

\begin{proof}
    By \Cref{lem:non-zero-morphism}, the existence of $f$ implies that $s_2 \le s_1+\eps$ and $t_2 \le t_1+\eps$. Assume that $\I^n[0,s_1,t_1)$ and $\I^n[0,s_2,t_2)$ are not $\eps$-interleaved. Then, by \Cref{lem:intervals-tepch-eps-interleaved}, $s_2 < s_1 - \eps$ or $t_2 < t_1 - \eps$. In the first case we have $\alpha(I_1) = s_1+t_1 > s_2+\eps+t_1 \ge s_2+t_2 = \alpha(I_2)$. Similarly, in the second case, we have $\alpha(I_1)=s_1+t_1 > s_1+t_2+\eps \ge s_2+t_2 = \alpha(I_2)$. Assuming that $\I^n[0,s_2,t_2)$ and $\I^n[0,s_3,t_3)$ are not $\eps$-interleaved either, with a similar computation we get $\alpha(I_2)>\alpha(I_3)$. Hence $\alpha(I_1)>\alpha(I_2)>\alpha(I_3)$, which contradicts the hypothesis.
\end{proof}

We introduce more notation.

\begin{definition}
    If $X \in \TEPCh$, $I=[0,s,t) \in \tBarc_n(X)$ and $\eps>0$, then we denote by $\eta_I^{\eps} \colon \I^n[0,s,t) \to \I^n[0,s,t)_\eps$ the canonical parametrized chain map given by the internal maps of $\I^n[0,s,t)$.  

    Let $X,Y \in \TEPCh$. For any $I=[0,s,t) \in \tBarc_n(X)$, we write $\iota^X_I\colon \I^n[0,s,t) \hookrightarrow X$ for the inclusion and $\pi^X_I\colon X \twoheadrightarrow \I^n[0,s,t)$ for the projection. Analogously, for $J \in \tBarc_n(Y)$, we write $\iota^Y_J$ and $\pi^Y_J$ for the inclusion and projection in $Y$. 
    Now let $(\phi,\psi)$ be an $\eps$-interleaving between $X$ and $Y$. For any $I \in \tBarc_n(X)$ and $J \in \tBarc_{n'}(Y)$, we write $\phi_{I,J} := (\pi^Y_J)_\eps \circ \phi \circ \iota^X_I$ and $\psi_{J,I} := (\pi^X_I)_\eps \circ \psi \circ \iota^Y_J$. 
\end{definition}

As a final preparation for the proof of \Cref{lem:apply-hall}, recall that by \Cref{lem:non-zero-parmaps} any morphism $\varphi\colon \I^n[0,s,t) \to \I^n[0,s',t')$ is uniquely determined by the scalar $\varphi(1_n^0) \in \I^n[0,s',t')^0_n = \Fi$. We denote this scalar by $\lambda(\varphi)$. This scalar has the following properties:
\begin{itemize}
    \item $\lambda(\varphi' + \varphi)= \lambda(\varphi') + \lambda(\varphi)$ for $\varphi'\colon \I^n[0,s,t) \to \I^n[0,s',t')$,
    \item $\lambda(\varphi'\circ\varphi)=\lambda(\varphi')\cdot\lambda(\varphi)$ for $\varphi'\colon \I^n[0,s',t') \to \I^n[0,s'',t'')$,
    \item $\lambda(\varphi_\eps)=\lambda(\varphi)$ for $0< \eps < t'$,
    \item $\lambda(\eta_I^{2\eps})=1$ whenever $I \in \tBarc_n^\eps(X)$ for some $X \in \TEPCh$.
\end{itemize}

\begin{proof}[Proof of \Cref{lem:apply-hall}]
    Let $(\phi,\psi)$ be an $\eps$-interleaving between $X$ and $Y$. We prove the statement for a subset $A\subseteq \tBarc_n^\eps(X)$, the case where  $A\subseteq \tBarc_n^\eps(Y)$ is analogous. Not that for any $I,I' \in A$ with $\alpha(I)\le \alpha(I')$ we have
    \begin{equation}\label{eq:same-or-different-int}
        (\pi^X_{I'})_{2\eps} \circ \psi_\eps \circ \phi \circ \iota^X_I 
        = \sum_{m\in \N} \sum_{J \in \tBarc_m(Y)} (\psi_{J,I'})_\eps \circ \phi_{I,J}
        = \sum_{J \in \tBarc_n(Y)} (\psi_{J,I'})_\eps \circ \phi_{I,J}
        = \sum_{J \in \mu_\eps(A)} (\psi_{J,I'})_\eps \circ \phi_{I,J},
    \end{equation}
    where the second equality follows from \Cref{lem:non-zero-morphism} and the third one holds since, for $J \notin \mu_\eps(A)$, at least one of the two morphisms $\psi_{J,I'}$ and $\phi_{I,J}$ must be zero by \Cref{lem:alpha-and-interleaving}. The left-hand side of \Cref{eq:same-or-different-int} is equal to $\eta_I^{2\eps}$ if $I=I'$ and equal to zero if $I\neq I'$. In particular, $\mu_\eps(A)\neq \emptyset$ when $A\neq \emptyset$.

    We order the elements of $A= \{I_1=[0,s_1,t_1), \ldots, I_r=[0,s_r,t_r) \}$ so that $\alpha(I_i) \le \alpha(I_{i'})$ for $i\le i'$, and we 
    set $\mu_\eps(A) = \{J_1=[0,s_1',t_1'), \ldots, J_q=[0,s_q',t_q') \}$. The goal is to show that $q\ge r$.

    Computing $\lambda$ on both sides of  \Cref{eq:same-or-different-int} applied to $I=I'=I_i \in A$, and using the properties of $\lambda$, we get
    \[
    1 = \lambda(\eta_{I_i}^{2\eps}) 
    =  \sum_{J \in \mu_\eps(A)} \lambda(\psi_{J,I_i}) \cdot \lambda(\phi_{I_i,J})
    = \sum_{k=1}^q \lambda(\psi_{J_k,I_i}) \cdot \lambda(\phi_{I_i,J_k}).
    \]
    If, on the other hand, we apply \Cref{eq:same-or-different-int} to $I=I_i$ and $I'=I_{i'}$, for $i<i'$, we get 
    \[
    0 = \sum_{J \in \mu_\eps(A)} \lambda(\psi_{J,I_{i'}}) \cdot \lambda(\phi_{I_i,J})
    = \sum_{k=1}^q \lambda(\psi_{J_k,I_{i'}}) \cdot \lambda(\phi_{I_i,J_k}).
    \]
    We can put these equations into the form of a matrix equation, which yields
    \begin{equation*}
        \begin{pmatrix}
            \lambda(\psi_{J_1,I_1}) & \cdots &\lambda(\psi_{J_q,I_1}) \\
            \vdots &\ddots & \vdots \\
            \lambda(\psi_{J_1,I_r}) & \cdots & \lambda(\psi_{J_q,I_r})
        \end{pmatrix}
        \begin{pmatrix}
            \lambda(\phi_{I_1,J_1}) & \cdots &\lambda(\phi_{I_r,J_1}) \\
            \vdots &\ddots & \vdots \\
            \lambda(\phi_{I_1,J_q}) & \cdots & \lambda(\phi_{I_r,J_q})
        \end{pmatrix}
        =
        \begin{pmatrix}
            1 & * &\cdots &* \\
            0 & 1 &\cdots & * \\
            \vdots & \vdots & \ddots &\vdots \\
            0 & 0 & \cdots &1
        \end{pmatrix}.
    \end{equation*}
    Note that for $i>i'$ we make no statement, which is why there are unknown entries on the upper right triangle of the matrix on the right-hand side. This matrix has rank $r$ and the product of the two matrices on the left-hand side has rank at most $q$, therefore $q \ge r$.
\end{proof}

\subsection{Parametrizing chain complexes}\label{sec:parametrize-ch}

We are now going to describe two methods of transforming a chain complex into a factored chain complex. This requires additional input. In the first place, we need to use based chain complexes. A \textbf{based chain complex} is a chain complex $C_\bullet$ together with a basis $\B_k$ for every $C_k$.
The other thing we need is to equip pairs of basis elements of adjacent degrees with weights. Formally, this is dealt with in the following definition.

\begin{definition}\label{def:weighted-based-chain-complex}
     A \textbf{weighted based chain complex} is a based chain complex $C_\bullet$, with bases $\B_k$, additionally equipped with functions $\operatorname{w}\colon \B_k \times \B_{k-1} \to [0,\infty)$ for all $k$. The values $\operatorname{w}(a,b)$, for $a \in \B_k$ and $b \in \B_{k-1}$, are called \textbf{weights}. A weighted based chain complex $(C_\bullet,\B_\bullet,\operatorname{w})$ is called \textbf{ordered}, if for every $t \in [0,\infty)$, a total order is given on the set
     \[
     \operatorname{w}^{-1}(t) := \bigcup_k \left\{ (a,b) \in \B_k \times \B_{k-1} \mid \operatorname{w}(a,b)=t \right\}.
     \]
     A weighted based chain complex is called \textbf{generic} if all the weights are non-zero and pairwise different.
     We denote by $\wbCh$ (resp. $\wbChgen$) the class of (ordered) weighted based chain complexes.
\end{definition}

A generic weighted based chain complex is in particular ordered, since the sets $\operatorname{w}^{-1}(t)$ have cardinality zero or one for all values of $t$.

The main idea for constructing a factored chain complex from an ordered weighted based chain complex is the following: Repeatedly simplify the chain complex with respect to the pair of minimal weight (breaking the tie by using the ordering, if needed) until we are left with a chain complex whose differential is zero in all degrees. The difference between the two constructions we present is how we build a factored chain complex from this sequence of simplifications. In the first construction, represented by the letter $X$, the resulting functor $X\colon [0,\infty) \to \Ch$ corresponds closely to this sequence of simplifications, in the sense that the chain complexes $X^t$ are given by the result of applying some of the simplifications, and the internal chain maps are given by compositions of the quotient maps. The parameter $t\in [0,\infty)$ can be thought of as time, in which case we can interpret the construction as applying one simplification after the other, always waiting for a time equal to the weight of a pair before simplifying that pair. 

We present the first construction only for the case of strictly positive weights. Pairs of weight zero could be allowed, resulting in multiple simplifications happening at the same time. However, we leave it to the reader to fill in the details for this, as it adds some technicalities and hinders the readability of the exposition.

\begin{construction}\label{def:assign-parchain-to-vectorfield}
    Given an ordered weighted based chain complex $C_\bullet = (C_\bullet,\B_\bullet,\operatorname{w})$ with strictly positive weights, we define the factored chain complex $X=X(C_\bullet)$ as follows.
    \begin{enumerate}[(1)]
        \item\label{item:assign-1} Let $t_0 := 0$ and $X^{t_0} := C_\bullet$. 

        \item\label{item:assign-2} Assume that we have already defined a sequence of numbers $0=t_0 < t_1 < \ldots < t_r < \infty$ and a functor $X\colon [0,t_r] \to \Ch$. Further assume that $X^{t_r}$ is an ordered weighted based chain complex, i.e.~we have a basis $\B^r_k$ for each $X^{t_r}_k$ and totally ordered weights $\operatorname{w}(a,b)$ between basis elements of adjacent degrees. 

        If all the differentials in the chain complex $X^{t_r}$ are zero, go to step $(\ref{item:assign-3})$.

        Otherwise, consider the set 
        \[
        A := \bigcup_{i\ge 1} \{ (a',b') \in \B^r_i \times \B^r_{i-1} \mid \langle \partial a',b' \rangle \neq 0 \} 
        \]
        and let $(a,b) \in A$ be the unique pair of minimum weight, or, if the minimal weight is not unique, the pair of minimal weight which is minimal with respect to the order. Define 
        \[
        t_{r+1} := t_r + \operatorname{w}(a,b).
        \]
        Define $X^{t_{r+1}}$ to be the chain complex resulting from \Cref{lem:chain-complex-simplification}, when applied to the chain complex $X^{t_r}$ and the element $a \in \B^{r}_k$. 
        We endow the vector spaces $X^{t_{r+1}}_i$ with the bases $\B^{r+1}_i$ coming from \Cref{lem:chain-complex-simplification}, when applied to the chain complex $X^{t_r}$, the bases $\B^r_i$, and the elements $a \in \B^{r}_k$, $b \in \B^r_{k-1}$. Consider basis elements $a' \in \B^r_i$ and $b' \in \B^r_{i-1}$, both different from $a$ and $b$. We define $\operatorname{w}([a'],[b']) := \operatorname{w}(a',b')$, thus the weights on the bases $\B^r_\bullet$ induce weights on the bases $\B^{r+1}_\bullet$. The same holds for the orderings of pairs with the same weight.
        
        For $t_r < t < t_{r+1}$, define $X^t := X^{t_r}$ and for $t_r \le s \le t < t_{r+1}$, let $X^{s\le t} := \one$ and define $X^{s\le t_{r+1}}$ to be the epimorphism from \Cref{lem:chain-complex-simplification}. The remaining internal chain maps are defined by composition, thus we have extended $X$ to a functor $X\colon [0,t_{r+1}] \to \Ch$. 
        Update $r$ to $r+1$ and repeat step $(\ref{item:assign-2})$.
        
        \item\label{item:assign-3} If all the differentials in the chain complex $X^{t_r}$ are zero, then we define $X^t := X^{t_r}$ for all $t_r < t < \infty$ and $X^{t \le s} = \one$ for all $t_r \le t \le s < \infty$. The remaining internal maps are defined by composition, so we have extended $X$ to a functor $X\colon [0,\infty) \to \Ch$.
    \end{enumerate}
\end{construction}

\begin{proposition}\label{prop:barcode-construction1}
    If $C_\bullet = (C_\bullet,\B_\bullet,\operatorname{w})$ is an ordered weighted based chain complex with strictly positive weights, then \Cref{def:assign-parchain-to-vectorfield} assigns to it a factored chain complex $X(C_\bullet)$. Moreover, if $(a_1,b_1) \in \B_{n_1}\times \B_{n_1-1},\ldots,(a_r,b_r) \in \B_{n_r}\times \B_{n_r-1}$ are the pairs that get simplified in the construction of $X(C_\bullet)$,
    then 
    \[
    X(C_\bullet) \cong \bigoplus_{i=1}^r \I^{n_i}[0,t_i,t_i) \oplus \bigoplus_{k=0}^\infty (\I^k[0,0,\infty))^{\beta_k},
    \]
    where $\beta_k=\dim(H_k(C_\bullet))$. The numbers $t_i$ are given by
    \[
    t_i = \operatorname{w}(a_1,b_1) + \cdots + \operatorname{w}(a_i,b_i).
    \]
\end{proposition}

\begin{proof}
    First note that, since $C_\bullet$ is ordered, the set $A$ contains a unique minimal pair in the first iteration. Since then we only get rid of one pair in each iteration and do not change any of the other weights, this holds also for later iterations of step (\ref{item:assign-2}).

    We need to check that the conditions from the beginning of step (\ref{item:assign-2}) are satisfied when we go to another iteration of step (\ref{item:assign-2}). They are satisfied for the first iteration of step (\ref{item:assign-2}), when $r=0$, due to the definitions in step (\ref{item:assign-1}). If $r>0$, the conditions are satisfied at the beginning of step (\ref{item:assign-2}) and the differentials in $X^{t_r}$ are not all zero, then it follows from \Cref{lem:chain-complex-simplification} that the conditions are also satisfied for the next iteration.

    We also need to show that the definition stops at some point, i.e.~that we do not get an infinite loop of repeating step (\ref{item:assign-2}). This holds because with each iteration we reduce the total dimension of the chain complex $X^{t_r}$ and since $X^{t_0}=C_\bullet$ has finite total dimension, we can do this only a finite number of times.

    Once we get to step (\ref{item:assign-3}), we get a parametrized chain complex $X=X(C_\bullet) \in \PCh$. It remains to check that $X$ is tame and epimorphic. Indeed $X$ is epimorphic since the only internal maps that are not isomorphisms are quotient maps of the form given by \Cref{lem:chain-complex-simplification}, or compositions of such maps. Also, by construction, $X$ is right-constant everywhere and fails to be left-constant exactly at the points $t_1,\ldots,t_r$, where $r$ is the number of times we have repeated step (\ref{item:assign-2}). There are only finitely many of these points, since we have shown before that we repeat step (\ref{item:assign-2}) only a finite number of times. It follows that $X$ is tame and hence a factored chain complex. 
    
    It remains to prove the claim about the decomposition of $X$. For any $i=1,\ldots,r$, picking the $i$-th pair $(a_i, b_i)$ yields a short exact sequence
    \[
    0 \longrightarrow D^{n_i} \overset{\iota}{\longrightarrow} X^{t_{i-1}} \overset{q}{\longrightarrow} X^{t_i} \longrightarrow 0,
    \]
    where $\iota$ is the unique chain map that sends $D^{n_i}_{n_i} \ni 1_{n_i} \mapsto a_i \in C_{n_i}$ and $q=q_{a_i}$ is the quotient map from \Cref{lem:chain-complex-simplification}. This sequence splits by \Cref{lem:chain-complex-ses-splits}, where the splitting depends on $b_i$. Thus we have $X^{t_{i-1}} = X^{t_i} \oplus D^{n_i} = X^{t_i} \oplus \ker X^{t_{i-1} \le t_i}$, and the result follows by induction.
\end{proof}

Now we present our second construction for assigning a factored chain complex to an ordered weighted based chain complex, represented by the letter $Y$. It is formulated without further conditions and has the advantage that when a pair of small weight gets simplified, this always yields a small tagged interval in the tagged barcode, even if this pair gets simplified at a later stage of the construction. In particular, pairs of weight zero, if they exist, do not contribute to the tagged barcode.

\begin{construction}\label{def:assign-parchain-to-vectorfield2}
    Given an ordered weighted based chain complex $C_\bullet = (C_\bullet,\B_\bullet,\operatorname{w})$, we define the factored chain complex $Y=Y(C_\bullet)$ as follows.
    \begin{enumerate}[(1)]
        \item Let $Z_\bullet := C_\bullet \in \wbCh$.
        
        \item Initialize  $Y$ by setting it to be equal to the zero factored chain complex discretized by $t_0=0$ and $t_1=+\infty$ with only zero vector spaces and null linear maps.

        \item Find the pair $(a,b) \in \B_k \times \B_{k-1}$ in $Z_\bullet$, among all $k$, that has the smallest weight $\operatorname{w}(a,b)$ (in case of multiple pairs sharing the smallest weight, pick the minimal one with respect to the order) and  satisfies the  condition: $\langle \partial a,b \rangle \neq 0$. If we cannot find such a pair, then all boundaries are zero in $Z_\bullet$ and in that case, we go to Step (\ref{step-final}). \label{step-iterate}
        
        \item Update $Y$ by summing to it  the interval functor $\I^n[0,t,t)$ in $\TEPCh$, where $t=\operatorname{w}(a,b)$ and $(a,b) \in \B_n \times \B_{n-1}$. If $\operatorname{w}(a,b)=0$, leave $Y$ as it is.

        \item Update $Z_\bullet$ by quotienting it by the $n$-disk of $\Ch$ generated by $a$ and $\partial a$ as in \Cref{lem:chain-complex-simplification}. Update the bases according to \Cref{lem:chain-complex-simplification}, applied to the elements $a \in \B_n$ and $b \in \B_{n-1}$. Given two basis elements $[a']$ and $[b']$ in the updated bases, in adjacent degrees, we define $\operatorname{w}([a'],[b']) := \operatorname{w}(a',b')$, thereby inducing weights on the updated bases. The new orderings of pairs with the same weight are defined analogously.

        \item Repeat from Step (\ref{step-iterate}).

        \item Add as many interval functors $\I^n[0,0,\infty)$ to $Y$ as needed to reach the dimension of $Z_n$ for each $n\ge 0$. \label{step-final}
    \end{enumerate}
\end{construction}

It immediately follows from the fact that $Y(C_\bullet)$ is constructed as the direct sum of factored chain complexes that \Cref{def:assign-parchain-to-vectorfield2} defines an assignment $Y\colon \wbChgen \to \TEPCh$ with the following property.

\begin{proposition}\label{prop:barcode-construction2}
    If $C_\bullet = (C_\bullet,\B_\bullet,\operatorname{w})$ is an ordered weighted based chain complex and $(a_1,b_1) \in \B_{n_1}\times \B_{n_1-1},\ldots,(a_r,b_r) \in \B_{n_r}\times \B_{n_r-1}$ are the pairs that get simplified in the construction of $Y=Y(C_\bullet)$, then 
    \[
    Y(C_\bullet) \cong \bigoplus_{i=1}^r \I^{n_i}[0,s_i,s_i) \oplus \bigoplus_k (\I^k[0,0,\infty))^{\beta_k},
    \]
    where $\beta_k=\dim(H_k(M;\Fi))$ and $s_i=\operatorname{w}(a_i,b_i)$.
\end{proposition}

Comparing \Cref{def:assign-parchain-to-vectorfield} and \Cref{def:assign-parchain-to-vectorfield2}, we see that for any $C_\bullet \in \wbChgen$, the same pairs get simplified in the construction of $X(C_\bullet)$ and $Y(C_\bullet)$, but at different times. Note that the sequence of simplification times is increasing by construction for $X$ but not necessarily for $Y$. The barcode resulting from the $Y$ construction always has bars of length given by the weights of the pairs, while the lenghts of the bars in $X$ accumulate. Thus, if we simplify a pair of small weight, this will always result in a small change of the barcode of $Y$, which is not necessarily true for $X$ (compare also the different upper bounds given in \Cref{lem:stability-X-Y}).

Based on \Cref{lem:chain-complex-simplification}, the explicit computation of the tagged barcode from \Cref{def:assign-parchain-to-vectorfield2} consists of Gaussian elimination and deletion of rows and columns. We include the pseudocode for completeness, see \Cref{alg:tagged-barcode-computation}. The pseudocode for computing the tagged barcode from \Cref{def:assign-parchain-to-vectorfield} would look very similar, with the only difference being that we would initialize the variable $t$ in the beginning by setting it to zero and then update it in every step of the while loop by adding $\operatorname{w}(b_j,b_i)$ to it in line~\ref{algo:tagged-barcode-t}. 

\begin{algorithm}
\caption{Tagged barcode computation (in $\Z/2\Z$)}\label{alg:tagged-barcode-computation}
\begin{algorithmic}[1]
\Require $(C_\bullet, \B_\bullet, \operatorname{w}_\bullet) \in \wbChgen$
\Ensure Tagged barcode of $Y(C_\bullet, \B_\bullet, \operatorname{w}_\bullet)$ and sequence of simplified pairs
\State $\tBarc := \emptyset$ and $\operatorname{Pairs} := \emptyset$
\State $d :=$ total boundary matrix of $C_\bullet$ w.r.t. bases $\B_\bullet$
\While{a non-zero column $j$ exists in $d$}
\State $i,j :=$ pair with $d[i,j]\neq 0$ and $t:= \operatorname{w}(b_j,b_i)$ minimal \label{algo:tagged-barcode-t}
\State Append $(n,[0,t,t))$ to $\tBarc$ and $(b_j,b_i)$ to $\operatorname{Pairs}$, with $n= \deg(b_j)$
\For{$k$ with $d[k,j] \neq 0$}
\State Add the $i$-th row to the $k$-th row
\EndFor
\State Delete rows $i,j$ and columns $i,j$ from $d$
\EndWhile
\For{all indices $j$ of remaining columns in $d$}
\State Append $(n,[0,0,+\infty))$ to $\tBarc$ with $n=\deg(b_j)$
\EndFor
\State Return $\tBarc$ and $\operatorname{Pairs}$
\end{algorithmic}
\end{algorithm}

\subsection{Stability}\label{sec:stability}

We want to give a stability result that bounds the interleaving distance between the factored chain complexes assigned by \Cref{def:assign-parchain-to-vectorfield} and \Cref{def:assign-parchain-to-vectorfield2} to sufficiently similar weighted based chain complexes. To do this, we start by introducing reparametrizations, since in the case of \Cref{def:assign-parchain-to-vectorfield}, the two factored chain complexes will be related by a reparametrization.

\begin{definition}
    A map $\alpha\colon [0,\infty) \to [0,\infty)$ is called a \textbf{reparametrization} if it is a homeomorphism (which implies that $\alpha(0)=0$ and that $\alpha$ is order-preserving) and $\eps_\alpha := \sup_{t \in [0,\infty)}|\alpha(t)-t| < \infty$. If $X\colon [0,\infty) \to \Ch$ is a parametrized chain complex, then we call $X \circ \alpha\colon [0,\infty) \to \Ch$ the \textbf{reparametrization of $X$ by $\alpha$}. This is again a parametrized chain complex, with $(X\circ \alpha)^t=X^{\alpha(t)}$ and $(X\circ \alpha)^{s\le t} = X^{\alpha(s) \le \alpha(t)}$.
\end{definition}

\begin{lemma}
    The inverse $\alpha^{-1}$ of a reparametrization is again a reparametrization and $\eps_{\alpha^{-1}}=\eps_\alpha$.
\end{lemma}

\begin{proof}
    The inverse of a homeomorphism is again a homeomorphism, so the only thing to show is that $\eps_{\alpha^{-1}}=\eps_\alpha$. This follows because for any $t \in [0,\infty)$, we have $|\alpha^{-1}(\alpha(t))- \alpha(t)| = |t-\alpha(t)| = |\alpha(t)-t|$. 
\end{proof}

\begin{proposition}\label{prop:reparam-interleaved}
    Given $\alpha$ and $X$ as in the definition above, we have that $X$ and $X \circ \alpha$ are $\eps_\alpha$-interleaved.
\end{proposition}

\begin{proof}
    Let $\eps := \eps_\alpha$ and $Y:=X\circ \alpha$. The internal maps of $Y$ are given by $Y^{s\le t} = X^{\alpha(s)\le \alpha(t)}$. We construct an $\eps$-interleaving between $X$ and $Y$ as follows.

    Since $\eps=\eps_\alpha=\eps_{\alpha^{-1}}$, we have $\alpha(t) \le t+\eps$ and $t\le \alpha(t+\eps)$ for all $t \in [0,\infty)$, where the second inequality results from applying $\alpha$, which is order preserving, to the inequality $\alpha^{-1}(t)\le t+\eps$. Defining a morphism $\psi\colon Y \to X_\eps$ means defining, for every $t\ge 0$, a map from $Y^t$ to $X^{t+\eps}$ commuting with the internal maps. Since $Y^t=X^{\alpha(t)}$ and $\alpha(t)\le t+\eps$, we can take $\psi^t$ to be the internal map of $X$, i.e.~$\psi^t:= X^{\alpha(t)\le t+\eps}$. Similarly we get $\phi\colon X \to Y_\eps$ by $\phi^t := X^{t\le \alpha(t+\eps)}$. One can check that $(\phi,\psi)$ is an interleaving pair for $X,Y$.
\end{proof}

We are now ready to prove our stability result for generic weighted based complexes. The following lemma gives an upper bound on the interleaving distance between the parametrizations of two weighted based chain complexes which are isomorphic and whose weights are similar enough. In the next section it will be used in the proof of \Cref{thm:local-stability} and \Cref{thm:comb-approx}.
 
\begin{lemma}\label{lem:stability-X-Y}
    Assume that we are given generic weighted based chain complexes $(C_\bullet,\B_\bullet,\operatorname{w}),(C'_\bullet,\B'_\bullet,\operatorname{w}')$ and bijections $\varphi\colon \B_k \to\B'_k$, inducing an isomorphism of chain complexes $C_\bullet \cong C'_\bullet$. 
    Assume further that the bijections $\varphi$ respect the ordering of the weights in $\B_\bullet$ and $\B'_\bullet$, i.e.~for all $(a,b) \in \B_k \times \B_{k-1}$ and $(c,d) \in \B_j \times \B_{j-1}$ we have $\operatorname{w}(a,b)< \operatorname{w}(c,d)$ if and only if $\operatorname{w}'(\varphi(a),\varphi(b))< \operatorname{w}'(\varphi(c),\varphi(d))$. If in the construction of $X(C_\bullet)$ and $Y(C_\bullet)$ the pairs 
    \[
    (a_1,b_1), \quad \ldots, \quad (a_n,b_n),
    \]
    get simplified in this order, then in the construction of $X(C'_\bullet)$ and $Y(C'_\bullet)$ the pairs 
    \[
    (\varphi(a_1),\varphi(b_1)), \quad \ldots, \quad (\varphi(a_n),\varphi(b_n)),
    \]
    get simplified in this order. Setting $d_\varphi(C_\bullet,C'_\bullet):= \max \{ |\operatorname{w}(a,b)-\operatorname{w}'(\varphi(a),\varphi(b))| \mid a \in \B_k,\ b \in \B_{k-1} \}$, the following inequalities hold:
    \begin{align*}
        d_I(X(C_\bullet),X(C'_\bullet)) &\le n d_\varphi(C_\bullet,C'_\bullet), \\
        d_I(Y(C_\bullet),Y(C'_\bullet)) &\le d_\varphi(C_\bullet,C'_\bullet).
    \end{align*}
\end{lemma}

\begin{proof}
    Since the orders of the weights in $C_\bullet$ and $C'_\bullet$ correspond via $\varphi$, the same is true in particular for the pair of basis elements with the smallest weight. This also does not change after any of the simplifications, since we update the weights by simply deleting some basis elements and not changing the other weights. It follows that whenever we simplify a pair $(a,b)$ in the construction of $X(C_\bullet)$, we simplify the corresponding pair $(\varphi(a),\varphi(b))$ in the construction of $X(C'_\bullet)$. Since we simplify the same pairs in the construction of $X$ and $Y$, the same is true also for $Y$.

    To prove the inequalities we use two different strategies. We begin with the second inequality, for which we apply \Cref{thm:iso-thm-tepch}, so that we can use the bottleneck distance instead of the interleaving distance. From \Cref{def:assign-parchain-to-vectorfield2} it is straightforward to see that the $k$-th tagged barcode of $Y(C_\bullet)$ consists of one copy of $[0,t,t)$ for any pair $(a,b) \in \B_k \times \B_{k-1}$ that gets simplified, where $t=\operatorname{w}(a,b)$. We can match each such tagged interval with the corresponding tagged interval $[0,t',t')$ in the $k$-th tagged barcode of $Y(C'_\bullet)$, where $t'=\operatorname{w}'(\varphi(a),\varphi(b))$. This yields a matching whose cost is upper bounded by $d_\varphi(C_\bullet,C'_\bullet)$.

    To show the first inequality we will use \Cref{prop:reparam-interleaved}, so we want to construct a reparametrization $\alpha\colon [0,\infty) \to [0,\infty)$ such that $X(C'_\bullet) \circ \alpha \cong X(C_\bullet)$ and $\eps_\alpha \le n d_\varphi(C_\bullet,C'_\bullet)$. 
    
    For $i=1,\ldots,n$, let $\operatorname{w}_i := \operatorname{w}(a_i,b_i)$ and $\operatorname{w}_i' := \operatorname{w}'(\varphi(a_i),\varphi(b_i))$. 
    Next we define 
    \[
    t_i := \operatorname{w}_1 + \cdots +\operatorname{w}_i \quad \text{and} \quad t_i' := \operatorname{w}_1' + \cdots +\operatorname{w}_i'.
    \]
    We construct the reparametrization $\alpha\colon [0,\infty) \to [0,\infty)$ as follows. We define $\alpha(0):=0$ and $\alpha(t_i):=t_i'$ for all $i$. In between $t_i$ and $t_{i+1}$ we interpolate linearly. Above $t_n$ we define $\alpha(t_r+s):=t_r'+s$ for all $s\ge 0$. One can check that this indeed defines a reparametrization and that the supremum of $|\alpha(t)-t|$ is attained at one of the $t_i$ and is bounded by $n d_\varphi(C_\bullet,C'_\bullet)$, since for all $i=1,\ldots,n$ we have
    \[
    |t_i-t_i'| \le |\operatorname{w}_1-\operatorname{w}_1'| +\cdots+ |\operatorname{w}_i-\operatorname{w}_i'| \le i d_\varphi(C_\bullet,C'_\bullet) \le n d_\varphi(C_\bullet,C'_\bullet).
    \]
    It remains to check that $X(C'_\bullet) \cong X(C_\bullet) \circ \alpha$. 
    By assumption the bijections $\varphi\colon \B_k \to \B'_k$ induce an isomorphism $X(C_\bullet)^0 \cong X(C'_\bullet)^0$. The isomorphism type of $X(C_\bullet)^t$ changes when we do the simplifications, which is exactly at the time $t_i$. For $X(C'_\bullet)$ it is at the time $t_i'$ instead. Since we do the same simplification moves in the same order, the complexes $X(C_\bullet)^t$ and $X(C'_\bullet)^{\alpha(t)}$ remain isomorphic for all $t\in [0,\infty)$.
\end{proof}

\section{The tagged barcode of a gradient-like Morse-Smale vector field}\label{sec:barcode-of-vector-field}

Let $M$ be a closed Riemannian manifold of dimension $n$ and denote by $d_M$ the distance on $M$ induced by the Riemannian structure.
The goal of this section is to assign a factored chain complex to a gradient-like Morse-Smale vector field in general position (defined below) defined on $M$. For this, we apply either of our two constructions from \Cref{sec:parametrize-ch} to the Morse complex $\MC_\bullet(v)$. This yields a sequence of algebraic simplifications, depending on the distances between the singular points of $v$, resulting in a factored chain complex. 

\subsection{General position}\label{sec:gen-pos} 

In our construction of a factored chain complex from a weighted based chain complex, we need that the weights are ordered. The simplest way to satisfy this condition is to demand that they are pairwise different, i.e. the weighted based chain complex is generic. We thus formulate an analogous condition for the distances between the singular points of the vector fields that we study, i.e. that no two of these distances are the same. We show that this condition is not too restricting, in the sense that the set of vector fields satisfying it forms an open and dense subset of all gradient-like Morse-Smale vector fields. 

\begin{definition}
    Given a finite set of points $S \subseteq M$, we define
    \[
    \xi(S) := \min \{|d_M(a,b)-d_M(a',b')| \mid a,b,a',b' \in S \text{ and } a\neq b,\ a'\neq b',\ \{a,b\}\neq \{a',b'\}\}.
    \]
    If $\xi(S)>0$, then we say that the points in $S$ are \textbf{in general position}. In plain words, this means that no two pairs of points of $S$ have the same distance. A gradient-like Morse-Smale vector field is said to be \textbf{in general position} if its set of singular points is in general position. We denote by $\X_{gMS+}(M)$ the set of gradient-like Morse-Smale vector fields in general position.
\end{definition}

\begin{proposition}\label{prop:general-position-open-dense}
    The set $\X_{gMS+}(M)$ is open and dense in $\X^1_{gMS}(M)$.
\end{proposition}

The part about openness is fairly straightforward, but for denseness we will first prove two lemmas. The first lemma states that if we are given some points on $M$, we can put them into general position by moving them by an arbitrarily small amount.

\begin{lemma}\label{lem:gen-pos-dense-1}
    Given distinct points $p_1,\ldots,p_n \in M$ and $\delta>0$, there exist distinct points $q_1,\ldots,q_n \in M$ such that $d_M(p_i,q_i)<\delta$ for all $i$ and the points $q_1,\ldots,q_n$ are in general position.
\end{lemma}

\begin{proof}
    For this proof, we call each instance of four points $p_i,p_j,p_r,p_s$ with $d_M(p_i,p_j)=d_M(p_r,p_s)$, where $p_i\neq p_j$, $p_r\neq p_s$ and $\{p_i,p_j\}\neq\{p_r,p_s\}$, a problem. We show that we can get rid of any given problem by moving just one point by an arbitrarily small amount. In order to see this, note that there are two essentially different cases. The first case is the one where all four points are distinct. In the second case we have just three distinct points, one of which appears on both sides of the equation. However, in both cases there is a point which appears only once. Without loss of generality let $p_i$ be that point. Then, consider a length minimizing geodesic $\gamma$ from $p_i$ to $p_j$. If we move $p_i$ along $\gamma$ towards $p_j$, we reduce $d_M(p_i,p_j)$ but do not change $d_M(p_r,p_s)$. If the amount we move $p_i$ is small enough, then we will also not introduce any new problems. Repeating this procedure finitely many times, we can get rid of all the problems one by one. For each point $p_i$ we denote the position where it ended up by $q_i$, which yields points in general position.
\end{proof}

Let us recall the definitions of the Whitney $C^0$- and $C^1$-topologies on $\X(M)$. We give the definitions in the same spirit as they are given in \cite{palis2012geometric}, however we define the $C^1$-topology on $\X(M)$ directly, instead of using an embedding $M \hookrightarrow \R^s$ and then topologizing $\X(M)$ as a closed subspace of $C^1(M,\R^s)$. The resulting topology is the same.

We start by saying that, for any matrix $A \in \R^{p\times m}$, we write $\norm{A}$ for the Frobenius norm of $A$, i.e. the square root of the sum of all squared entries of $A$. In the special case where $p=1$ or $m=1$, this equals the standard Euclidean norm on $R^m$ or $\R^p$, respectively. Recall that the Frobenius matrix norms are submultiplicative, i.e. for any $A \in \R^{k\times p}$ and $B \in \R^{p\times m}$ we have $\norm{AB} \le \norm{A} \cdot \norm{B}$.

For any $r>0$, we denote by $B_r(0) \subseteq \R^m$ the ball with radius $r$ centered around the origin. We will consider smooth functions $F\colon B_4(0) \to \R^p$, for arbitrary values of $p$. This general setup allows us to define, at the same time, the $C^0$-norms and $C^1$-norms of functions $f\colon B_4(0) \to \R$, vector fields $\vec{v}\colon B_4(0) \to \R^m$, and their respective differentials. Note that the differential of a vector field on $B_4(0)$ is a function $D\vec{v}\colon B_4(0) \to \R^{m\times m}$, and we identify $\R^{m\times m} \cong \R^{m^2}$ in order to include also this in our definition. Hence, for any smooth function $F\colon B_4(0) \to \R^p$ we define
\begin{align*}
    \norm{F}_0 &:= \sup_{x \in B_2(0)} \norm{F(x)}, \\
    \norm{F}_1 &:= \max \{ \norm{F}_0, \norm{DF}_0 \}.
\end{align*}
Note that, even though we are considering functions defined on $B_4(0)$, in the definition of these norms we only look at the values on $B_2(0)$. Thus they are technically speaking not norms, but defining them in this way will be convenient for us when defining the $C^0$- and $C^1$-norms for vector fields on manifolds.

Let us cover $M$ by sets $V_1,\ldots,V_k$ such that each of the sets is contained in a larger set $V_i \subseteq U_i$ on which we are given a chart $\sigma_i\colon U_i \to B_4(0)$ with $\sigma_i(V_i)=B_2(0)$. Given these charts, any vector field $v\in \X(M)$ induces Euclidean vector fields $\vec{v}_i\colon B_4(0) \to \R^m$ by
\begin{center}
    \begin{tikzcd}
        B_4(0) \ar[rrrr,bend right=20,"=:\vec{v}_i"] \ar[r,"\sigma_i^{-1}"] &U_i \ar[r,"v|_{U_i}"] &TU_i \ar[r,"\cong"] &U_i \times \R^m \ar[r,"\operatorname{pr}_2"] &\R^m,
    \end{tikzcd}
\end{center}
where the diffeomorphism $TU_i \cong U_i \times \R^m$ is induced by the chart $\sigma_i$.
Explicitly, if we denote by $(x_i^1,\ldots,x_i^m)$ the local coordinates on $U_i$ induced by $\sigma_i$, then 
${\vec{v}_i(x) = \begin{pmatrix}
    v_i^1(\sigma_i^{-1}(x))\\
    \vdots \\
    v_i^m(\sigma_i^{-1}(x))
\end{pmatrix}}$, 
where the functions $v_i^j\colon U_i \to \R$ are defined through the equation $v(p) = \sum_j v_i^j(p) \cdot \frac{\partial}{\partial x_i^j}|_p$.
We define the $C^0$-norm and the $C^1$-norm of $v\in \X(M)$ by  
\[
\norm{v}_0 := \max_{i=1,\ldots,k} \norm{\vec{v}_i}_0, \quad \quad
\norm{v}_1 := \max_{i=1,\ldots,k} \norm{\vec{v}_i}_1.
\]

One can check that this defines a norm on $\X(M)$ and that the topology induced from this norm does not depend on the choices that are involved. We now state and prove the second lemma needed for \Cref{prop:general-position-open-dense}. It will be used to show that it is possible to $C^1$-perturb a vector field and move a singular point to an arbitrary point in a small neighbourhood.

\begin{lemma}\label{lem:gen-pos-dense-2}
    Let $\vec{v}\colon B_4(0) \to \R^m$ be a smooth map and let $\eps>0$. Then there exists $\delta>0$ such that for all $y \in B_\delta(0)$ there exists a diffeomorphism $\varphi=\varphi_y\colon B_4(0)\to B_4(0)$ that satisfies
    \begin{itemize}
        \item $\varphi(x)=x$ for $\norm{x}\ge 1$,
        \item $\varphi(y)=0$,
        \item $\norm{\vec{v}-\vec{v}\circ \varphi}_1 < \eps$.
    \end{itemize}
\end{lemma}

\begin{proof}
    We start by picking a smooth function $\rho\colon \R^m \to \R$ with the following properties:
    \begin{itemize}
        \item $\rho(0)=1$,
        \item $0\le \rho(x) \le 1$ for $\norm{x}\le \frac{1}{2}$,
        \item $\rho(x)=0$ for $\norm{x}\ge \frac{1}{2}$.
    \end{itemize}
    We can explicitly choose $\delta$ to be 
    \[
    \delta := \min \left\{ \frac{1}{2}, \frac{\eps}{L_1}, \frac{1}{\norm{\rho}_1}, \frac{\eps}{L_2+\norm{\vec{v}}_1\cdot \norm{\rho}_1} \right\},
    \]
    where the constants $L_1=L_1(\vec{v})$ and $L_2=L_2(\vec{v})$ are the Lipschitz constants of $\vec{v}$ and $D\vec{v}$, i.e.
    \[
    L_1 := \sup \left\{ \frac{\norm{\vec{v}(x)-\vec{v}(y)}}{\norm{x-y}} \mid x \neq y \in B_2(0) \right\}
    \quad \text{and} \quad
    L_2 := \sup \left\{ \frac{\norm{D\vec{v}(x)-D\vec{v}(y)}}{\norm{x-y}} \mid x \neq y \in B_2(0) \right\}.
    \]
    The reason for this choice will become clear later. The values of $L_1$ and $L_2$ are bounded by the $C^1$-norms of $\vec{v}$ and $D\vec{v}$, respectively. In particular, they are finite, as the closure of $B_2(0)$ is compact. We now define the map $\varphi=\varphi_y\colon B_4(0) \to B_4(0)$ by $\varphi(x) := x - \rho(x-y)\cdot y$. For  $\norm{y}< \delta$ and $x \in B_2(0)$, we have
    \begin{align*}
        \norm{D\varphi(x)} = \norm{\one - D(\rho(x-y)\cdot y)} = \norm{\one - \nabla\rho(x-y)\cdot y^T} &\ge \norm{\one} - \norm{\nabla\rho(x-y)\cdot y^T} \\
        &\ge 1 - \norm{\nabla\rho(x-y)} \cdot \norm{y^T} \\ 
        &\ge 1 - \norm{\rho}_1\cdot \norm{y} > 0.
    \end{align*}
    For $\norm{x}\ge 1$ we have $\varphi(x)=x$ since for $\norm{y}\le \frac{1}{2}$ and $\norm{x}\ge 1$ we have $\norm{x-y}\ge \frac{1}{2}$ and thus $\rho(x-y)=0$. Therefore, $\varphi$ is a diffeomorphism.
    
    For any $x \in B_2(0)$ we get
    \[
    \norm{\vec{v}(x)-\vec{v}(\varphi(x))} \le L_1 \norm{x-\varphi(x)} = L_1 \norm{\rho(x-y)\cdot y} \le L_1 \norm{y}
    < L_1 \delta \le \eps.
    \]
    Also we have
    \begin{align*}
        \norm{D\vec{v}(x) - D(\vec{v}\circ\varphi)(x)} &= \norm{D\vec{v}(x)-D\vec{v}(\varphi(x)) \cdot D\varphi(x)} \\ 
        &\le \norm{D\vec{v}(x) - D\vec{v}(\varphi(x))} + \norm{D\vec{v}(\varphi(x))-D\vec{v}(\varphi(x))\cdot D\varphi(x)} \\
        &\le L_2 \norm{x-\varphi(x)} + \norm{D\vec{v}(\varphi(x))} \cdot \norm{\one - D\varphi(x)} \\
        &\le L_2 \norm{y} + \norm{\vec{v}}_1 \cdot \norm{(\nabla \rho(x-y)) \cdot y^T} \\
        &\le L_2 \norm{y} + \norm{\vec{v}}_1 \cdot \norm{\rho}_1 \cdot \norm{y} \\
        &= (L_2 + \norm{\vec{v}}_1 \cdot \norm{\rho}_1 ) \norm{y}
        < (L_2 + \norm{\vec{v}}_1 \cdot \norm{\rho}_1 ) \delta \le \eps,
    \end{align*}
    where in the third inequality we use that $\one - D \varphi(x) = D(x - \varphi(x)) = D(\rho(x-y)\cdot y) = \nabla \rho(x-y) \cdot y^T$. This finishes the proof.
\end{proof}

We are now ready to prove \Cref{prop:general-position-open-dense}.

\begin{proof}[Proof of \Cref{prop:general-position-open-dense}]
    Let $v \in \X_{gMS+}(M)$ and choose $\eps>0$ small enough, such that for any $a,b,c,d \in \Sing(v)$ with $a\neq b$, $c\neq d$ and $\{a,b\} \neq \{c,d\}$ we have $\eps< |d_M(a,b)-d_M(c,d)|$. By \Cref{thm:str-stab} there exists a neighbourhood $\Nn$ of $v$ in $\X^1(M)$ such that for all $w \in \Nn$, the singular points of $w$ are within $\eps$-distance from the singular points of $v$. By our choice of $\eps$ this implies that $\Nn \subseteq \X_{gMS+}(M)$, so we have shown that $\X_{gMS+}(M)$ is open.

    It remains to show that $\X_{gMS+}(M)$ is dense in $\X^1_{gMS}(M)$. For this, let $v \in \X^1_{gMS}(M)$ and let $\Nn \subseteq \X^1_{gMS}(M)$ be an open neighbourhood of $v$. Denote by $p_1,\ldots,p_n \in M$ the singular points of $v$. We want to construct $w \in \Nn$ with singular points $q_1,\ldots,q_n$ which are in general position. 

    We cover $M$ by $V_1,\ldots,V_k$ ($k\ge n$), contained in charts $(U_i,\sigma_i)$, satisfying
    \begin{itemize}
        \item $\sigma_i\colon U_i \to B_4(0)\subseteq \R^m$ with $\sigma_i(V_i)=B_2(0)$,
        \item $W_i \cap U_j = \emptyset$ for $i\neq j$ with $W_i:= \sigma_i^{-1}(B_1(0))$,
        \item $p_i \in V_i$ and $\sigma_i(p_i)=0$ for $i=1,\ldots,n$,
        \item $\norm{\sigma_i(q)}=d_M(p_i,q)$ for all $i=1,\ldots,n$ and for all $q \in W_i$.
    \end{itemize}
    The last condition can be satisfied by choosing the $\sigma_i$ to agree with the exponential map at $p_i$ for nearby points. The topology induced on $\X(M)$ by the norm $\norm{\argdot}_1$ yields $\X^1(M)$. Now since $\Nn$ is a neighbourhood of $v$ in $\X^1(M)$, there exists $\eps>0$ such that $\{w \in \X^1(M) \mid \norm{v-w}_1 < \eps \} \subseteq \Nn$. Thus it suffices to show that there exists $w \in \X^1_{gMS+}(M)$ with $\norm{v-w}_1 < \eps$. We want to show the following claim.\vspace{1em}

    \underline{Claim.} There exists $\delta>0$ such that, for arbitrarily chosen $q_1,\ldots,q_n \in M$ with $d_M(p_i,q_i)<\delta$ for all $i$, there exists $w \in \X^1(M)$ with $\norm{v-w}_1 < \eps$ and $\Sing(w)=\{q_1,\ldots,q_n\}$. \vspace{1em}

    Assuming this claim, the theorem follows since by \Cref{lem:gen-pos-dense-1} we can choose the points $q_i \in M$ in such a way that $w \in \X^1_{gMS+}(M)$. It thus remains only to show the claim. For this we apply \Cref{lem:gen-pos-dense-2} to each of the local vector fields $\vec{v}_i\colon B_4(0) \to \R^m$ and obtain $\delta_i>0$ such that for $y \in B_{\delta_i}(\sigma_i(p_i))$ there exists $\varphi_i\colon B_4(0)\to B_4(0)$ with the properties listed in \Cref{lem:gen-pos-dense-2}. Define $\delta:=\min_i \delta_i$. 
    Thus for $q_i \in M$ with $d_M(p_i,q_i) < \delta$, there exists $\varphi_i \colon B_4(0) \to B_4(0)$ such that $\varphi_i$ is the identity outside of $B_1(0)$, $\varphi_i(\sigma_i(q_i))=0$, and $\norm{\vec{v}_i - \vec{v}_i \circ \varphi_i}_1 < \eps$. 
    
    We now define $w \in \X(M)$ as follows. We define $w_i \in \X(U_i)$ by
    \begin{center}
        \begin{tikzcd}
            U_i \ar[rrrrr,bend right=20,"=:w_i"] \ar[r,"\Delta"] &U_i \times U_i \ar[r,"\id \times \sigma_i"] &U_i \times B_4(0) \ar[r,"\id \times \varphi_i"] &U_i \times B_4(0) \ar[r,"\id \times \vec{v}_i"] &U_i \times \R^m \ar[r,"\cong"] &TU_i.
        \end{tikzcd}
    \end{center}
    Note that $w_i=v$ on $U_i \setminus W_i$. Thus, for $i\neq j$, since $U_i \cap W_j=\emptyset$ and $W_i \cap U_j=\emptyset$, we have $w_i=w_j=v$ on $U_i \cap U_j$, thus we can glue together the $w_i$ to obtain $w \in \X(M)$ with $w|_{U_i}=w_i$ for all $i$. 

    We first check that $\Sing(w)=\{q_i\}$. Outside of the sets $W_i$, $i=1,\ldots,n$, we have $w=v\neq 0$. Inside $W_i$, $v$ has exactly one zero, namely $p_i$. For any $q \in W_i$ we thus have
    \[
    w(q)=0 \iff w_i(q)=0 \iff \vec{v}_i(\varphi_i(\sigma_i(q)))=0 \iff \varphi_i(\sigma_i(q))=0 \iff \sigma_i(q) = \sigma_i(q_i) \iff q=q_i.
    \]
    Thus $\Sing(w)= \bigcup_i \Sing(w)\cap W_i = \bigcup_i \{q_i\}$. It remains to show that $\norm{v-w}_1 < \eps$. Indeed, we have
    \[
    \norm{v-w}_1 = \max_i \norm{\vec{v}_i - \vec{w}_i}_1 = \max_i \norm{\vec{v}_i - \vec{v}_i\circ \varphi_i}_1 < \eps. 
    \]
    This proves the claim and hence the proposition.
\end{proof}

\subsection{Assigning a factored chain complex to a vector field}

Now we are in the position to assign a factored chain complex to a gradient-like Morse-Smale vector field in general position $v \in \X_{gMS+}(M)$. The idea is to take the Morse complex $\MC_\bullet(v)$ and apply a sequence of simplifications, each determined by a pair $(a,b)$ of singular points of index $k$ and $k-1$, for some $k=1,\ldots,n$. This yields a factored chain complex with constant homology whose tagged barcode consists of finite bars of the form $[0,t,t)$ and infinite bars $[0,0,\infty)$. Each finite bar corresponds to a pair of singular points and the value $t$ is related to their distance on $M$. The infinite bars correspond to singular points which are not paired up, and they can be viewed as generators for the homology of $M$.

More precisely, we endow $\MC_\bullet(v)$ with some additional information in order to apply \Cref{def:assign-parchain-to-vectorfield} or \Cref{def:assign-parchain-to-vectorfield2}.
We view the Morse complex $\MC_\bullet(v)$ as a based chain complex, where the basis in each degree is given by the singular points of that index. If we are moreover given a distance $d_M$ on $M$ (for example the shortest path distance induced from a Riemannian metric), then we view $\MC_\bullet(v)$ as a weighted based chain complex, with the weight of a pair of basis elements given by the distance of the corresponding singular points on $M$. 
Note that in case $M$ is not connected, the shortest path distance is not defined between points of different connected components. In this case, we set it to be any fixed number greater or equal to the maximum distance between any two points from the same component. One can check that this yields a metric. 
We will slightly abuse notation and always write $\MC_\bullet(v)$ for the Morse complex of $v$, but consider it to be a chain complex, or a based chain complex, or a weighted based chain complex, depending on the context.
If $v$ is in general position, then $\MC_\bullet(v)$ is generic, thus in particular ordered, so we can apply \Cref{def:assign-parchain-to-vectorfield} or \Cref{def:assign-parchain-to-vectorfield2} and construct $X(v):= X(\MC_\bullet(v))\in \TEPCh$ or $Y(v):= Y(\MC_\bullet(v))\in \TEPCh$, respectively. Note that the basis elements that appear in \Cref{def:assign-parchain-to-vectorfield} or \Cref{def:assign-parchain-to-vectorfield2} always correspond to singular points of $v$, even after multiple iterations. We therefore get a sequence of pairs of singular points $(a_1,b_1),\ldots,(a_r,b_r)$, whose indices always differ by one. 
We say that the pairs $(a_1,b_1),\ldots,(a_r,b_r)$ are simplified in the construction of $X(v)$ or $Y(v)$.

The resulting tagged barcode, together with the pairs of singular points that were simplified in the construction, then gives a topological invariant of the vector field. It can be interpreted as a sequence of algebraic simplifications that one can do in order to simplify the Morse complex of the vector field.

\begin{remark}
    One could also use other weights than the distance between the singular points. One may try to use the length of a separatrix or the maximal value of the norm of the vector field along a separatrix, instead. However, in these cases, one would have weights only between points that are connected by a separatrix. It is then not clear a priori how one would update the weights after a simplification of the Morse complex. For this reason, we are using the distance between the points. However, thanks to the generality of the definitions and results of \Cref{sec:parametrize-ch}, other choices for the weights are not in principle precluded.

    In the case where we are studying a Morse-Smale function $f\colon M \to \R$ via the Morse complex of its gradient, $\MC_\bullet(-\nabla f)$, there is another obvious choice for the weights, namely the distance of function values. In that case, the tagged barcode one obtains is closely related to the persistence barcode of $f$, see \Cref{sec:scalar-fields}. Since we are also allowing non-conservative vector fields, however, this approach does not work more generally.
\end{remark}

\begin{remark}[Invariance] Usually, when talking about (topological/geometric) invariants, one would like to describe a set of transformations of the input objects under which the outcome is invariant. In our case, if two vector fields $v,w \in \X_{gMS+}(M)$ are topologically equivalent and the topological equivalence can be realized by a homeomorphism $h\colon M \to M$ that restricts to an isometry of the singular points, then $X(v) \cong X(w)$ and $Y(v) \cong Y(w)$. 
\end{remark}

To the best of our knowledge, ours is the first method that produces a barcode of a vector field with this invariance property.

\begin{remark}[Non-triviality]
    Non-triviality refers to the fact that an invariant is able to distinguish input objects if they are \lq\lq different enough" in some sense. In our case, we can say that the constructions $X$ and $Y$ are sensitive enough to distinguish the vector fields $v$ and $w$ in either of the following situations:
    \begin{enumerate}[(i)]
        \item The two vector fields $v$ and $w$ are defined on  manifolds $M$ and $M'$ whose singular homology groups are not isomorphic.
        \item The two vector fields $v$ and $w$ are both defined on $M$, but there exists $1\le k\le \dim M$ such that $|\Sing_k(v)| \neq |\Sing_k(w)|$.
        \item The two vector fields $v$ and $w$ are both defined on $M$, and $|\Sing_k(v)| = |\Sing_k(w)|$ for all $k$, but the distance values between singular points of $v$ and those of $w$ are all distinct.
    \end{enumerate}
\end{remark}

\begin{figure}[t]
    \centering
    \includegraphics[height=5cm]{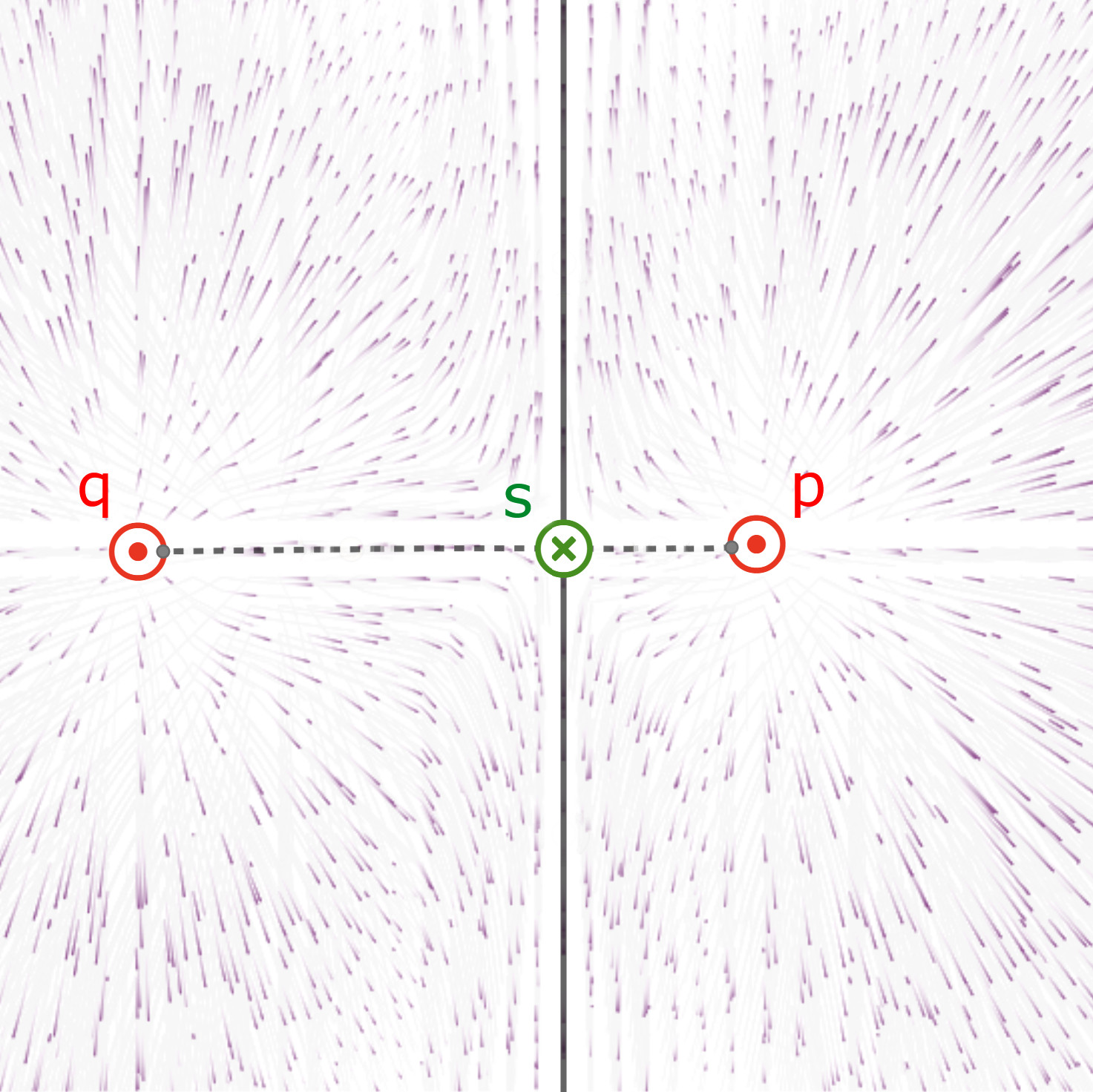}
    \hspace{1cm}
    \includegraphics[height=5cm]{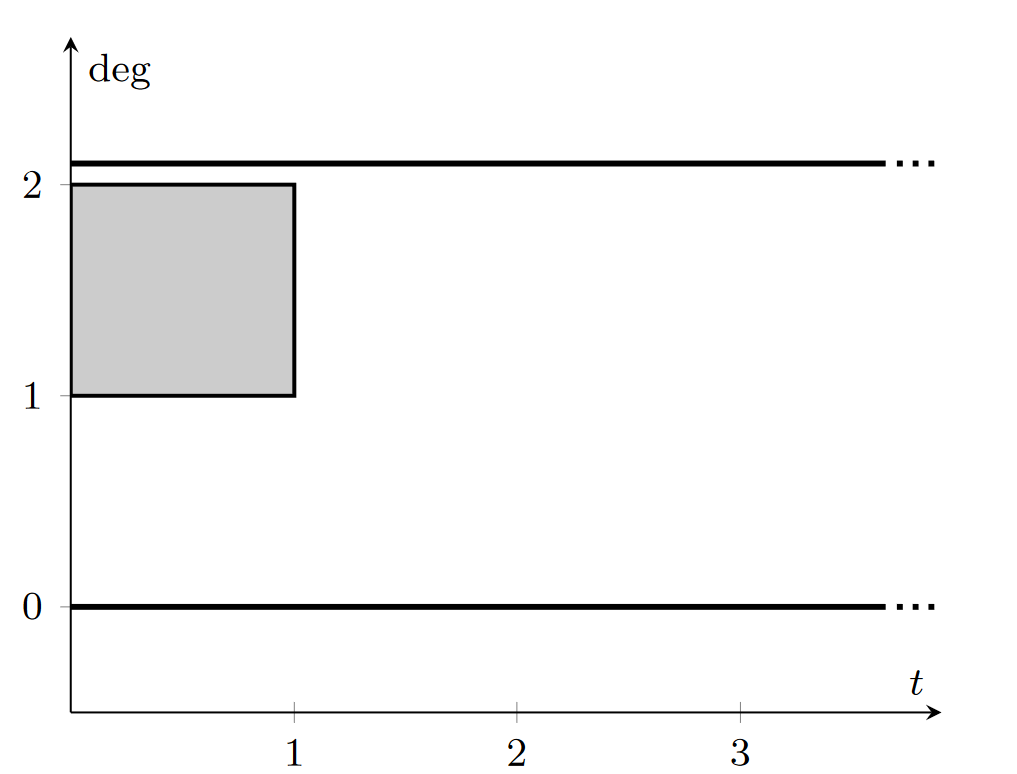}
    \caption{Left: Gradient-like Morse-Smale vector field $v$ on the 2-sphere with two singular points $p,q$ of index 2, one singular point $s$ of index 1, and one singular point $x$ of index 0. The point $x$ is represented in the image by the whole boundary.
    Right: Visualization of the resulting tagged barcode $\tBarc(Y(v))$, as explained in \Cref{ex:tagged-barcode}.}
    \label{fig:vector-field-barcode}
\end{figure}
 
\begin{example}\label{ex:tagged-barcode}
    In \Cref{fig:vector-field-barcode} we see a gradient-like Morse-Smale vector field $v \in \X_{gMS+}(S^2)$. It is defined on the 2-sphere, which is represented as a square, where all of the boundary is contracted to a point $x$, which is a singular point of index 0 for $v$. We denote the remaining singular points by $p,q,s$. Let us assume that $S^2$ is equipped with a Riemannian metric such that $d(p,s)=1$, $d(q,s)=2$, and also all other distances are pairwise different and larger than $1$, so that $v$ is in general position. The Morse complex of $v$ is 
    \begin{center}
        \begin{tikzcd}
            \Fi^2= \Span(p,q) \ar[rr,"{[1 , 1]}"]
            &&\Fi=\Span(s) \ar[rr,"{[0]}"]
            &&\Fi=\Span(x). 
        \end{tikzcd}
    \end{center}
    The weights are given by the distances, so for example $\operatorname{w}(p,s)=1$ and $\operatorname{w}(q,s)=2$. The only pairs that we could possibly simplify are $(p,s)$ and $(q,s)$, so since $\operatorname{w}(p,s)<\operatorname{w}(q,s)$ we will simplify $(p,s)$ at time $1$, giving rise to the tagged interval $[0,1,1)$ in degree 2 the tagged barcode of $X(v)$. After this simplification, all differentials are zero, so $q$ and $x$ cannot be simplified, giving rise to the tagged intervals $[0,0,\infty)$ in the degrees 2 and 0 respectively. On the right of \Cref{fig:vector-field-barcode} we visualize the tagged barcode in the following way: Finite bars $[0,t,t)$ in degree $n$ are showed as a gray block spanning from degree $n-1$ to degree $n$ on the vertical axis and from $0$ to $t$ on the horizontal axis. Infinite bars $[0,0,\infty)$ in degree $n$ are drawn as a black horizontal line on the height $n$.
\end{example}

\begin{figure}[ht]
    \centering
    \includegraphics[width=0.6\linewidth]{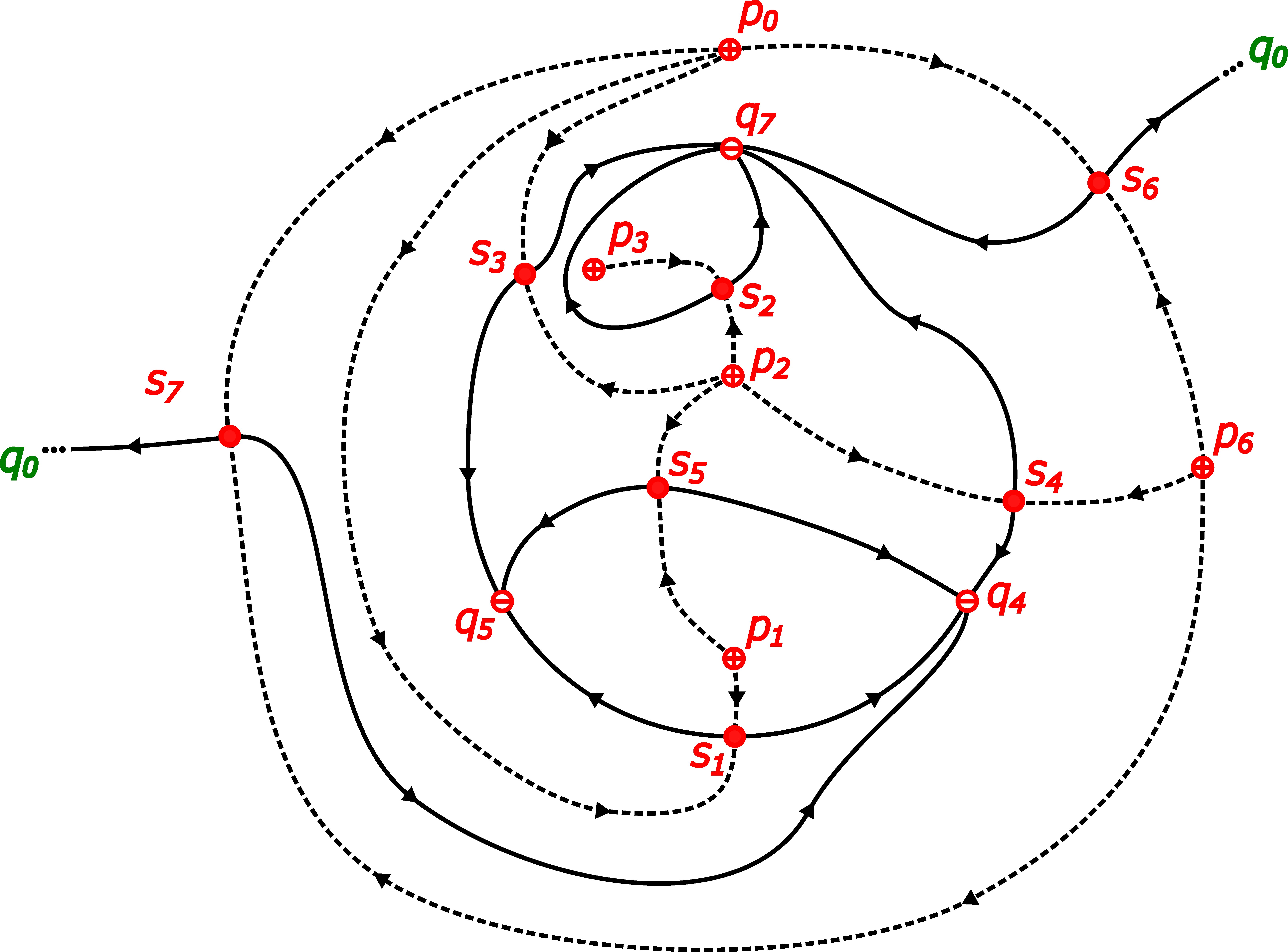}
    \resizebox{0.38\linewidth}{!}{
\begin{tikzpicture}[scale=0.45]
    \draw[step=1cm,gray!30] (-0.8,-0.8) grid (10.8,10.8); 
    \draw[step=5cm,gray!70,very thick] (-0.8,-0.8) grid (10.8,10.8); 

    \draw[thick,->] (-0.8,0) -- (10.8,0) node[above] {$s$}; 
    \draw[thick,->] (0,-0.8) -- (0,10.8) node[above] {$t$}; 

    \foreach \s in {1,2,...,10} {
        \draw[thick] (\s,0.1) -- (\s,-0.1); 
    }
    \foreach \s in {1,5} {
        \node at (\s+0.3,-0.4) {\s}; 
    }

    \foreach \t in {1,2,...,9} {
        \draw[thick] (0.1,\t) -- (-0.1,\t); 
    }
    \foreach \t in {1,5} {
        \node at (-0.4,\t+0.3) {\t}; 
    }

    \node at (-0.7,10.3) {$\infty$};

    \draw[thick] (0,0) -- (10.8,10.8); 

        \node at (0,10) {\color{red}\scalebox{1}{$\ast$}};
    
    \foreach \y in {1.5,2.6,7.9} {
        \node at (\y,\y) {\color{green}\scalebox{0.5}{$\blacksquare$}}; 
    }

    \foreach \x in {1.1,1.2,0.9,4.2} {
        \node at (\x,\x) {\color{blue}\scalebox{1}{$\circ$}}; 
    }

        \node at (0,10) {\color{blue}\scalebox{1.2}{$\circ$}};

    \draw[fill=white,draw=black,thick,rounded corners=1mm] (5.3,1.2) rectangle (10.5,4.8); 
    \node[align=center] at (7.9,3) { 
        {\color{red}\scalebox{1}{$\ast$}} $\I^0[0,s,t)$ \\[2pt] 
        {\color{green}\scalebox{0.5}{$\blacksquare$}} $\I^1[0,s,t)$ \\[2pt] 
        {\color{blue}\scalebox{1}{$\circ$}} $\I^2[0,s,t)$   
    };
\end{tikzpicture}}
    \caption{Left: Gradient-like Morse-Smale vector field on the 2-sphere with five sources, seven saddles, and four sinks. Right: The corresponding tagged barcode is shown in a persistence diagram like fashion as explained in Example \ref{ex:tagged-barcode-complicated}.}
    \label{fig:vector-field-barcode-complicated}
\end{figure}

\begin{example}\label{ex:tagged-barcode-complicated}
    In \Cref{fig:vector-field-barcode-complicated} we see a more complicated gradient-like Morse-Smale vector field $v \in \X_{gMS+}(S^2)$ that calls for a different visualization of its tagged barcode. Moreover, in this example we get also pairings of critical points that are not connected by a flow line and also some pairs get simplified later than pairs having a larger distance. The vector field $v$ has sources $p_i$, saddles $s_i$, and sinks $q_i$, with the indices chosen in such a way that, after applying \Cref{alg:tagged-barcode-computation} to the Morse complex $\MC_\bullet(v)$, the simplified pairs are $(p_1,s_1), (p_2,s_2), (p_3,s_3), (s_4,q_4), (s_5,q_5), (p_6,s_6), (s_7,q_7)$, in this order. The points $p_0$ and $q_0$ do not get paired. The vector field is represented in such a way that all the fixed points lie on the north hemisphere except for $q_0$ which is located at the south pole. The distances in the image reflect the distances between the fixed points on the manifold. The weights of the relevant pairs, given by their distances, are $\operatorname{w}(p_1,s_1)=1.1, \operatorname{w}(p_2,s_2)=1.2, \operatorname{w}(p_3,s_3)=0.9, \operatorname{w}(s_4,q_4)=1.5, \operatorname{w}(s_5,q_5)=2.6, \operatorname{w}(p_6,s_6)=4.2, \operatorname{w}(s_7,q_7)=7.9$. The Morse complex of $v$ is 
    \begin{center}
        \begin{tikzcd}[ampersand replacement=\&, column sep = 2.2cm]
            \Fi^5= \Span(p_0,p_1,p_2,p_3,p_6) \ar[r,"{\begin{bmatrix}
                1 &1 &0 &0 &0 \\
                0 &0 &1 &1 &0 \\
                1 &0 &1 &0 &0 \\
                0 &0 &1 &0 &1 \\
                0 &1 &1 &0 &0 \\
                1 &0 &0 &0 &1 \\
                1 &0 &0 &0 &1
            \end{bmatrix}}"]
            \&\Fi^7=\Span(s_1,\ldots,s_7) \ar[r,"{\begin{bmatrix}
                0 &0 &0 &0 &0 &1 &1 \\
                1 &0 &0 &1 &1 &0 &1 \\
                1 &0 &1 &0 &1 &0 &0 \\
                0 &0 &1 &1 &0 &1 &0 
            \end{bmatrix} \vspace{1.5cm}}"{yshift=0.4cm}]
            \&\Fi^4=\Span(q_0,q_4,q_5,q_7). 
        \end{tikzcd}
    \end{center}
    
    It follows that 
    \begin{align*}
        Y(v) =\hspace{1em} &\I^0[0,0,\infty) \oplus \\
        &\I^1[0,1.5,1.5) \oplus \I^1[0,2.6,2.6) \oplus \I^1[0,7.9,7.9) \oplus \\
        &\I^2[0,1.1,1.1) \oplus \I^2[0,1.2,1.2) \oplus \I^2[0,0.9,0.9) \oplus \I^2[0,0,\infty).
    \end{align*}
\end{example}

\subsection{Local stability}\label{sec:local-stab}

Recall that $\X^1_{gMS}(M)$ is the space of gradient-like Morse-Smale vector fields on $M$ in general position, endowed with the Whitney $C^1$ topology. We endow $\TEPCh$ with the topology induced by the interleaving distance. The goal of this section is to prove the following statement.

\begin{theorem}\label{thm:local-stability}
    The maps $X,Y \colon \X^1_{gMS+}(M) \to \TEPCh$, which assign to a gradient-like Morse-Smale vector field in general position $v$ the factored chain complexes $X(v)$ and $Y(v)$, built as in \Cref{def:assign-parchain-to-vectorfield} and \Cref{def:assign-parchain-to-vectorfield2}, respectively, are continuous.
\end{theorem}

\begin{proof}
    In order to show that the map $X\colon \X^1_{gMS+}(M) \to \TEPCh$ is continuous, it suffices to show that for every $v \in \X^1_{gMS+}(M)$ and every $\eps >0$ there exists a neighbourhood $\Nn$ of $v$ in $\X^1_{gMS+}(M)$ such that for all $w \in \Nn$, we have $d_I(X(v),X(w)) \le \eps$.
    Denote by 
    \[
    (a_1,b_1), \quad \ldots, \quad (a_n,b_n),
    \]
    the pairs of singular points that are simplified by applying the construction to $v$, in this order.
    Let $\eps'>0$ be a positive number that satsifies
    \[
    2n\eps' \le \eps, 
    \qquad
    4\eps' \le \xi(\Sing(v)),
    \qquad
    2\eps' \le d_M(p,q) \text{ for all } p\neq q \in \Sing(v).
    \]
    By \Cref{thm:str-stab}, there exists a neighbourhood $\Nn$ of $v$ in $\X_{gMS+}^1(M)$ such that for all $w \in \Nn$ there exists a topological equivalence $\varphi$ between $v$ and $w$ such that $d_M(p,\varphi(p))\le \eps'$ for all $p \in \Sing(v)$. By \Cref{lem:top-equiv-induces-iso}, $\varphi$ induces an isomorphism between the Morse complexes of $v$ and $w$. We now check that the conditions for \Cref{lem:stability-X-Y} are satisfied, i.e.~we want to show that for any $a,b,c,d \in \Sing(v)$, if $d_M(a,b)<d_M(c,d)$, then $d_M(\varphi(a),\varphi(b)) < d_M(\varphi(c),\varphi(d))$.

    To show this, note that, if $d_M(a,b)<d_M(c,d)$, then
    \begin{align*}
        d_M(\varphi(a),\varphi(b)) - d_M(\varphi(c),\varphi(d)) 
        &\le d_M(\varphi(a),a) + d_M(a,b) + d_M(b,\varphi(b)) \\
        &\hspace{10pt} + d_M(c,\varphi(c)) - d_M(c,d) + d_M(d,\varphi(d)) \\
        &< d_M(a,b) - d_M(c,d) + 4\eps' \\
        &\le d_M(a,b) - d_M(c,d) + \xi(\Sing(v))
        \le 0,
    \end{align*}
    from which it follows that $d_M(\varphi(a),\varphi(b)) < d_M(\varphi(c),\varphi(d))$. We have thus shown that the weights of pairs from $\Sing(v)$ are in the same order as the weights of the corresponding pairs from $\Sing(w)$, thus, by \Cref{lem:stability-X-Y}, in the construction of $X(w)$ the pairs
    \[
    (\varphi(a_1),\varphi(b_1)), \quad \ldots, \quad (\varphi(a_r),\varphi(b_r)),
    \]
    get simplified in this order and $d_I(\MC_\bullet(v),\MC_\bullet(w)) \le n d_\varphi(\MC_\bullet(v),\MC_\bullet(w))$. Note that for any singular points $a,b \in \Sing(v)$, we have
    \[ 
    |d(a,b)-d(\varphi(a),\varphi(b))| \le d(a,\varphi(a))+d(b,\varphi(b)) \le 2\eps',
    \]
    thus $d_\varphi(\MC_\bullet(v),\MC_\bullet(w)) \le 2\eps'$. Combining this yields the desired inequality.
    If we replace $X$ by $Y$, the same proof works and when choosing $\eps'$ suffices that $2 \eps' \le \eps$.
\end{proof}

\begin{figure}
     \centering
     \includegraphics[height=6cm]{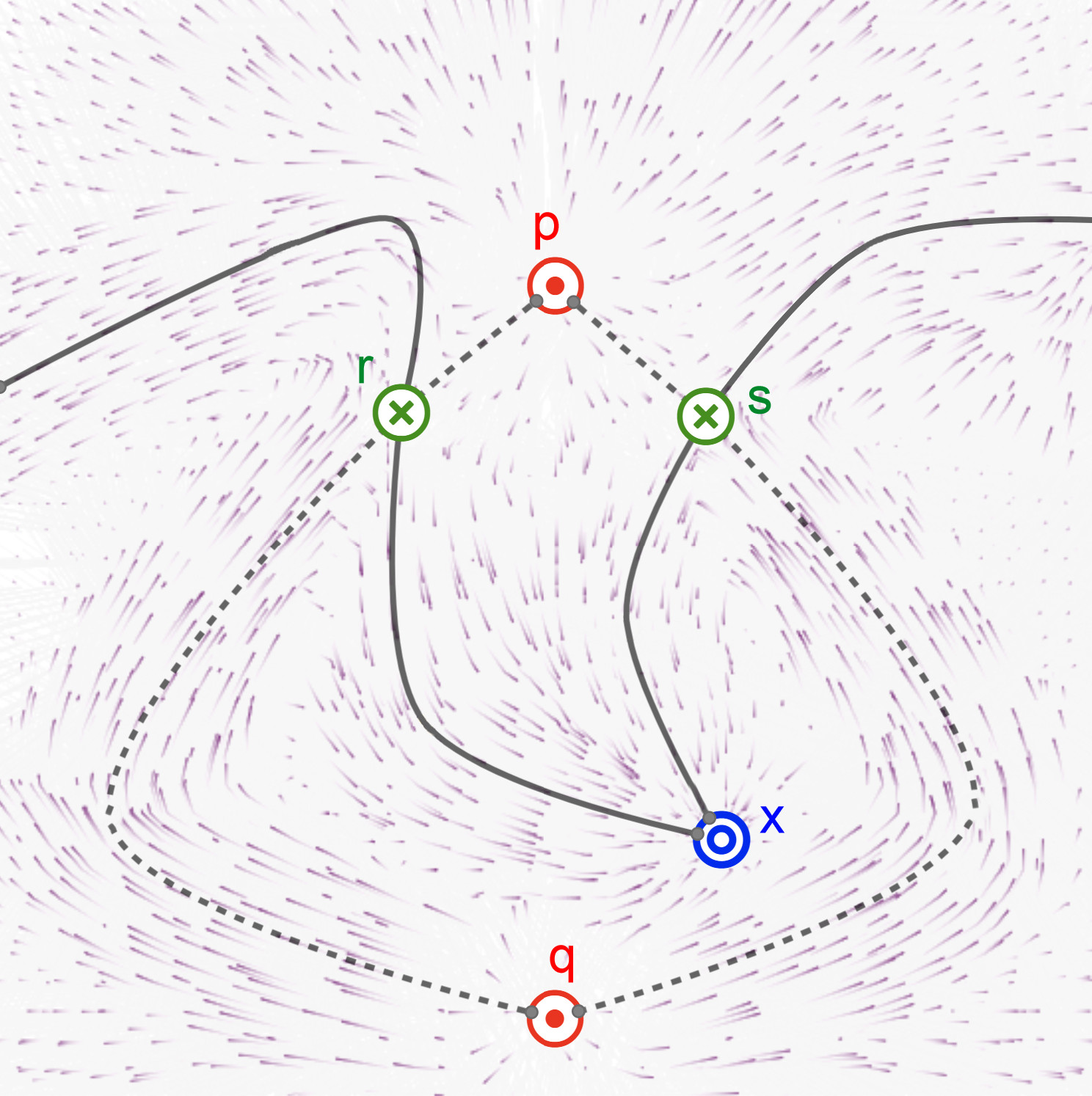}
     \hspace{2cm}
     \includegraphics[height=6cm]{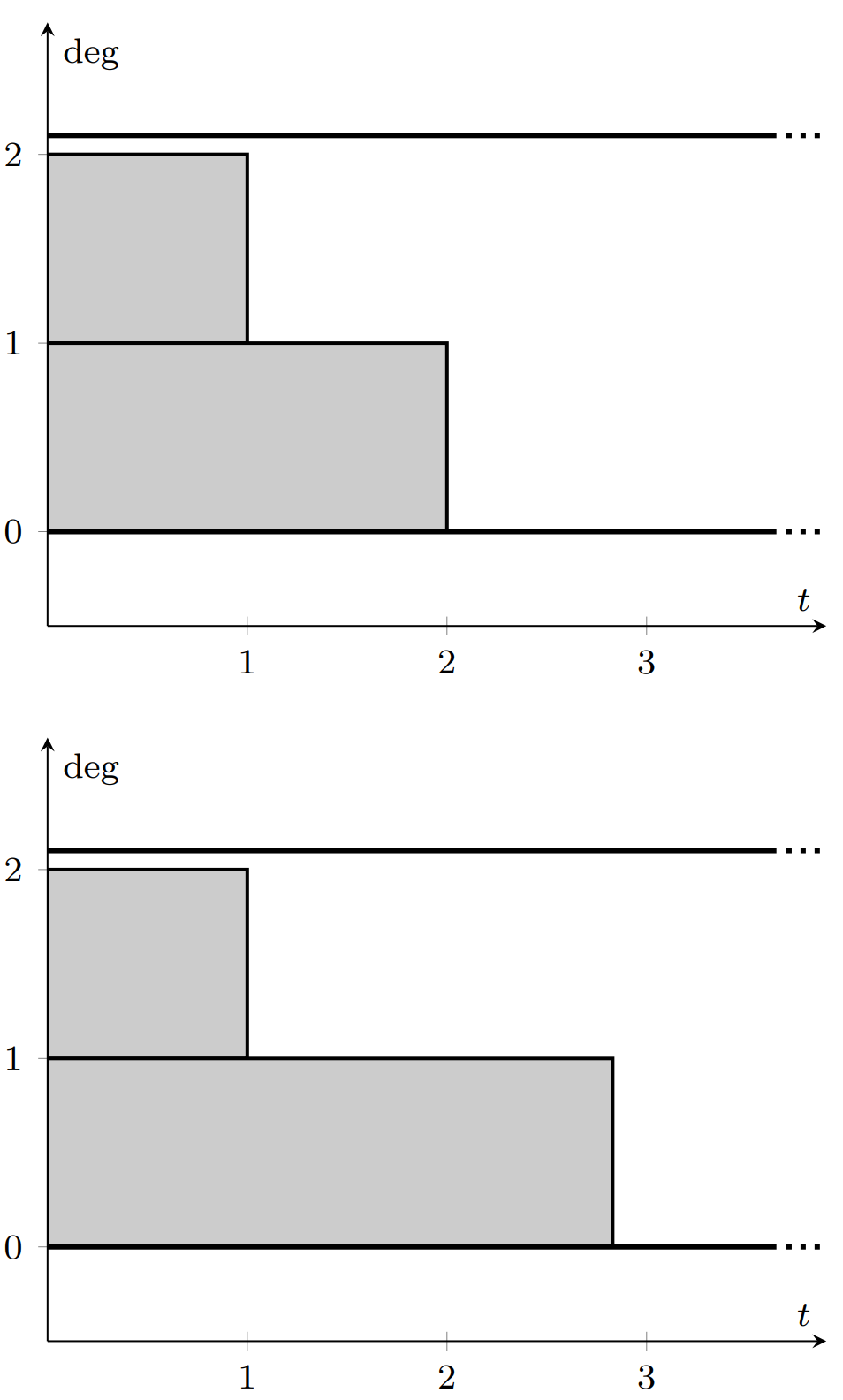}
    \caption{Different orderings of the pairs $(p,r)$ and $(p,s)$ yield different tagged barcodes for this vector field, as explained in \Cref{ex:different-barcodes}.}
    \label{fig:breaking-ties}
\end{figure}

\begin{example}\label{ex:different-barcodes}
    Consider the vector field $v$, defined on the 2-sphere, from \Cref{fig:breaking-ties}.  Again, the whole boundary of the image represents a single point $y$ of $S^2$, which is a sink for $v$. The other singular points of $v$ are two sources $p,q$, two saddle points $r,s$, and a sink $x$. 
    The relevant flow lines for the definition of the Morse complex are visible in the picture. Assume the 2-sphere is equipped with a Riemannian metric such that $d(p,r)=d(p,s)=1$, $d(s,x)=2$, $d(r,x)=\sqrt{8}$, whereas all other pairs of singular points have larger distances. 
    As weights are given by distances, the pairs $(p,r)$ and $(p,s)$ both have the same weight, so our constructions depend on the ordering of these two pairs. 
    If we simplify the pair $(p,r)$ first, then we next simplify the pair $(s,x)$, which has weight $2$. In contrast, if we simplify the pair $(p,s)$ first, then we next simplify the pair $(r,x)$, which has weight $\sqrt{8} \approx 2.8$. 
    This shows that if the weights are not unique, then it can happen that different orderings lead to different tagged barcodes. In \Cref{sec:scalar-fields} we will see that in the case of scalar fields, the situation is different, i.e. we can define a slightly different tagged barcode which is independent of the ordering of the pairs.
\end{example}

\subsection{Combinatorial approximations}\label{sec:comb-approx}

Given a smooth gradient-like Morse-Smale vector field $v$ on a smooth manifold $M$, one can approximate it by a sequence of triangulations of $M$ and combinatorial vector fields on these triangulations. 
First, \Cref{thm:Gallais} asserts that there exists a  triangulation $T$  of $M$ and a combinatorial vector field $V$ on it such that: 
(1) there is a bijection between the set of singular points $p$ of $v$ and the set of critical cells $\sigma_p$ of $V$ such that $p\in \sigma_p$ and $\ind_v(p)=\dim(\sigma_p)$, 
(2) for each pair of critical cells $\sigma_p$ and $\sigma_q$ such that $\dim(\sigma_p)=\dim(\sigma_q)+1$, $V$-paths from hyperfaces of $\sigma_p$ to $\sigma_q$ are in bijection with integral curves of $v$ up to renormalization connecting $q$ to $p$, 
(3) this bijection induces an isomorphism between the Morse complexes  $\MC_\bullet(v)$ and $\overline{\MC}_\bullet(V)$ associated with $v$ and $V$, respectively. 
Then, \Cref{thm:Zhukova} states that if we replace $T$ by its barycentric subdivision $\Delta(T)$, the combinatorial vector field $V$ on $T$ induces a combinatorial vector field $\Delta(V)$ on the barycentric subdivision $\Delta(T)$. By iterating this process we get a combinatorial approximation theorem.

As we did before with the smooth Morse complex, we want to view the combinatorial Morse complex as a weighted based chain complex. To do this, we assume that some extra structure is present on the simplicial complex.

\begin{definition}
    A \textbf{metric simplicial complex} is a simplicial complex equipped with a metric on its set of cells. A metric simplicial complex is called \textbf{generic}, if all the distances are pairwise different. 
\end{definition}

\begin{remark}
    Given a simplicial complex $K$, the combinatorial Morse complex can be viewed as an assignment $\overline{\X}_g(K) \to \operatorname{b}\Ch$, where the bases on $\overline{\MC}_\bullet(V)$ are given by the canonical bases of critical cells.

    If $K$ is moreover a metric simplicial complex, then we view the combinatorial Morse complex as a map $\overline{\X}_g(K) \to \wbCh$
    . The bases of $\overline{\MC}_\bullet(V)$ is given as before by the critical cells and the weights are taken to be the distances between those cells. Thus in that case we can assign a factored chain complex to the combinatorial Morse complex, according to either \Cref{def:assign-parchain-to-vectorfield} or \Cref{def:assign-parchain-to-vectorfield2}. We again shorten our notation by writing $X(V):= X(\overline{\MC}_\bullet(V))$ and $Y(V):= Y(\overline{\MC}_\bullet(V))$ for any $V \in \overline{\X}_g(K)$.
\end{remark}

If $K$ is a simplicial complex and $\sigma$ is a $k$-cell of $K$, then $|\sigma| \subseteq |K|$ has a well-defined \textbf{barycenter}. We denote this point by $b(\sigma) \subseteq |\sigma|$.

Using the barycenters, every triangulated Riemannian manifold induces a metric simplicial complex in the following way: If $(M,M',\phi)$ is a triangulated Riemannian manifold, then we have the simplicial complex $M'$ and for the distance on the cells of $M'$ we take the distance between the barycenters, i.e.~we define $d(\sigma,\tau) := d_M(\phi(b(\sigma)),\phi(b(\tau)))$.

Since there is a canonical homeomorphism between $|K|$ and $|\Delta(K)|$, we will identify these two spaces. Therefore, if $(M,M',\phi)$ is a triangulated manifold, then so is $(M,\Delta(M'),\phi)$.

\begin{definition}
    Let $M$ and $v$ be as in \Cref{thm:Gallais}. A \textbf{triangulation of $v$} is a triple $(M',\phi,V)$ such that 
    \begin{itemize}
        \item $(M',\phi)$ is a triangulation of $M$,
        \item $V \in \overline{\X}_g(M')$,
        \item $V$ satisfies the conditions $(\ref{item:gallais-1})$, $(\ref{item:gallais-2})$, $(\ref{item:gallais-3})$ from \Cref{thm:Gallais}.
    \end{itemize}
\end{definition}

\begin{lemma}\label{lem:subdiv-is-triang}
    Let $M$ and $v$ be as in \Cref{thm:Gallais} and let $(M',\phi,V)$ be a triangulation of $v$. Then it is possible to choose orderings of the vertices of the critical cells of $V$ in such a way that also $(\Delta(M'),\phi,\Delta(V))$ is a triangulation of $v$.
\end{lemma}

\begin{proof}
    Since $(M',\phi,V)$ is a triangulation of $v$, for every singular point $p$ of $v$ there exists a critical cell $\sigma$ of $V$ such that $p \in \sigma$. Since the $k$-cells of $\Delta(M')$ cover the $k$-cells of $M'$, there exists a cell $\sigma'$ in $\Delta(M')$, with $|\sigma'|\subseteq |\sigma|$, such that $p \in \sigma'$. We can choose this cell $\sigma'$ (which may not be unique) for the definition of $\Delta(V)$. Doing so for every singular point of $v$ implies, by \Cref{thm:Zhukova}, that $\Delta(V)$ again satisfies all the conditions from \Cref{thm:Gallais} and so $(\Delta(M'),\phi,\Delta(V))$ is a triangulation of $v$.
\end{proof}

Now we are ready to state and prove our combinatorial approximation theorem.

\begin{theorem}\label{thm:comb-approx}
    Let $M$ be an oriented Riemannian manifold and let $v \in \X_{gMS+}(M)$. Let $(M',\phi,V)$ be a triangulation of $v$. Then for every $\eps>0$ there exists $N \in \N$ such that for all $n\ge N$ we have
    \[
    d_I \left( X(v), X(\Delta^n(V)) \right) < \eps.
    \]
\end{theorem}

The statement of the theorem is also true when we replace $X$ with $Y$. The proof is analogous. Also, when writing $\Delta^n(V)$, we are hiding from the notation the dependence on some choices. The precise statement would be that there exists a sequence of $N$ choices, so that for any choice we do afterwards, the statement from the theorem holds.

\begin{proof}
    By iterated application of \Cref{lem:subdiv-is-triang}, $(\Delta^n(M'),\phi,\Delta^n(V))$ is a triangulation of $v$ for all $n$, given the right choices in the definition of $\Delta(V)$, $\Delta^2(V)$, etc. Since the cells in $\Delta^n(V)$ become smaller as $n$ gets larger, the barycenters of $\sigma^{(n)}_p$ converge to $p$ as $n$ goes to infinity, where by $\sigma^{(n)}_p$ we denote the barycenter of the critical cell in $\Delta^n(M')$ that contains $p$. If we choose $N$ large enough such that $\phi(|\sigma_p^{(N)}|) \subseteq B_{\eps'}(p)$, where $\eps'>0$ is a small enough positive number such that the conditions of \Cref{lem:stability-X-Y} are satisfied for $\MC_\bullet(V)$ and $\overline{\MC}_\bullet(\Delta^N(V))$. The conditions for such $\eps'$ are analogous to the conditions in the proof of \Cref{thm:local-stability}.
\end{proof}

\subsection{The case of scalar fields and connections to persistent homology}\label{sec:scalar-fields}

We now turn our attention to scalar fields, i.e. functions, on a Riemannian manifold. As it is often done, we restrict ourselves to the case of Morse-Smale functions. The gradient of such a function is a gradient-like Morse-Smale vector field. The goal of this section is twofold. First, we relate the classical persistence barcode of the function to the tagged barcode of its gradient field. Secondly, we show that, in contrast to what happens for gradient-like vector fields (see \Cref{ex:different-barcodes}), for gradient fields we do not need further genericity conditions.

Given a Morse-Smale function, we can endow the Morse complex of its gradient with a different set of weights. Instead of distances between singular points we can use the difference in function values. For convenience, we will assume that $f$ takes only non-negative values. Since we consider only closed manifolds, this can always be achieved by adding a constant. 

We now introduce a modification of \Cref{def:weighted-based-chain-complex} for scalar fields. This is analogous to  \cite{Edelsbrunner2023depth}.

\begin{definition}
    A \textbf{based chain complex with a filter} is a based chain complex $(C_\bullet,\B_\bullet)$ together with a \textbf{monotone} function $f\colon \bigcup_k \B_k \to [0,\infty)$, that is, for all $a \in \B_k$, $b \in \B_{k-1}$, if $\langle \partial a,b \rangle \neq 0$, then $f(b)<f(a)$.
\end{definition}

Therefore, if we are given a Morse-Smale function $f$, then we can view the Morse complex $\MC_\bullet(-\nabla f)$ as both a based chain complex with a filter, where the filter is given by the function values of $f$ on the singular points,  and a weighted based chain complex, where the weights are equal to $|f(a)-f(b)|$.
 
Recall that a filtered chain complex is a parametrized chain complex which is tame and where all the internal chain maps are monomorphisms. A  based chain complex with a filter $(C_\bullet,\B_\bullet,f)$ induces a filtered chain complex $F(C_\bullet,\B_\bullet,f)\colon [0,\infty) \to \Ch$, $t \mapsto C^{f\le t}_\bullet$, in the following way. The chain complex $C^{f\le t}_\bullet$ consists of the vector spaces $C^{f\le t}_k := \Span(\{b \in \B_k \mid f(b)\le t \}) \subseteq C_k$. The differential $\partial\colon C^{f\le t}_k \to C^{f\le t}_{k-1}$ is given by the restriction of the differential in $C_\bullet$. This is well-defined because the function $f$ is monotone.

Applying homology then yields persistence modules $H_n(F(C_\bullet,\B_\bullet,f))$, one for each $n$, as explained in \Cref{sec:par-vec-ind-from-par-chain}. These are parametrized vector spaces and hence have persistence barcodes due to \Cref{thm:str-thm-vect}.

On the other hand, given a filtered based chain complex $(C_\bullet,\B_\bullet,f)$, this induces a weighted based chain complex $(C_\bullet,\B_\bullet,\operatorname{w})$, where the weights are given by $\operatorname{w}(a,b):=|f(a)-f(b)|$. Because it has different weights than in the previous sections, the factored chain complexes resulting from applying either \Cref{def:assign-parchain-to-vectorfield} or \Cref{def:assign-parchain-to-vectorfield2} are hence also different. However, if we apply them in this setting, the result is related to the persistence barcode of $f$. More precisely, we will show that it is possible to obtain the tagged barcode of $Y(C_\bullet,\B_\bullet,\operatorname{w})$ by mapping finite intervals $[s,t)$ from $\Barc(H_n(F(C_\bullet,\B_\bullet,f)))$ to $[0,t-s,t-s)$ in $\tBarc_{n+1}(Y(C_\bullet,\B_\bullet,\operatorname{w}))$ and infinite intervals $[s,\infty)$ to $[0,0,\infty)$ in $\tBarc_n(Y(C_\bullet,\B_\bullet,\operatorname{w}))$. The proof strategy is by induction over the number of pairs that get simplified in our construction and show that if we simplify a pair $(a,b) \in \B_n\times \B_{n-1}$, then there is an interval present in $\Barc(H_{n-1}(F(C_\bullet,\B_\bullet,f)))$ of the form $[f(b),f(a))$. Moreover, the persistence barcode resulting from that simplification is equal to the persistence barcode of the original filtered chain complex with that interval removed.

Formally, we want to define the dashed map in the following diagram that makes it commute:

\begin{center}
\begin{tikzcd}
    (C_\bullet,\B_\bullet,f) \ar[r,mapsto,"\operatorname{w}(\argdot{,}\argdot):=|f(\argdot)-f(\argdot)|"] \ar[d,mapsto,"F"] &(C_\bullet,\B_\bullet,\operatorname{w}) \ar[d,mapsto,"Y"] \\
    F((C_\bullet,\B_\bullet,f)) \ar[d,mapsto,"\Barc"] &Y((C_\bullet,\B_\bullet,\operatorname{w})) \ar[d,mapsto,"\tBarc"] \\
    \displaystyle \Barc \left( \bigoplus_{n \in \N} H_n(F(C_\bullet,\B_\bullet,f)) \right) = \bigcup_{n\in \N} \Barc(H_n(F(C_\bullet,\B_\bullet,f)) \ar[r,mapsto,dashed] & \displaystyle \bigcup_{n\in \N} \tBarc_n(Y(C_\bullet,\B_\bullet,\operatorname{w}))
\end{tikzcd}
\end{center}

This will be done in \Cref{thm:persistence-and-tagged-barcode-comparison}. In order to state this result, we use \textbf{filtered interval spheres} \cite{CGJL}. This is the monomorphic analogue to the interval functors in $\TEPCh$.
    Recall that $\iota^n\colon S^n \hookrightarrow D^{n+1}$ denotes the inclusion of chain complexes. Given $n \in \N$, $0\le s < \infty$ and $s \le t \le \infty$, we define $\II^n[s,t) \in \TPCh$ by
    \[
    (\II^n[s,t))^r = \begin{cases}
        0, \text{ if } 0\le r <s,\\
        S^n, \text{ if } s \le r <t,\\
        D^{n+1}, \text{ if } t\le r,
    \end{cases}
    \quad \text{and} \quad 
    (\II^n[s,t))^{q\le r} = \begin{cases}
        0, \text{ if } q <s,\\
        \one_{S^n}, \text{ if } s\le q \le r <t,\\
        \iota^n, \text{ if } s\le q <t \le r,\\
        \one_{D^{n+1}}, \text{ if } t \le q.
    \end{cases}
    \] 

The following lemma is based on Proposition 3.2 in \cite{CGJL}. In the proof, we explain  the necessary adjustments to apply that result.

\begin{lemma}\label{lem:simplify-pair-ab}
    Let $(C_\bullet,\B_\bullet,f)$ be a  based chain complex with a filter. Let $(a,b) \in \B_n\times \B_{n-1}$, such that $\langle \partial a,b \rangle \neq 0$ and
    \begin{enumerate}[(i)]
        \item\label{item:split-cond-1} $f(b) = \max \{f(b') \mid \langle \partial a,b' \rangle \neq 0 \}$,
        \item\label{item:split-cond-2} $f(a) = \min \{f(a') \mid \langle \partial a',b \rangle \neq 0 \}$.
    \end{enumerate}
    Denote by $(\overline{C}_\bullet,\overline{\B}_\bullet)$ the based chain complex resulting from applying \Cref{lem:chain-complex-simplification} to the pair $(a,b)$. Define the filter $\overline{f}$ on $(\overline{C}_\bullet,\overline{\B}_\bullet)$ by $\overline{f}([c]):=f(c)$ for $c \in \B_i$. Then we have an isomorphism of filtered chain complexes
    \[
    F(C_\bullet,\B_\bullet,f) \cong F(\overline{C}_\bullet,\overline{\B}_\bullet,\overline{f}) \oplus \II^{n-1}[f(b),f(a)).
    \]
\end{lemma}

\begin{proof}
    Proposition 3.2 of \cite{CGJL} is formulated in terms of a filtered chain complex $X$ and a pair of generators $(x_i,x_j)$ satisfying the so-called split conditions. The conclusion states that $X \cong \II^{n-1}[s,e) \oplus X'$, where $n$ is the degree of $x_j$, $s$ is the entrance time of $x_j$, $e$ is the entrance time of $x_i$ and $X'$ is another filtered chain complex. Note that in \cite{CGJL}, the pair of generators $(x_i,x_j)$ is written in increasing degree, whereas we write pairs in decreasing degree. We apply the result to $X=F(C_\bullet,\B_\bullet,f)$ and the pair of generators $(b,a)$, noting that the split conditions from \cite{CGJL} are equivalent to our conditions $\langle \partial a,b \rangle \neq 0$ and $(\ref{item:split-cond-1}-\ref{item:split-cond-2})$.
    
    It remains to show that $\overline{f}$ indeed defines a monotone function on $(\overline{C}_\bullet,\overline{\B}_\bullet)$ and that the filtered chain complex $F(\overline{C}_\bullet,\overline{\B}_\bullet,\overline{f})$ agrees with $X'$ from \cite{CGJL}. To check the former, consider a pair $(c,d) \in \B_k \times \B_{k-1}$ with $\langle\partial^{\overline{C}} [c],[d]\rangle \neq 0$ with respect to the basis $\overline{\B}_{k-1}$. If also $\langle \partial c,d \rangle \neq 0$ with respect to $\B_{k-1}$, then we have $\overline{f}([c])-\overline{f}([d])=f(c)-f(d)>0$, so $\overline{f}$ satisfies the monotonicity condition for the pair $([c],[d])$. Note that for $k\neq n$, $\langle\partial^{\overline{C}} [c],[d]\rangle \neq 0$ implies that already $\langle \partial c,d \rangle \neq 0$ with respect to the basis $\B_{k-1}$ (in the case $k=n-1$ this uses the fact that we quotient out $\partial a$, which satisfies $\partial(\partial a)=0$). The only case remaining is thus $k=n$ and $d \notin \partial c$. This implies that $\langle \partial c,b \rangle \neq 0$ and $\langle\partial a,d\rangle \neq 0$ (see \Cref{lem:chain-complex-simplification} $(\ref{item:chain-complex-simplification-3})$). Thus, by $(\ref{item:split-cond-1})$ and $(\ref{item:split-cond-2})$, we have $f(c)\ge f(a)$ and $f(d)\le f(b)$, and thus
    \[
    \overline{f}([c]) - \overline{f}([d]) = f(c) - f(d) \ge f(a) - f(b) > 0,
    \]
    so $\overline{f}$ is monotone. The fact that $F(\overline{C}_\bullet,\overline{\B}_\bullet,\overline{f})$ is isomorphic to $X'$ from \cite{CGJL} can be checked by comparing the differential $\delta'$ from \cite{CGJL} with the induced differential from \Cref{lem:chain-complex-simplification}.
\end{proof}

Due to the additivity of homology, this implies that the persistence barcodes of $H_i(F(C_\bullet,\B_\bullet,f))$ can be obtained from those of $H_i(\overline{C}_\bullet,\overline{f})$ by adding one copy of $[f(b),f(a))$ to $\Barc(H_{n-1}(\overline{C}_\bullet,\overline{f}))$ in degree $n-1$.

\begin{theorem}\label{thm:persistence-and-tagged-barcode-comparison}
    Let $(C_\bullet,\B_\bullet,f)$ be a  based chain complex with filter and let $(C_\bullet,\B_\bullet,\operatorname{w})$ be the corresponding weighted based chain complex, with the weights given by $\operatorname{w}(a,b)=|f(a)-f(b)|$. In case of multiple pairs having the same weight, we choose an arbitrary order. Then there is a bijection of multisets
    \[
    \bigcup_{n\in \N} \Barc(H_n(F(C_\bullet,\B_\bullet,f)) \longrightarrow \bigcup_{n\in \N} \tBarc_n(Y(C_\bullet,\B_\bullet,\operatorname{w})),
    \]
    defined by
    \begin{align*}
        \Barc(H_n(F(C_\bullet,\B_\bullet,f))) \ni [s,t) &\longmapsto [0,t-s,t-s) \in \tBarc_{n+1}(Y(C_\bullet,\B_\bullet,\operatorname{w})), \\
        \Barc(H_n(F(C_\bullet,\B_\bullet,f)) \ni [s,\infty) &\longmapsto [0,0,\infty) \in \tBarc_n(Y(C_\bullet,\B_\bullet,\operatorname{w})).
    \end{align*}
\end{theorem}

In words, the theorem says that we can obtain the tagged barcode of $Y(C_\bullet,\B_\bullet,\operatorname{w})$ from the persistence barcode of $\bigoplus_nH_n(F(C_\bullet,\B_\bullet,f))$ by shifting all the bars to the left until they start at zero. The finite bars change degree by one, the infinite bars remain in the same degree.

\begin{proof}
    For brevity, we write $F(C_\bullet):=F(C_\bullet,\B_\bullet,f)$ and $Y(C_\bullet):=Y(C_\bullet,\B_\bullet,\operatorname{w})$ during this proof. Note that $Y(C_\bullet)$ a priori depends on the chosen ordering of the pairs of equal weight, since this influences which pairs get simplified in \Cref{def:assign-parchain-to-vectorfield2}. We do an induction over the number $m$ of finite tagged intervals in $\tBarc(Y(C_\bullet))$. 
    
    If $m=0$, that means that the differential in $C_\bullet$ is zero. In that case, for any $n\in \N$, the tagged barcode $\tBarc_n(Y(C_\bullet))$ consists of one copy of $[0,0,\infty)$ for each element of the basis $\B_n$. In the persistence barcode $\Barc(H_n(F(C_\bullet))$, on the other hand, each $b \in \B_n$ yields a bar of the form $[f(b),\infty)$. Thus the claim of the theorem holds true in this case.

    For the induction step, let $m\ge 1$ and assume that we have proven the theorem for $m-1$. Since $m\ge1$, the differential in $C_\bullet$ is non-zero, hence at least one pair gets simplified in the construction of $Y(C_\bullet)$. Let $(a,b)\in \B_n\times \B_{n-1}$ be the first pair that gets simplified. This means that $\operatorname{w}(a,b)$ is minimal among all pairs $(a',b')$ with $\langle \partial a',b' \rangle \neq 0$, which implies that the conditions from \Cref{lem:simplify-pair-ab} are satisfied for the pair $(a,b)$. Denote by $(\overline{C}_\bullet,\overline{\B}_\bullet)$ the based chain complex resulting from applying \Cref{lem:chain-complex-simplification} to the pair $(a,b)$. Note that this is equal to the based chain complex from \Cref{lem:simplify-pair-ab} as well as the one from \Cref{def:assign-parchain-to-vectorfield2}. Also the filter $\overline{f}$ from \Cref{lem:simplify-pair-ab} induces the same weights $\overline{\operatorname{w}}$ as the ones from \Cref{def:assign-parchain-to-vectorfield2}. Therefore, by the induction hypothesis, we have the above-described bijection of multisets
    \[
    \bigcup_{i\in \N} \Barc(H_i(F(\overline{C}_\bullet))) \longrightarrow \bigcup_{i\in \N} \tBarc_i(Y(\overline{C}_\bullet)).
    \]
    
    Now we want to extend the bijection to $F(C_\bullet)$ and $Y(C_\bullet)$. By construction of $Y(C_\bullet)$, it is clear that $\tBarc_i(Y(C_\bullet)) = \tBarc_i(Y(\overline{C}_\bullet))$ for $i\neq n$ and $\tBarc_n(Y(C_\bullet)) = \tBarc_n(Y(\overline{C}_\bullet)) \sqcup \{ [0,\operatorname{w}(a,b),\operatorname{w}(a,b))\}$. By \Cref{lem:simplify-pair-ab} and additivity of homolgy, we have 
    \[
    \Barc(H_i(F(C_\bullet)))= \begin{cases}
        \Barc(H_i(F(\overline{C}_\bullet))) &\text{if } i\neq n-1,\\
        \Barc(H_{n-1}(\overline{C}_\bullet)) \sqcup \{ [f(b),f(a))\} &\text{if } i=n-1.
    \end{cases}
    \]
    Thus, to extend the bijection of multisets, it is enough to map
    \[
    [f(b),f(a)) \mapsto [0,f(a)-f(b),f(a)-f(b)) = [0,\operatorname{w}(a,b),\operatorname{w}(a,b)),
    \]
    yielding a bijection of multisets between the persistence barcode of $\bigoplus_{i\in\N}H_i(F(C_\bullet))$ and the tagged barcode of $Y(C_\bullet)$. This completes the proof.
\end{proof}

\begin{corollary}
    Given $(C_\bullet,\B_\bullet,\operatorname{w}) \in \wbChgen$, where the weights are induced from a monotone function, then the tagged barcode of $X(C_\bullet)$ and $Y(C_\bullet)$ does not depend on the ordering of the weights.
\end{corollary}

\begin{proof}
    If the weights are induced from a monotone function $f$, then by \Cref{thm:persistence-and-tagged-barcode-comparison} we can compute the tagged barcode of $Y(C_\bullet)$ also from the persistence barcode of $\bigoplus_n H_n(F(C_\bullet,\B_\bullet,f))$. This barcode comes from the decomposition of a filtered vector space, which by \Cref{thm:str-thm-vect} is unique, and hence does not depend on the ordering of the pairs needed to compute $Y(C_\bullet)$. Thus, also the tagged barcode of $Y(C_\bullet)$ must be independent of the order. Since the tagged barcode of $X(C_\bullet)$ can be obtained from the barcode of $Y(C_\bullet)$ by adding up the lenghts of the intervals (compare \Cref{prop:barcode-construction1} for $X(C_\bullet)$ with \Cref{prop:barcode-construction2} for $Y(C_\bullet)$), this extends also to $X(C_\bullet)$.
\end{proof}

\paragraph{\textbf{Acknowledgments.}}
The visualization of the Morse-Smale vector field in \Cref{fig:vector-field-barcode,fig:breaking-ties} was created using the tool MVF Designer \cite{Zhou2019MVFDD}. The term {\em factored} when talking about parametrized objects with epimorphisms as internal maps is borrowed in analogy to the term factored representation as used by Raphael Bennet-Tennenhaus in one of his talks. The authors presented the ideas of this work at various workshops and are grateful for all the stimulating comments and questions received at these venues.
This work was carried out under the auspices of  INdAM-GNSAGA and within the activities of ARCES (University of Bologna). C.L. was partially supported by DISMI-UniMORE-FAR2023.

\bibliographystyle{amsplain}
\bibliography{Bibliography}

\end{document}